\documentclass[pdflatex,sn-basic,Numbered]{sn-jnl}
\usepackage{graphicx}
\usepackage{multirow}
\usepackage{amsmath,amssymb,amsfonts}
\usepackage{amsthm}
\usepackage{mathrsfs}
\usepackage[title]{appendix}
\usepackage{xcolor}
\usepackage{textcomp}
\usepackage{manyfoot}
\usepackage{booktabs}
\usepackage{algorithm}
\usepackage{algorithmicx}
\usepackage{algpseudocode}
\usepackage{listings}
\usepackage{subfig}
\usepackage{changepage}
\allowdisplaybreaks

\DeclareMathOperator*{\argmin}{arg\,min}

\theoremstyle{thmstyleone}
\newtheorem{theorem}{Theorem}
\newtheorem{Proposition}[theorem]{Proposition} 
\newtheorem {Cor}[theorem]{Corollary}
\newtheorem {Assumption}[theorem]{Assumption}
\newtheorem {Lemma}[theorem]{Lemma}
\theoremstyle{thmstyletwo}
\newtheorem{exam}{Example}
\newtheorem{remark}{Remark}

\theoremstyle{thmstylethree}
\newtheorem{definition}{Definition}

\makeatletter
\newenvironment{breakablealgorithm}
  {
   \begin{center}
     \refstepcounter{algorithm}
     \hrule height.8pt depth0pt \kern2pt
     \renewcommand{\caption}[2][\relax]{
       {\raggedright\textbf{\fname@algorithm~\thealgorithm} ##2\par}
       \ifx\relax##1\relax
         \addcontentsline{loa}{algorithm}{\protect\numberline{\thealgorithm}##2}
       \else
         \addcontentsline{loa}{algorithm}{\protect\numberline{\thealgorithm}##1}
       \fi
       \kern2pt\hrule\kern2pt
     }
  }{
     \kern2pt\hrule\relax
   \end{center}
  }
\makeatother

\begin{document}

\title[Article Title]{Mean Field Games for Renewable Energy Development}

\author[1]{\fnm{Luciano} \sur{Campi}}\email{luciano.campi@unimi.it}
\author*[1]{\fnm{Zhuoshu} \sur{Wu}}\email{zhuoshuwu@hotmail.com}
\affil[1]{\orgdiv{Dipartimento di Matematica ``Federigo Enriques''}, \orgname{Universit\`a degli Studi di Milano}, \orgaddress{\street{Via Saldini 50}, \city{Milan}, \postcode{20123}, \country{Italy}}}

\abstract{We propose a mean field game (MFG) framework to model the evolution of renewable energy production in competitive electricity markets. Producers interact through the spot price while optimising their profits under production, installation, and capacity adjustment costs, as well as the generation uncertainty. We first formulate the market as an $N$-player stochastic differential game and analyse its mean field game limit as $N\to\infty$. We characterise the representative producer's optimal control via forward-backward stochastic differential equations (FBSDEs) derived from the stochastic maximum principle and determine the corresponding equilibrium spot price. We establish existence and uniqueness of solutions to the FBSDEs and prove that the MFG admits a unique equilibrium, which yields an $\epsilon$-Nash equilibrium of the $N$-player game. We then extend the model to a Stackelberg MFG to incorporate the role of a social planner. The planner's optimisation problem leads to an extended Hamilton-Jacobi-Bellman (HJB) system, for which we prove existence and uniqueness of viscosity solutions. Finally, we implement a deep learning-based numerical scheme to approximate the equilibrium and investigate the impact of policy interventions on capacity dynamics. Our results highlight how optimal subsidy design depends on prevailing market conditions and can mitigate both capacity shortages and overproduction.}


\maketitle

\section{Introduction}\label{sec1}

Electricity generation remains a major contributor to global CO$_2$ emissions, despite a steady decline in emissions intensity driven by the rapid expansion of renewable and nuclear energy sources \footnote{\href{https://www.iea.org/reports/electricity-2025/emissions}{Clean energy mitigates $CO_2$ power emissions in 2025-2027}}. Increasing the share of electricity generated from renewable sources is therefore critical to decarbonising the energy system and achieving global climate targets.

The large-scale deployment of renewable energy technologies, particularly wind and solar photovoltaics, has become a cornerstone of sustainable electricity generation. While this shift brings substantial environmental and economic benefits, it also introduces new challenges. In particular, the intermittency of renewable generation and the possibility of overproduction, where generation systematically exceeds real-time system demand, can lead to price volatility, inefficient investment, and system instability. Addressing these challenges requires models that capture both strategic interactions among producers and the role of policy interventions.

In this paper, we focus on electricity markets and extend the MFG framework proposed in \cite{ABBC2023} to study the strategic interactions among infinitely many energy producers under uncertainty. Mean field games, introduced in \cite{LL2007} and \cite{MRP2006}, provide a tractable framework for analysing games with infinitely many players and have found numerous applications in economics and energy markets, including competition in the oil market \cite{GLL2010}, renewable energy certificate markets \cite{SFJ2022} and electricity market interactions \cite{CR2017}. Incorporating a social planner, who takes into account the energy supply and demand, naturally extends the MFG to a Stackelberg MFG, in which a principal (representing the regulatory authority) interacts with a mean field of energy producers, as illustrated in \cite{ACDL2022} for epidemic control problems.

We begin by formulating the renewable energy market as an $N$-player stochastic game in which each producer controls its generation capacity to maximise expected profit. The dynamics of generation capacity are subject to both idiosyncratic and common sources of uncertainty that capture geographical, weather, and technological shocks. Moreover, we assume no individual producer can influence the electricity spot price in a given direction through its own energy generation.

Instead of solving the $N$-player stochastic game directly, we adopt the MFG approach to approximate optimal behaviours of large populations of interacting agents, making it well suited to renewable energy markets. In the presence of common noise, the empirical distribution evolves stochastically; see \cite{CFS2013}. The probabilistic approach based on the stochastic maximum principle developed in \cite{CD2018} is applied to accommodate the common noise, and the representative producer's optimal control is characterised through a McKean-Vlasov FBSDE. We prove existence and uniqueness of solutions and develop the deep BSDE method proposed in \cite{HL2020} to approximate the mean field equilibrium.

To investigate the role of policy interventions, we extend the model to a Stackelberg MFG involving a social planner and a continuum of producers. The social planner designs installation subsidies (or taxes) to balance supply and demand while considering the cost of intervention, see \cite{SC1999}. This leads to a stochastic control problem for a forward-backward system.

Optimality conditions for the stochastic control problems governed by fully coupled FBSDEs have been studied extensively via the stochastic maximum principle in \cite{P1993}, \cite{BS2010}, \cite{WZ2013} and the references therein. As noted in \cite{BS2010}, these results typically rely on convexity assumptions that are not satisfied in our setting. We therefore adopt a dynamic programming approach and derive an extended HJB system characterising the planner's value function and backward state dynamics. In \cite{WY2008} and \cite{LW2014}, the dynamic programming principle (DPP) is applied to study this kind of problem with a recursive cost functional; more precisely, the cost functional is given in the form of backward stochastic differential equations while the state dynamics are described by the stochastic differential equations. For more general cost functionals, in \cite{XY2022}, an extended HJB system is derived to characterise the value function and BSDE, which is closely related to \cite{BC2010},  \cite{BM2010} and \cite{BMM2017}.

Building on these results, we analyse the social planner's optimisation problem by restricting attention to a class of admissible Markov controls, within which we demonstrate the existence of optimal controls, as established in \cite{PT1999},  the decoupling field of the associated BSDE admits a unique viscosity solution to a backward quasilinear second order parabolic partial differential equation (PDE). By \cite{P2009}, the DPP applies to the value function, yielding the associated HJB equation. We then prove existence and uniqueness of a viscosity solution to this HJB equation. Together with the PDE associated with the BSDE, this yields an extended HJB system. Finally, we prove the continuous differentiability of the value function and adopt the numerical approach proposed in \cite{DL2024} to solve the extended HJB system, which is defined in the sense of \cite[Definition 4.1]{BMM2017}. Numerical experiments are conducted to assess the impact of social planner interventions on renewable energy capacity development.

The main contributions of this work are threefold:
\begin{enumerate}
    \item We develop a stochastic MFG framework with common noise for renewable energy production, establish existence and uniqueness of equilibrium via McKean-Vlasov FBSDE, and show that the equilibrium provides an $\epsilon$-Nash equilibrium of the $N$-player game with $\epsilon = O\left( \frac{1}{\sqrt{N}} \right)$.
    \item We introduce a Stackelberg extension incorporating a social planner and derive the associated extended HJB system, proving existence and uniqueness of viscosity solutions.
    \item We propose a deep learning approach to numerically approximate solutions and conduct simulations to provide quantitative insights into the impact of policy interventions on capacity dynamics.
\end{enumerate}

The remainder of the paper is structured as follows. Section \ref{SDGs} formalises the $N$-player stochastic game and the optimisation problem of individual producers. Section \ref{MFGs} introduces the mean field limit and characterises the MFG solution via FBSDEs. Sections \ref{SMP} and \ref{MFGSNM} present theoretical existence and uniqueness results and the numerical implementation of FBSDEs. Section \ref{EPSINP} proves that the MFG solution provides a sensible approximation of the Nash equilibria for the $N$-player game. Section \ref{ASMFG} formulates the  Stackelberg MFG between the social planner and producers, while Section \ref{TEHJBS}  establishes the extended HJB system. Finally, Section  \ref{NAOEHJBS1} presents numerical results and discusses policy implications for renewable energy development.

\section{Model} \label{MFGM}
In this section, we describe a stochastic differential game with $N$ players, then formulate the limit problem as an MFG with common noise. For reasons that will become clear later, we do not analyse the $N$-player games directly. Instead, the setup in Section \ref{SDGs} serves as a motivation for the formulation of the MFG problem in Section \ref{MFGs}, where we show that solving the MFG problem yields an $\epsilon$-Nash equilibrium for the N-player game.

\subsection{$N$-player Stochastic Differential Games} \label{SDGs}

Given a finite horizon $T>0$, we consider a complete probability space $(\Omega, \mathcal{F}, \mathbb{P})$ supporting a sequence $(W^i)_{i \geq 0}$ of independent Brownian motions, where $W^0$ plays the r\^ole of common noise. We denote by $\mathbb F^{N} = (\mathcal F_t^N)_{t\in [0,T]}$ the filtration generated by $W^0, \ldots, W^N$, and by $(\mathcal{F}_t^0)_{t\in[0, T]}$ the filtration generated by $W^0$, both augmented by the $\mathbb{P}$-null sets. Throughout this paper, for any filtration $\mathbb G$, we use the spaces
\begin{align*}
\mathcal{L}_{\mathbb{G}}^2([0,T];\mathbb{R})
:= \Big\{ \beta:\Omega\times[0,T]\to\mathbb{R} \ \Big| \
&\beta \text{ is } \mathbb{G}\text{-progressively measurable},\\
&\mathbb{E}\Big[\sup_{t\in[0,T]} |\beta_t|^2\Big] < \infty \Big\},\\
\mathcal{H}_{\mathbb{G}}^2([0,T];\mathbb{R})
:= \Big\{ \beta:\Omega\times[0,T]\to\mathbb{R} \ \Big| \
&\beta \text{ is } \mathbb{G}\text{-progressively measurable},\\
&\mathbb{E}\Big[\int_0^T |\beta_t|^2 \, dt\Big] < \infty \Big\}.
\end{align*}
We assume there are $N$ producers and index them by $i$. We denote by $X^i=(X_t^i)_{t\in[0, T]}$ the energy generation capacity, $\alpha^i=(\alpha_t^i)_{t\in[0, T]}$ the capacity adjustment (positive for installation and negative for decommissioning) rate of producer $i$ at time $t\in[0, T]$. Each player $i\in\mathcal{N} :=\{ 1, \dots, N\}$ controls its state process $X_t^i$, given by
\begin{align}
dX_t^i = (-\delta X_t^i  + \alpha_t^i) dt + \sigma dW_t^i + \sigma^0 dW_t^0, \,\,\, X_0^i = \xi_0,
\end{align}
by selecting an admissible control $\alpha^i$ in $\mathcal{H}_{\mathbb{F}^N}^2([0, T]; \mathbb{R})$; where $\xi_0$ is a non-negative constant initial state of producer $i$ and $\delta\in(0, 1)$ is the decay rate of the generation capacity. The diffusive terms represent the stochastic nature of renewable resources. The idiosyncratic noise $\sigma dW_t^i$ captures local weather variability (e.g., cloud cover or local wind turbulence) and site-specific technical failures. The common noise $\sigma^0 dW_t^0$ represents regional climatic shocks and shifts in environmental policy that affect all renewable producers simultaneously; additionally, $\sigma, \sigma^0 \neq 0$.

Now, we formulate an $N$-player stochastic differential game, where producers simultaneously seek to achieve maximum profits in a competitive market while interacting with one another through the average generation capacities. Producers receive profit through the generation of electricity by means of renewable energy. 

Given the other producers' strategies, producer $i$ selects a control $\alpha^i\in\mathcal{H}_{\mathbb{F}^N}^2([0, T]; \mathbb{R})$ in order to maximise the expected profit given as follows
\begin{align}
J^i\left( \alpha^i, \alpha^{-i} \right) = \mathbb{E}\left[ \int_0^{T}  f(t, X_t^i, \bar{X}_t, \alpha_t^i) dt \right], \label{NGAIN}
\end{align}
where $\alpha^{-i}$ denotes the strategy profile of all players except $i$, $\bar{X}_t = \frac{1}{N} \sum_{i=1}^N X_t^i$ is the empirical mean of $(X_t^i)_{i\in\mathcal{N}}$ and $f$ is the running profit function defined by
\begin{align}
f(t, x, \bar{x}, \alpha) &=  x \left( P(t, \bar{x}) - c_p \right) - c_I \alpha - c_a \alpha^2.\label{RPFNP}
\end{align}
It is assumed that the function $f$ is identical for all producers. 

The objective function hence consists of three distinct terms. The first one represents the running profits associated with the generation capacity; more precisely, the producer makes profits $x \left( P(t, \bar{x})-c_p\right)$ per unit time for its generation capacity, where $P(t, \bar{x})$ is the spot price of one unit of electricity and $c_p>0$ is the fixed cost of producing one unit of electricity. The second term corresponds to the fixed cost of installing generation capacity with $c_I >0$, while the third term corresponds to the fixed adjustment cost $c_a>0$ that prevents producers from instantaneously and drastically changing their installation rate at zero cost. 

\begin{Assumption}\label{PILC}
The price function $P:[0, T]\times\mathbb{R}\mapsto [0, P_{cap}]$ is Lipschitz continuous in $\bar{x}$ with Lipschitz constant $L$ and uniformly in $t$. Moreover, the mapping $\bar{x} \mapsto P(t, \bar{x})$ is decreasing.
\end{Assumption}

\begin{remark}
The decreasing price function with respect to the empirical mean of $(X_t^i)_{i\in\mathcal{N}}$ follows from \cite[Page 698]{ABBC2023}, which argues that renewable energy drives electricity prices down as its market share increases, see also \cite[Page 2]{BP2019}. More specifically, this feature captures the cannibalisation effect. Because renewables have near-zero operating costs, they are always sold first in the market, which drives down the overall electricity price. As a result, when total renewable capacity grows, producers collectively lower the market price, reducing the revenue available for everyone. In addition, the time element of the price function reflects the seasonal effects, see \cite[Page 895]{SS2000}. 
\end{remark}

Note that the strategy of the other producers has an effect on the profit of producer $i$ through the empirical mean and that is the main feature that makes this setup a game. We are seeking an equilibrium concept widely used in game theory settings called the Nash equilibrium, whose definition is given as follows:
\begin{definition}\label{NE}
A set of strategies $(\hat{\alpha}^i)_{i\in\mathcal{N}}$ is said to be a Nash equilibrium if for every player $i\in\mathcal{N}$ and ${\alpha}^i\in\mathcal{H}_{\mathbb{F}^N}^2([0, T]; \mathbb{R})$,
\begin{align}
J^i(\hat{\alpha}^i, \hat{\alpha}^{-i}) = \sup_{\alpha^i\in\mathcal{H}_{\mathbb{F}^N}^2([0, T]; \mathbb{R})} J^i(\alpha^i, \hat{\alpha}^{-i}),
\end{align}
where $\hat \alpha^{-i}$ denotes the vector $(\hat \alpha^1,\ldots, \hat \alpha^{i-1},\hat \alpha^{i+1},\ldots, \hat \alpha^N)$. 
\end{definition}
This means that an individual producer cannot increase its expected profit by unilaterally deviating from the Nash strategy.

\subsection{Mean Field Game with Common Noise} \label{MFGs}
The stochastic game specified previously is intractable in general; hence, we focus our efforts on solving the N-player game in the limit as $N\to\infty$ following the MFG methodology. The MFG below is defined on the same probability space $(\Omega, \mathcal F,\mathbb P)$ as all $N$-player models of the previous section for $N \geq 1$. The relevant filtrations are $\mathbb F=(\mathcal F_t)_{t\in [0,T]}$, which is generated by $W^0,W^1$, and $\mathbb F^0$. Both are assumed to be completed with $\mathbb P$-null sets. In this section, we use $W$ for $W^1$ for notational simplicity.

Now we are ready to state the mean field game problem: (i) Given a stochastic process $\mu\in\mathcal{L}_{\mathbb{F}^0}^2([0, T]; \mathbb{R})$, solve the following stochastic control problem for a representative producer:
\begin{align}
\sup_{\alpha\in \mathcal{H}_{\mathbb{F}}^2([0, T]; \mathbb{R})} J(\alpha; \mu), \label{MFGP}
\end{align}
where $J(\alpha; \mu)=  \mathbb{E} \left[ \int_0^T  f(t, X_t, \mu_t, \alpha_t)dt \right]$; subject to the dynamical constraint
\begin{align}
dX_t = \left( -\delta X_t  + \alpha_t \right)  dt + {\sigma} dW_t + \sigma^0 dW_t^0, \,\,\, X_0=\xi_0, \label{SDESU}
\end{align}
with $\xi_0$ being a non-negative constant initial state. Note that the state dynamic is the same as that of an individual producer in the $N$-player game, with the player index suppressed. (ii) Determine a stochastic process $\mu^X$ such that $\mu_t^X = \mathbb{E}[ \hat{X}_t \mid \mathcal{F}_t^0]$ for all $t\in[0, T]$, where $\hat{X}$ is the optimal trajectory associated with the optimal control obtained from (i). 

This leads to the following definition.
\begin{definition}\label{MFE}
A strong mean field equilibrium (MFE) is an admissible control
$\hat{\alpha}\in\mathcal{H}_{\mathbb{F}}^2([0,T];\mathbb{R})$ such that
\begin{align}
J(\hat{\alpha};\mu^X)\geq J(\alpha;\mu^X),
\end{align}
for every admissible control $\alpha\in\mathcal{H}_{\mathbb{F}}^2([0,T];\mathbb{R})$, where the mean field $\mu^X$ satisfies
\begin{align}
\mu_t^X=\mathbb{E}[\hat{X}_t\mid\mathcal{F}_t^0],\qquad t\in[0,T],
\end{align}
with $\hat{X}$ denoting the state process associated with $\hat{\alpha}$.
\end{definition}

We shall formally define the map whose fixed points are obtained by means of forward-backward system of the conditional McKean-Vlasov type after the discussion of the Stochastic Maximum Principle. The MFE in Definition \ref{MFE} is used to construct approximate Nash equilibria for the corresponding games introduced in Definition \ref{NE} with finitely many players.

\subsubsection{Stochastic Maximum Principle} \label{SMP}

The stochastic maximum principle (SMP) is an approach to control problems that studies optimality conditions fulfilled by an optimal control, see \cite[Page 149]{P2009} for details. It gives sufficient and necessary conditions for the existence of an optimal control in terms of solvability of a BSDE as an adjoint process. In this section, we apply SMP to problem \eqref{MFGP} to obtain the unique optimal control.
We begin with the definition of the Hamiltonian of our model.

\begin{definition} \label{DOH}
The reduced Hamiltonian is defined as follows:
\begin{align}
H (t, x, \mu, y, \alpha) = (-\delta x+\alpha)y +  x \left( P(t, \mu) - c_p\right) - c_I\alpha - c_a \alpha^2,
\end{align}
for $t\in[0, T]$ and $x, y, \alpha, \mu\in\mathbb{R}$.
\end{definition}

\begin{Cor}\label{CCOH}
The mapping $(x, \alpha)\mapsto H (t, x, \mu, y, \alpha)$ is concave. 
\end{Cor}
\begin{proof}
The conclusion follows from the negative semi-definite Hessian of function $H$.
\end{proof}

Next, suppose $\mu=(\mu_t)_{t\in[0, T]}\in \mathcal{L}_{\mathbb{F}^0}^2([0, T]; \mathbb{R})$ is given, then we have a standard optimal stochastic control problem and via an appeal to the SMP, we arrive at the following result:
\begin{Proposition} \label{GANSXYZZ}
Suppose there exists an adapted solution $(\hat{X}, \hat{Y}, \hat{Z}, \hat{Z}^0)$ to the FBSDE:
\begin{align}
d\hat{X}_t &= \frac{\partial}{\partial y} H(t, \hat{X}_t, \mu_t,\hat{Y}_t, \hat{\alpha}_t) dt + \sigma dW_t + \sigma^0 dW_t^0,\nonumber\\
&= \left( -\delta \hat{X}_t +\frac{1}{2 c_a}( \hat{Y}_t-c_I) \right)  dt +  \sigma  dW_t + \sigma^0  dW_t^0, \,\,\, \hat{X}_0=\xi_0 \label{FSDE}\\
d\hat{Y}_t &= -\frac{\partial}{\partial x} H(t, \hat{X}_t, \mu_t,\hat{Y}_t, \hat{\alpha}_t) dt + \hat{Z}_t dW_t + \hat{Z}_t^0 dW_t^0, \nonumber\\
&= \left( \delta \hat{Y}_t +  c_p -  P(t, \mu_t) \right)dt + \hat{Z}_t dW_t + \hat{Z}_t^0 dW_t^0, \,\,\, \hat{Y}_T=0, \label{BSDE}
\end{align}
such that
\begin{align*}
\mathbb{E}\left[ \sup_{t\in[0, T]} (\hat{X}_t^2 + \hat{Y}_t^2) + \int_0^T \left( \hat{Z}_t^2 + (\hat{Z}_t^0)^2 \right) dt \right]<\infty,
\end{align*}
then, $\hat{\alpha_t} = \frac{1}{2c_a}(\hat{Y}_t-c_I)$ is the unique optimal control to problem \eqref{MFGP} given $(\mu_t)_{t\in[0, T]}$. 
\end{Proposition}

\begin{proof}
Immediate from \cite[Theorem 6.4.6, Page 150]{P2009}.
\end{proof}

Having stated Proposition \ref{GANSXYZZ}, the natural next step is to show that for any fixed $\mu$, the FBSDE \eqref{FSDE}-\eqref{BSDE} is uniquely solvable.

\begin{theorem}\label{MTIS2}
For any fixed $\mu\in \mathcal{L}_{\mathbb{F}^0}^2([0, T]; \mathbb{R})$, the forward-backward system \eqref{FSDE}-\eqref{BSDE} admits a unique solution $(\hat{X}, \hat{Y}, \hat{Z}, \hat{Z}^0)\in\mathcal{L}_{\mathbb{F}}^2([0, T]; \mathbb{R})\times\mathcal{L}_{\mathbb{F}}^2([0, T]; \mathbb{R}) \times\mathcal{H}_{\mathbb{F}}^2([0, T]; \mathbb{R}) \times \mathcal{H}_{\mathbb{F}}^2([0, T]; \mathbb{R}) $.  
\end{theorem}
\begin{proof}
Before proceeding to the well-posedness of the FBSDE, note that for any given $(\mu_t)_{t\in[0, T]}$, BSDE \eqref{BSDE} is not coupled with SDE \eqref{FSDE} and hence its well-posedness can be established separately. With \eqref{BSDE} of linear form, the conclusion follows easily from \cite[Theorem 4.3.1, Page 84]{ZJ2017}. Then, given that $\hat{Y}_t = \mathbb{E}\left[\int_t^T e^{-\delta(s-t)} \left( P(s, \mu_s) - c_p \right) ds |\mathcal{F}_t \right] $, an application of general theory of linear stochastic differential equations in \cite[Theorem 3.3.1, Page 68]{ZJ2017} yields the desired assertion.
\end{proof}

The following corollary states that for any given $\mu$, the stochastic control problem for a representative producer is uniquely solvable.
\begin{Cor} \label{SIOXYZ}
For any fixed $\mu\in\mathcal{L}_{\mathbb{F}^0}^2([0, T]; \mathbb{R})$, the stochastic control problem \eqref{MFGP} admits a unique optimal control given by $\hat{\alpha}_t = \frac{1}{2c_a}\left(\hat{Y}_t - c_I\right)$, where $(\hat{X}, \hat{Y}, \hat{Z}, \hat{Z}^0)\in\mathcal{L}_{\mathbb{F}}^2([0, T]; \mathbb{R}) \times \mathcal{L}_{\mathbb{F}}^2([0, T]; \mathbb{R})\times \mathcal{H}_{\mathbb{F}}^2([0, T]; \mathbb{R}) \times \mathcal{H}_{\mathbb{F}}^2([0, T]; \mathbb{R}) $ is the unique solution to the FBSDEs \eqref{FSDE}-\eqref{BSDE}.
\end{Cor}
 
We now present and verify our main result that the MFG stated in Section \ref{MFGs} is uniquely solvable; in particular, this corresponds to the matching problem (ii) (or the fixed point step) in Section \ref{MFGs}.

Let us first denote $\mu_t^X= \mathbb{E}\left[ X_t | \mathcal{F}_t^0 \right]$, $\mu_t^Y = \mathbb{E}\left[ Y_t | \mathcal{F}_t^0 \right]$. Now, taking conditional expectation given filtration set ${\mathcal{F}_t^0}_{t\in[0, T]}$ generated by $(W_t^0)_{t\in[0, T]}$ in FBSDE \eqref{FSDE}-\eqref{BSDE} and using the fact that, in equilibrium (that is, after solving for the fixed point), we have $\mu_t=\mu_t^X$ for all $t\in[0, T]$, we obtain:
\begin{align}
d\mu_t^X &=   \left( -\delta \mu_t^X + \frac{\mu_t^Y - c_I}{2c_a}  \right) dt + \sigma^0 dW_t^0,\,\,\,\,\,\,\,\,\,\, \mu_0^X = \xi_0,\label{MUFSDE}\\
d\mu_t^Y &= \left( \delta \mu_t^Y +  c_p - P(t, \mu_t^X) \right) dt + \hat{Z}_t^0 dW_t^0, \,\,\,\,\,\,\,\,\,\,\,\,\,\, \mu_T^Y=0,\label{MUBSDE}
\end{align}
where equation \eqref{MUBSDE}, that is, $(\mu^Y, \hat{Z}^0)$ being $\mathbb{F}^0$-progressively measurable, follows simply from the process $\mu^X$ being $\mathbb{F}^0$-progressively measurable, see \cite[Page 22]{BP2019}; and note that well-posedness of \eqref{MUFSDE}-\eqref{MUBSDE} implies the uniqueness and existence of the solution to the MFG, that is, the mapping $\mu_t^X= \mathbb{E}\left[ X_t | \mathcal{F}_t^0 \right]$ has a unique fixed point.

\begin{Proposition} \label{GWPOMFG}
The FBSDEs \eqref{MUFSDE}-\eqref{MUBSDE} admit a unique solution $(\mu^X, \mu^Y, \hat{Z}^0)\in \mathcal{L}_{\mathbb{F}^0}^2([0, T]; \mathbb{R})\times\mathcal{L}_{\mathbb{F}^0}^2([0, T]; \mathbb{R})\times\mathcal{H}_{\mathbb{F}^0}^2([0, T]; \mathbb{R})$ and it holds that for all $t\in[0, T]$, $\mu_t^Y = \psi(t, \mu_t^X)$, where function $\psi$ is continuous.
\end{Proposition}
\begin{proof}
The conclusion follows directly from \cite[Theorem 5.1 and Corollary 4.1]{PT1999}.
\end{proof}

\begin{theorem}
The MFG has a unique solution $(\mu^X, \hat{\alpha})$.
\end{theorem}
\begin{proof}
This follows immediately from Proposition \ref{GWPOMFG}.
\end{proof}

We conclude this section by commenting on the sensitivity of $\hat{\alpha}_t$ with respect to the parameters  of the problem: (i) the higher the costs of production and installation, the less generation capacity is developed; (ii) $\hat{\alpha}_t$ increases as the spot price increases, this is to be expected because the spot market offers better rewards; (iii) when the adjustment cost increases, there is less new capacity installed.

\subsubsection{The $\epsilon$-Nash Property} \label{EPSINP}

In the previous section, we investigated the MFG version, that is, the infinite-player limit of our original $N$-player game. Now, we are ready to demonstrate that the solution of the MFG is a sensible approximation of the original problem; in particular, we show that solving the mean field problem in Section \ref{MFGs}  corresponds to an $\epsilon$-Nash equilibrium for the finite-player problem stated in Section \ref{SDGs}. 

We start with the definition of the $\epsilon$-Nash equilibrium, which says that a strategy profile is an $ \epsilon$-Nash equilibrium if no individual producer can improve their objective by more than $\epsilon$ through a unilateral deviation, while the strategies of all other producers remain unchanged.
\begin{definition} \label{deop}
Given $\epsilon>0$, a strategy profile $(\hat{\alpha}^1, \dots, \hat{\alpha}^N) \in \mathcal{H}_{\mathbb{F}^N}^2([0, T]; \mathbb{R}^N)$ forms an $\epsilon$-Nash equilibrium for the collection of objective functions $(J^i)_{i\in\mathcal{N}}$, if
\begin{align}
J^i(\hat{\alpha}^i, \hat{\alpha}^{-i}) \leq \sup_{\alpha^i\in\mathcal{H}_{\mathbb{F}^N}^2([0, T]; \mathbb{R})} J^i(\alpha^i, \hat{\alpha}^{-i}) \leq J^i(\hat{\alpha}^i, \hat{\alpha}^{-i}) + \epsilon.
\end{align}
\end{definition}

For all $i\in\mathcal{N}$, we let
$\hat{\alpha}^i = \frac{1}{2c_a}(\hat{Y}_t^i-c_I)$, where $(\hat{X}^i, \hat{Y}^i, \hat{Z}^i, \hat{Z}^{0, i})$ are defined as that in Proposition \ref{GANSXYZZ}, that is,
\begin{align}
d\hat{X}_t^i &= \left( -\delta \hat{X}_t^i +\frac{1}{2 c_a}( \hat{Y}_t^i-c_I) \right)  dt +  \sigma  dW_t^i + \sigma^0  dW_t^0, \,\,\, \hat{X}_0^i=\xi_0 \label{FSDEepsilon}\\
d\hat{Y}_t^i &= \left( \delta \hat{Y}_t^i +  c_p -  P(t, \hat{\mu}_t^X) \right)dt + \hat{Z}_t^i dW_t^i + \hat{Z}_t^{i, 0} dW_t^0, \,\,\, \hat{Y}_T^i=0. \label{BSDEepsilon}
\end{align}
To prove the $\epsilon$-Nash property, we need to establish a useful property first.
\begin{Lemma}\label{ESEPSI}
Take an arbitrary control $\alpha^i\in\mathcal{H}_{\mathbb{F}^N}^2([0, T]; \mathbb{R})$. 
Then,
\begin{align}
|J^i(\alpha^i, \hat{\alpha}^{-i}) - J(\alpha^i; \mu^X)| \leq O\left( \frac{1}{\sqrt{N}} \right), \label{JAHATA}
\end{align}
where $J(\alpha^i; \mu^X) =  \mathbb{E}\left[ \int_0^T f(t, X_t^i, \mu_t^X, \alpha_t^i)dt \right]$ is defined in \eqref{MFGP}.
\end{Lemma}
\begin{proof}
We follow the proof from \cite[Page 39]{SFJ2022}. By the definition of $J^i$ in \eqref{NGAIN} and the Lipschitz property of the price function in Assumption \ref{PILC}, 
\begin{align}
|J^i(\alpha^i, \hat{\alpha}^{-i}) - J(\alpha^i; \mu^X)|&\leq L \mathbb{E}\left[\int_0^T |X_t^i(\bar{X}_t - \mu_t^X)|dt\right]\nonumber\\
&\leq L \left( \int_0^T \underbrace{\mathbb{E}\left[ \sup_{t\in[0, T]}|X_t^i|^2\right]}_{\mathbf{I}}dt \right)^{\frac{1}{2}}\left( \int_0^T \underbrace{\mathbb{E}\left[ |\bar{X}_t - \mu_t^X|^2\right]}_{\mathbf{II}} dt\right)^{\frac{1}{2}}.
\end{align}

(i) First, we estimate term $\mathbf{I}$ by observing that
\begin{align*}
\sup_{s\in[0, t]} |X_s^i|^2 \leq C_1 \left(1 + \int_0^t |X_r^i|^2 dr +  \int_0^t |\alpha_s^i|^2 ds + \sup_{s\in[0, t]}|W_s^i|^2 + \sup_{s\in[0, t]} |W_s^0|^2\right)
\end{align*}
and then by taking the expectation, we arrive at
\begin{align*}
\mathbb{E}\left[\sup_{s\in[0, T]} |X_s^i|^2\right] \leq C_2\left(1 +\mathbb{E}\left[ \int_0^T |X_t^i|^2 dt \right]+  \mathbb{E}\left[\int_0^T |\alpha_t^i|^2 dt \right]\right)
\end{align*}
where $C_1, C_2$ depend only on $T, \delta, \sigma, \sigma^0$ and $\xi_0$. By Gr\"{o}nwall's inequality, we have
\begin{align}
\mathbf{I}\leq C_3 \left(1+  \mathbb{E}\left[\int_0^T |\alpha_t^i|^2 dt \right]\right) \leq C_I, \label{ILIMIT}
\end{align}
where the second inequality follows as $\alpha^i$ is admissible and hence square integrable.

(ii) Next, we estimate term $\mathbf{II}$ by noting that
\begin{align}
\mathbb{E}\left[|\bar{X}_t - \mu_t^X|^2 \right]&= \mathbb{E}\left[ \bigg{|}\frac{1}{N}\sum_{\substack{j = 1\\ j\neq i}}^N \underbrace{\left( \hat{X}_t^j - \mathbb{E}\left[ \hat{X}_t^j |\mathcal{F}_t^0 \right] \right)}_{\gamma_t^j} + \frac{1}{N}(X_t^i - \mu_t^X)\bigg{|}^2\right]\nonumber\\
&\leq \frac{2}{N^2} \underbrace{\mathbb{E}\left[\bigg{|} \sum_{\substack{j = 1\\ j\neq i}}^N \gamma_t^j \bigg{|}^2\right]}_{\mathbf{II}_1} + \frac{2}{N^2}\underbrace{\mathbb{E}\left[\bigg{|}X_t^i - \mu_t^X\bigg{|}^2\right]}_{\mathbf{II}_2}.
\end{align}
It is clear that 
\begin{align}
\mathbf{II}_1 &= \sum_{\substack{j = 1\\ j\neq i}}^N \mathbb{E}\left[|\gamma_t^j|^2\right] + \sum_{\substack{j, k =1 \\ j, k\neq i\\ j\neq k}}^N \mathbb{E}\left[\gamma_t^j\gamma_t^k\right], \nonumber\\
&= (N-1) \mathbb{E}\left[ |\gamma_t^j|^2 \right] \label{II1}
\end{align}
where the second equality follows from the identical distribution of $\gamma_t^j$ for $j\in\mathcal{N}, j\neq i$ and the fact that $\hat{X}_t^j, \hat{X}_t^k$ are independent conditioning on $\mathcal{F}_t^0$, that is,
\begin{align}
\mathbb{E}\left[\gamma_t^j\gamma_t^k\right]=\mathbb{E}\left[ \mathbb{E}[\gamma_t^j \gamma_t^k |\mathcal{F}_t^0] \right] = \mathbb{E}\left[ \mathbb{E}\left[ \hat{X}_t^j \hat{X}_t^k |\mathcal{F}_t^0 \right]  -  \mathbb{E}\left[ \hat{X}_t^j |\mathcal{F}_t^0 \right] \mathbb{E}\left[ \hat{X}_t^k |\mathcal{F}_t^0 \right]  \right]=0.
\end{align}
Equation \eqref{II1}, together with Corollary \ref{SIOXYZ} and Proposition \ref{GWPOMFG} shows that $\mathbf{II}_1 \leq C_4 (N-1)$.
As for $\mathbf{II}_2$, by Corollary \ref{SIOXYZ} and Proposition \ref{GWPOMFG}, we see that
\begin{align*}
\mathbf{II}_2 &= \mathbb{E}\left[ |X_t^i - \hat{X}_t^i + \hat{X}_t^i - \mu_t^X| ^2\right] \leq 2 \mathbb{E}\left[ |X_t^i - \hat{X}_t^i | ^2 \right] + 2 \mathbb{E}\left[ | \hat{X}_t^i - \mu_t^X|^2\right] \\
&\leq C_5 (1 + \mathbb{E}\left[ |X_t^i - \hat{X}_t^i | ^2 \right]),
\end{align*}
where $X_t^i$ and $\hat{X}_t^i$ satisfy
\begin{align*}
dX_t^i &= (-\delta X_t^i + \alpha_t^i) dt + \sigma dW_t^i + \sigma^0 dW_t^0,\\
d\hat{X}_t^i &= (-\delta \hat{X}_t^i + \hat{\alpha}_t^i) dt + \sigma dW_t^i + \sigma^0 dW_t^0,
\end{align*}
such that by Gr\"{o}nwall's inequality and admissibility of $\alpha^i, \hat{\alpha}^i$, it follows that
\begin{align*}
 \mathbb{E}\left[ |X_t^i - \hat{X}_t^i | ^2 \right] &\leq C_6 \int_0^T \delta^2  \mathbb{E}\left[ |X_s^i - \hat{X}_s^i | ^2 \right] ds + C_7 \int_0^T  \mathbb{E}\left[ |\alpha_s^i - \hat{\alpha}_s^i|^2 \right]ds \\
 &\leq C_8 \int_0^T  \mathbb{E}\left[ |\alpha_s^i - \hat{\alpha}_s^i|^2 \right]ds \leq C_9,
\end{align*}
where $C_4, C_5, C_6, C_7, C_8$ depend only on $T, \delta, \sigma, \sigma^0$ and $\xi_0$. We therefore prove that
\begin{align}
\mathbf{II} \leq \frac{2}{N^2} \left( C_4 (N-1) + C_5 (1+ C_9) \right). \label{IILIMIT2}
\end{align}

(iii) By combining \eqref{ILIMIT} and \eqref{IILIMIT2}, inequality \eqref{JAHATA} follows.
\end{proof}

\begin{theorem}
For all $i\in\mathcal{N}$,
\begin{align}
J^i(\hat{\alpha}^i, \hat{\alpha}^{-i}) \leq \sup_{\alpha^i\in\mathcal{H}_{\mathbb{F}^N}^2([0, T]; \mathbb{R})} J^i({\alpha}^i, \hat{\alpha}^{-i}) \leq J^i(\hat{\alpha}^i, \hat{\alpha}^{-i}) + \epsilon. \label{EPSI}
\end{align}
In particular, $\epsilon = O\left( \frac{1}{\sqrt{N}} \right)$.
\end{theorem}
\begin{proof}
By the definition of supremum, the left-hand inequality in \eqref{EPSI} is immediate. It remains to show that the right-hand inequality is also satisfied.
Observe that for any arbitrary control $\alpha^i\in\mathcal{H}_{\mathbb{F}^N}^2([0, T]; \mathbb{R})$, Lemma \ref{ESEPSI} entails that for $\epsilon_N = O\left( \frac{1}{\sqrt{N}} \right)$,
\begin{align}
J^i(\alpha^i, \hat{\alpha}^{-i}) \leq \epsilon_N + J(\alpha^i; \mu^X) \leq  \epsilon_N + J(\hat{\alpha}^i;  \mu^X), \label{TEPSI1}
\end{align}
where the second inequality is due to the mean field optimality; after which, we change roles of $J^i(\hat{\alpha}^i; \mu^X)$ and $J(\hat{\alpha}^i, \hat{\alpha}^{-i})$,
\begin{align*}
-\epsilon_N\leq J(\alpha^i; \mu^X) - J^i(\alpha^i, \hat{\alpha}^{-i}) \leq \epsilon_N
\end{align*}
such that due to arbitrariness of $\alpha^i$,
\begin{align}
J(\hat{\alpha}^i; \mu^X) \leq  J^i(\hat{\alpha}^i, \hat{\alpha}^{-i}) + \epsilon_N, \label{TEPSI2}
\end{align}
and we obtain, by combining \eqref{TEPSI1} and \eqref{TEPSI2},
\begin{align}
J^i(\alpha^i, \hat{\alpha}^{-i}) \leq J^i(\hat{\alpha}^i, \hat{\alpha}^{-i}) + 2\epsilon_N. \label{TEPSI3}
\end{align}
Taking the supremum on both sides of \eqref{TEPSI3} completes the proof.
\end{proof}

This section, therefore, shows that the MFG solution provides an $\epsilon$-Nash solution to the $N$-player game.

\subsubsection{Numerical Solutions and Price Functions} \label{MFGSNM}
We follow the numerical scheme proposed in \cite[Page 7]{GMW2022}. First, by rewriting the forward-backward system \eqref{MUFSDE} - \eqref{MUBSDE} in a forward manner, we consider the Euler--Maruyama discretised forward system on a regular time grid $t_i = i\Delta t, \Delta t = \frac{T}{M}$ for $i\in\{0, \dots, M\}$:
\begin{align}
\mu_{t_{i+1}}^X &= \mu_{t_{i}}^X + \left( -\delta \mu_{t_{i}}^X + \frac{1}{2c_a}  \left( \mu_{t_{i}}^Y - c_I\right)\right) \Delta t + \sigma^0 \Delta W_{i}^0,\label{DFSDE}\\
\mu_{t_{i+1}}^Y & =  \mu_{t_{i}}^Y  + \left( \delta \mu_{t_{i}}^Y + c_p - P\left(t_i, \mu_{t_{i}}^X\right)  \right) \Delta t + \hat{Z}_{t_i}^0 \Delta W_{i}^0, \label{DBSDE}
\end{align}
with the Brownian motion increments $\Delta W_i^0 = (W_{t_{i+1}}^0 - W_{t_i}^0)$ and terminal condition $\mu_{t_{M}}^Y = 0$, initial condition $\mu_{t_{0}}^X =\xi_0$; the process $\hat{Z}_{t_i}^0$ is approximated by a single feedback neural network $z_{\theta_z}(t_i, \mu_{t_{i}}^X)$ and $\mu_0^Y$ by a neural network $y_{\theta_y}(0, \mu_0^X)$ with parameters $\theta = (\theta_y, \theta_z)$. The motivation for such approximation comes from the decoupling field in Proposition \ref{GWPOMFG}. As mentioned before, the forward-backward system can be transformed into a `forward-forward' system and an optimisation problem aiming to satisfy the terminal condition of the BSDE through minimising the loss function $L_Y = \mathbb{E}\left[ \left(\mu_{t_M}^Y\right)^2 \right]$. 

The following algorithms are implemented in Python with TensorFlow Library and the neural network is trained with the batch size of $2000$ and $1000$ training iterations, learning rate of $10^{-4}$.
\begin{breakablealgorithm}
\caption{Algorithm Solving FBSDEs \eqref{MUFSDE} - \eqref{MUBSDE}}
\begin{algorithmic}
\State Let $y_{\theta_y}(\cdot, \cdot)$ be a neural network with parameter $\theta_y$ defined on $[0, T]\times\mathbb{R}$, $z_{\theta_z}(\cdot, \cdot)$ be a neural network with parameter $\theta_z$ defined on $[0, T]\times\mathbb{R}$ so that $(\theta_y, \theta_z)$ is initialised with value $\pmb{\theta}_0 = (\theta_y^0, \theta_z^0)$. Let $K$ be the iterations, $B$ be the Batch size and $(\pmb{\rho}_k)_{k = 0, \dots, K-1}$ be the learning rate.
\For{$k$ from $0$ to $K-1$}
\State Choose a value for $\xi_0$, that is, choose the initial condition for all the samples
\State $\forall j\in B$,  $\mu_0^{Y, j} = y_{\theta_y^k}( 0, \xi_0)$
\For{$i$ from $0$ to $M-1$}
\For{$j$ from $1$ to $B$}
\State $t_i = i \Delta t$
\State Sample $\Delta W_{i}^0$ from normal distribution with mean $0$ and variance $\Delta t$
\State $\mu_{t_{i+1}}^{X, j} = \mu_{t_{i}}^{X, j} + \left( -\delta \mu_{t_{i}}^{X, j} + \frac{1}{2c_a}  \left( \mu_{t_{i}}^{Y, j} - c_I\right)\right) \Delta t + \sigma^0 \Delta W_{i}^0$
\State $\mu_{t_{i+1}}^{Y, j}  = \mu_{t_{i}}^{Y, j}  +  \left( \delta \mu_{t_{i}}^{Y, j} + c_p - P\left(t_i, \mu_{{t_{i}}}^{X, j}\right)  \right) \Delta t +z_{\theta_z^k}(t_i, \mu_{{t_{i}}}^{X, j}) \Delta W_{i}^0$
\EndFor
\EndFor
\State $J(\pmb{\theta}_k) = \frac{1}{B} \sum_{j=1}^B (\mu_{T}^{Y, j})^2$
\State Compute $\triangledown J(\pmb{\theta}_k)$, the gradient of $J$ with respect to $\mathbb{\theta}_k$, by back-propagation.
\State Update $\pmb{\theta}_{k+1} =  \pmb{\theta}_{k} - \pmb{\rho}_k \triangledown J(\pmb{\theta}_k)$.
\EndFor
\end{algorithmic}
\end{breakablealgorithm}

\begin{exam}[Solar Photovoltaic Technology] \label{SPTExam}
According to the renewable power generation cost in 2024 presented by the International Renewable Energy Agency \footnote{\href{https://www.irena.org/-/media/Files/IRENA/Agency/Publication/2025/Jul/IRENA_TEC_RPGC_in_2024_2025.pdf}{The Renewable power generation cost in 2024.}}, since 2010, the solar PV has experienced the most rapid cost reductions: (1) the average utility-scale operation and maintenance costs in Europe were reported at 11.3 USD per kW per year; (2) the global weighted average levelised cost of electricity (LCOE) of utility-scale PV plants was reported at 0.043 USD per kWh. The weighted average wholesale price for solar PV-generated electricity was reported approximately 125 dollars per MWh from a range of plausible values 38 dollars per MWh to 303 dollars per MWh. The median degradation rate of photovoltaic modules is $0.5\%$ per year \footnote{\href{https://www.nlr.gov/pv/lifetime}{Photovoltaic Lifetime Project.}}, according to the United States National Renewable Energy Laboratory. 
That is, assuming the average number of operating hours for a $1 MW$ solar PV installation (which requires 4 to 5 acres of land for installation of panels) is $2000$ hours per year, $c_p = \frac{11.3 \$ \slash kW}{2000h} = 0.00565 \$ \slash kWh = 5.65 \$ \slash MWh $, $c_I = 0.043 \$ \slash kWh - c_p = 0.03735 \$ \slash kWh = 37.35 \$ \slash MWh$, $c_a  =1 \$^2 \slash MWh^2$, $\delta = 0.5\%$. As for the price function, we assume it is influenced both by the supply and demand, that is, the marginal price model developed in \cite{ACL2013} by taking into account how the marginal capacity uncertainty contributes to future prices ($ \$\slash MWh$):
\begin{align}
P(t, \mu_t^X) = 
\begin{cases}
\min\bigg\{ P_{cap},  p_0 + \frac{p_1}{(\mu_t^X - D)^r} \bigg\}, & \mu_t^X - D > 0, \\
P_{cap}, & \mu_t^X - D \leq 0,
\end{cases}\label{PF1}
\end{align}
where $P_{cap} = 300 \$\slash MWh$ is the capped spot price and $p_0= 30 \$ \slash MWh, p_1 = 27500 \$$, see \cite[page 6]{ACL2013} and $D$ is a constant for simplicity.

From Figure \ref{Exam02} and Figure \ref{Exam04}, we see that, due to the excess demand, the electricity spot price remains around its capped price such that the expected production profit remains at its highest,
\begin{align*}
\mu_t^{Y} = \left( 300 - 5.65 \right) \int_t^T e^{-0.005 (s-t)}  ds,
\end{align*} 
and thereby, $\hat{\alpha}_t > 0$ for $t\in[0, 0.873]$ and $T=1$ (for $t\in[0, 1.87]$ and $T=2$), that is, as long as the expected profit (of producing one unit of energy) exceeds the installation cost $c_I$, more generation capacities are installed. In the long run, under such market mechanism, the supply shortage will be resolved.

\begin{figure}[H]
\centering
\subfloat[$\mu^X$]{\includegraphics[width=0.48\textwidth]{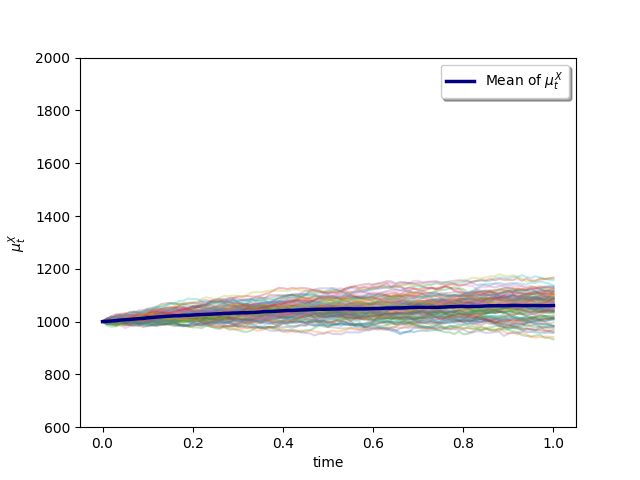}}
\subfloat[$\mu^Y$ and $\hat{\alpha}$]{\includegraphics[width=0.48\textwidth]{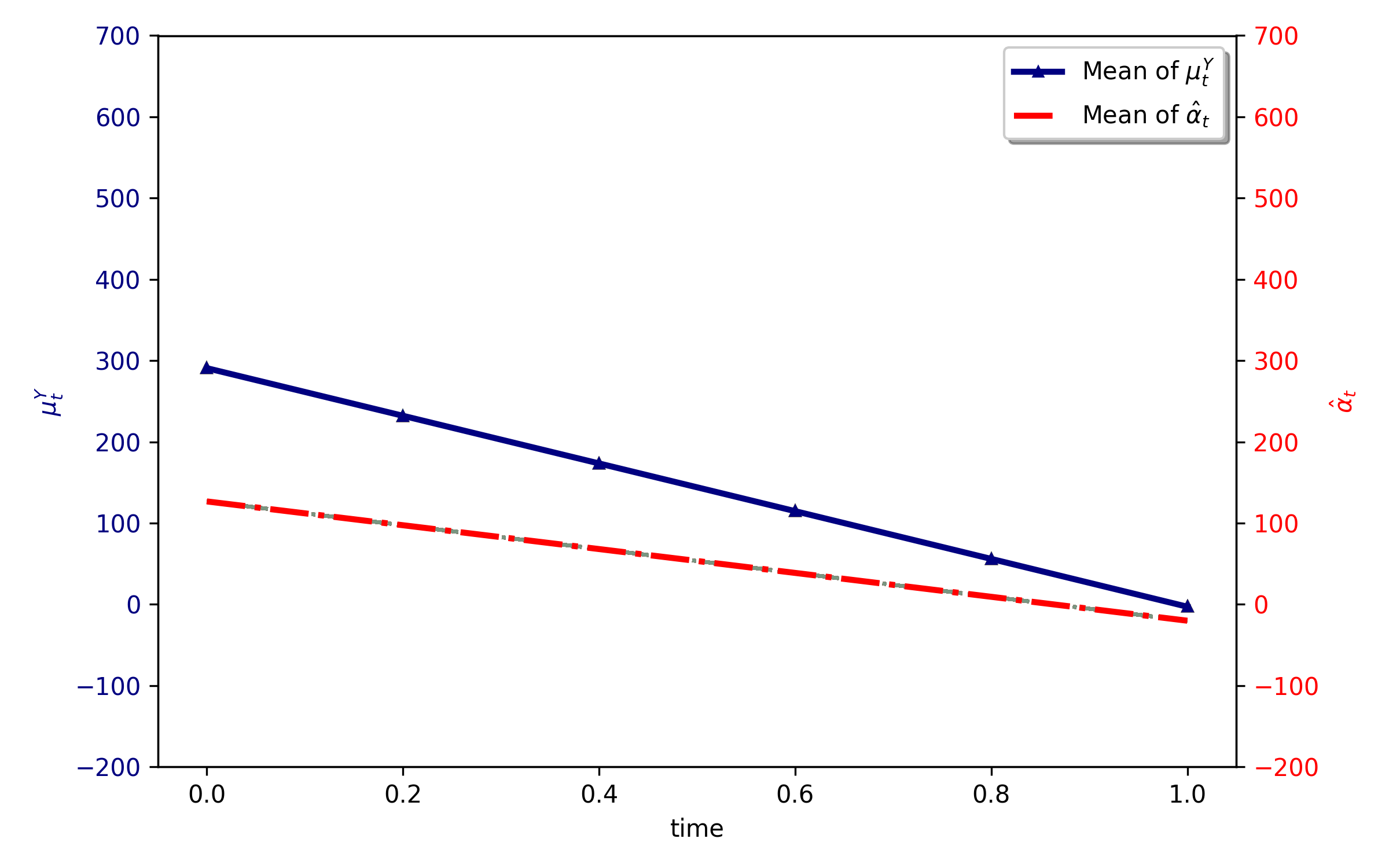}}
\caption{$\mu_0^X = 1000 MWh$, $D = 1500 MWh$, $\sigma^0 = 50$, $T=1$, $r = 1$, $L_Y\approx 0.00002$.}
\label{Exam02}
\end{figure}
\begin{figure}[H]
\centering
\subfloat[$\mu^X$]{\includegraphics[width=0.48\textwidth]{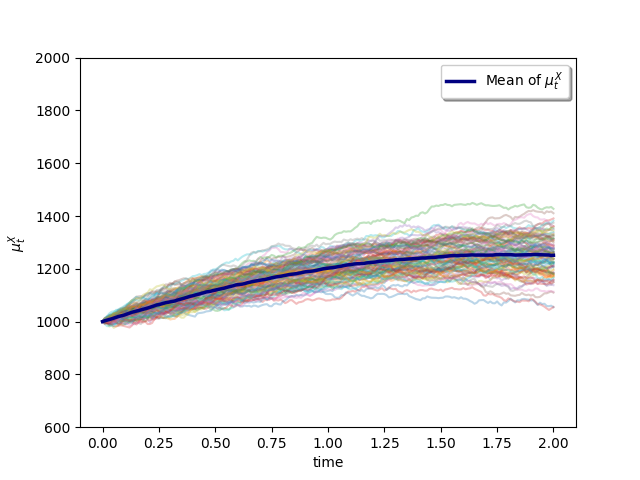}}
\subfloat[$\mu^Y$ and $\hat{\alpha}$]{\includegraphics[width=0.48\textwidth]{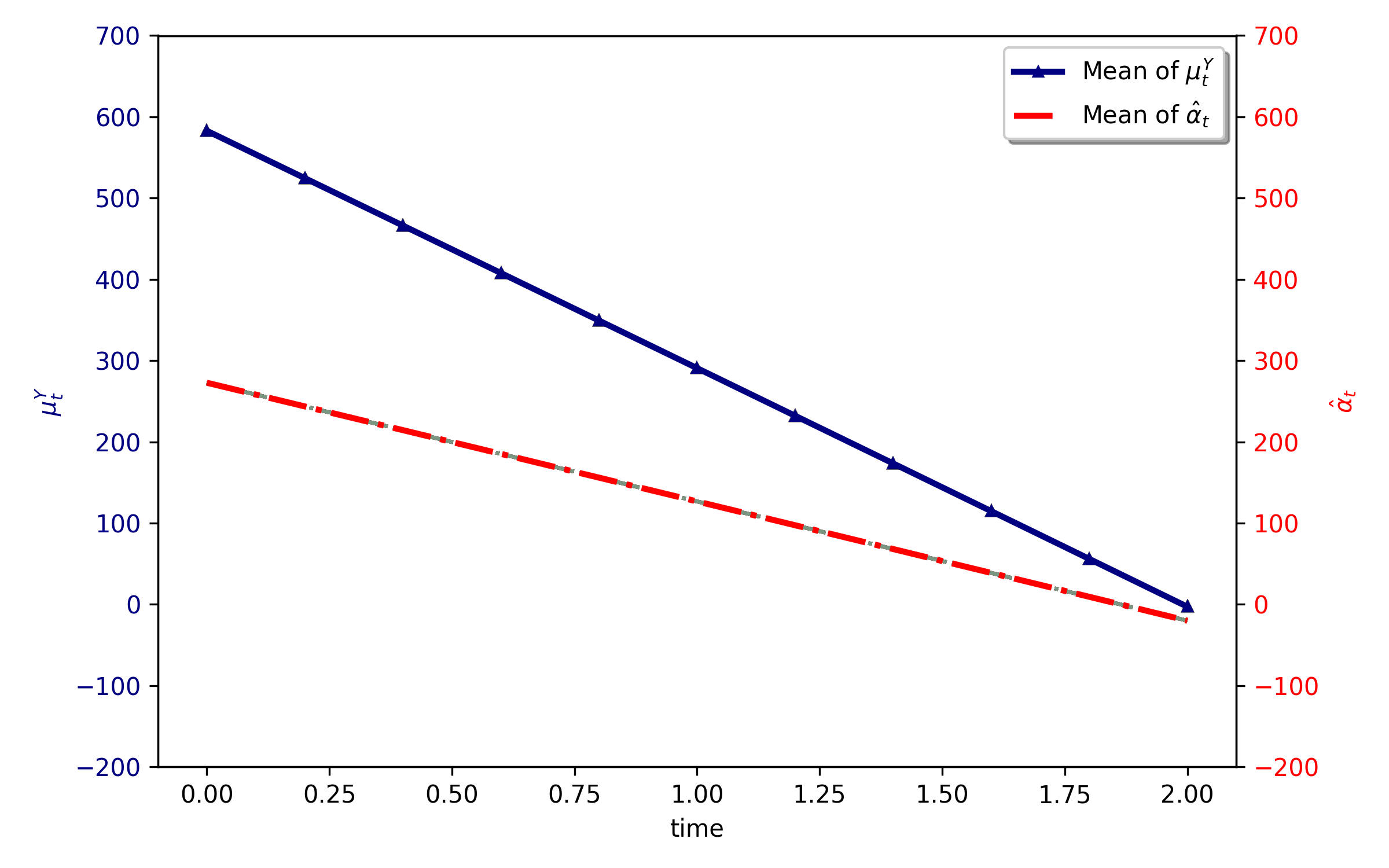}}
\caption{$\mu_0^X = 1000 MWh$, $D = 1500 MWh$, $\sigma^0 = 50$, $T=2$, $r = 1$, $L_Y\approx0.00005$.}
\label{Exam04}
\end{figure}

Regarding Figures \ref{Exam002}-\ref{Exam14}, under risk neutrality and quadratic costs, volatility in common shocks affects the dispersion of capacity paths but not the direction of capacity trends. Also note that, in Figures \ref{Exam02}, \ref{Exam04}, \ref{Exam002} and \ref{Exam14}, the price function remains at its cap price, as $\mu_t^X\leq D$ for all t$\in[0, T]$.

\begin{figure}[H]
\centering
\subfloat[$\mu^X$]{\includegraphics[width=0.48\textwidth]{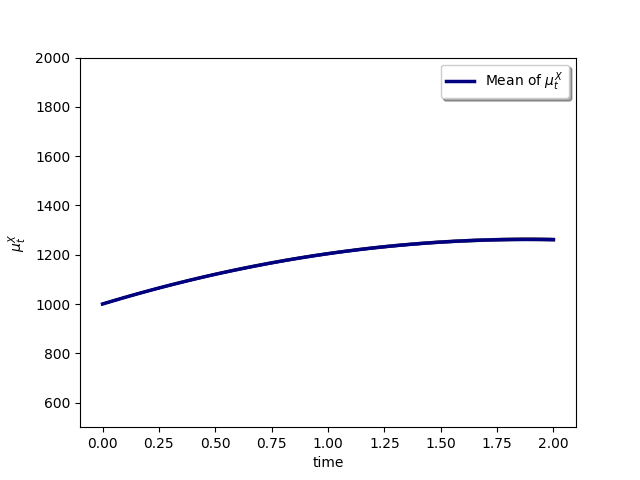}}
\subfloat[$\mu^Y$ and $\hat{\alpha}$]{\includegraphics[width=0.48\textwidth]{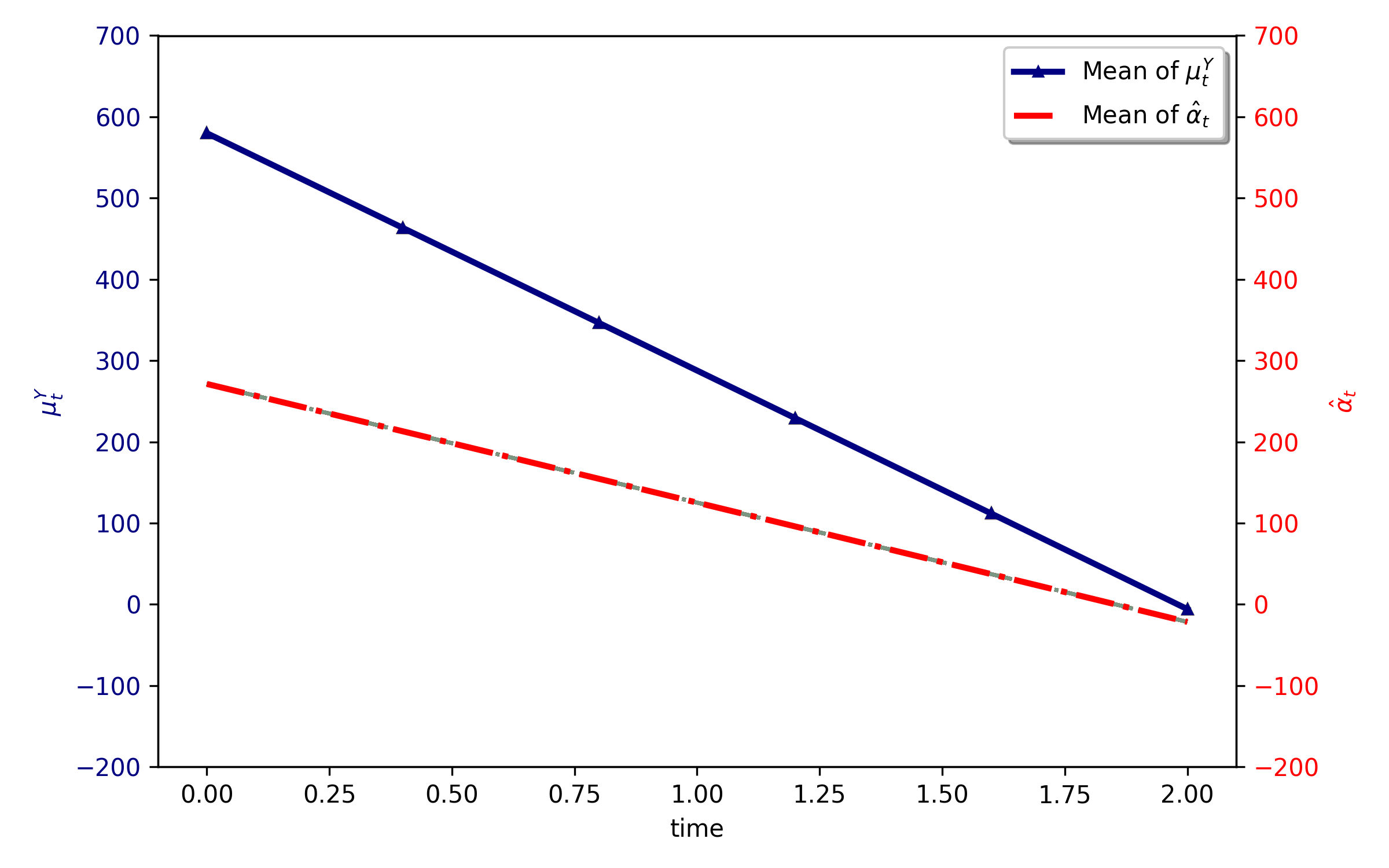}}
\caption{$\mu_0^X = 1000 MWh$, $D = 1500 MWh$, $\sigma^0 = 1$, $T=2$, $r = 1$, $L_Y\approx 0.00202$.}
\label{Exam002}
\end{figure}

\begin{figure}[H]
\centering
\subfloat[$\mu^X$]{\includegraphics[width=0.48\textwidth]{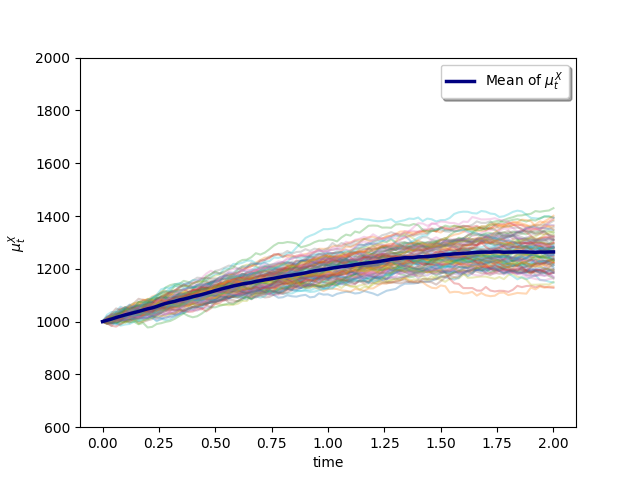}}
\subfloat[$\mu^Y$ and $\hat{\alpha}$]{\includegraphics[width=0.48\textwidth]{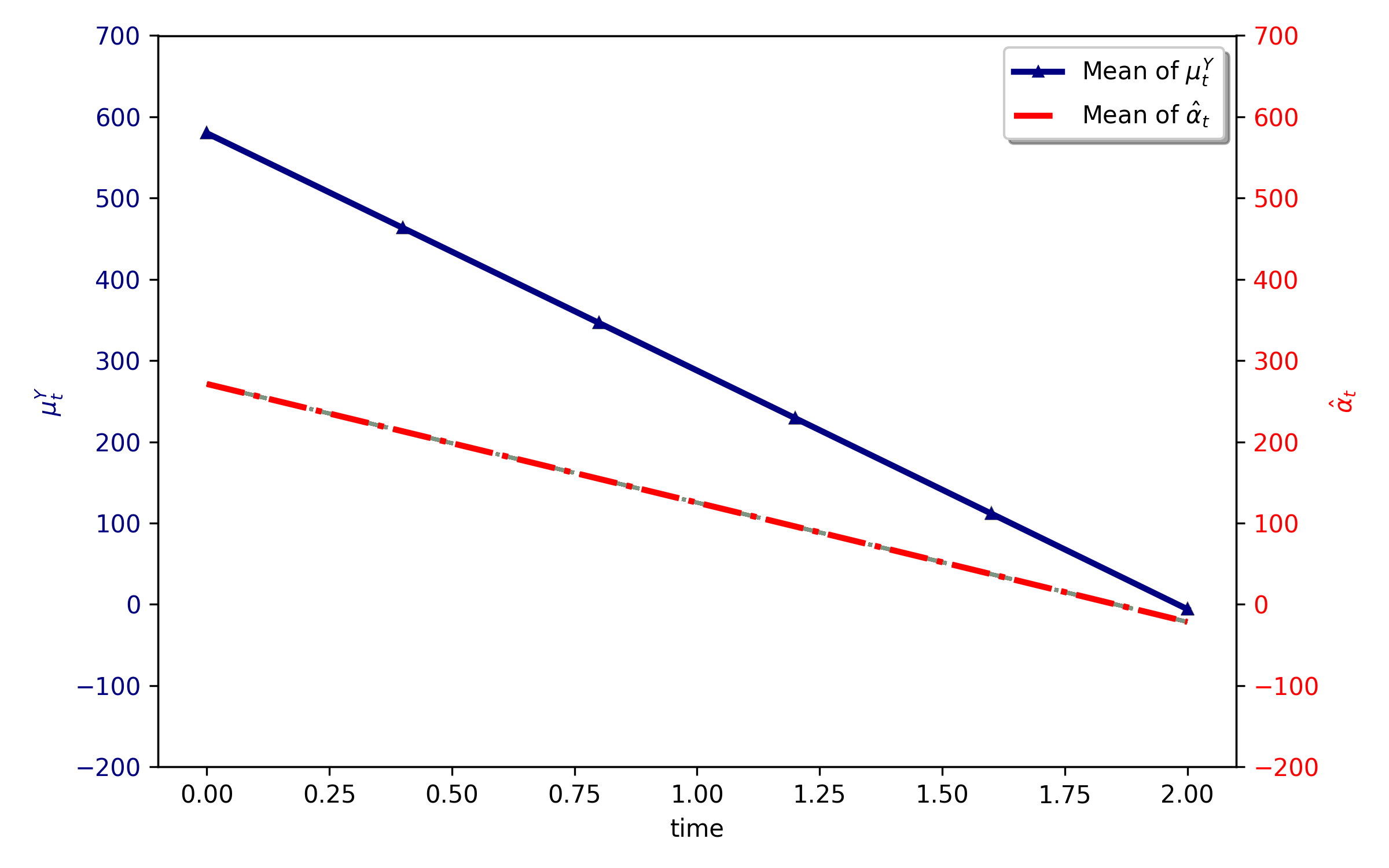}}
\caption{$\mu_0^X = 1000 MWh$, $D=1500MWh$, $\sigma^0  = 50$, $T=2$, $r = 2$, $L_Y\approx 0.01068$.}
\label{Exam14}
\end{figure}

As for the case of excess supply in Figures \ref{Exam01} and \ref{Exam03}, since the cost of producing solar power is relatively cheap and continues to decline due to technology advancement, the expected profit remains positive; at around $t=0.5$ for $T =1$ ($t = 1.5$ for $T = 2$), the installation cost exceeds the expected profit and thereby, no new capacity is installed afterwards. Overall, during this period, the solar energy generation capacity remains fairly flat under the current market mechanism.

\begin{figure}[H]
\centering
\subfloat[$\mu^X$ and price function]{\includegraphics[width=0.48\textwidth]{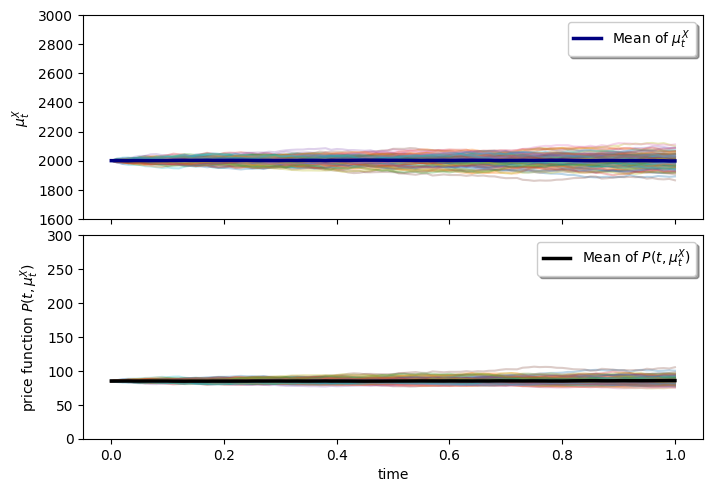}}
\subfloat[$\mu^Y$ and $\hat{\alpha}$]{\includegraphics[width=0.48\textwidth]{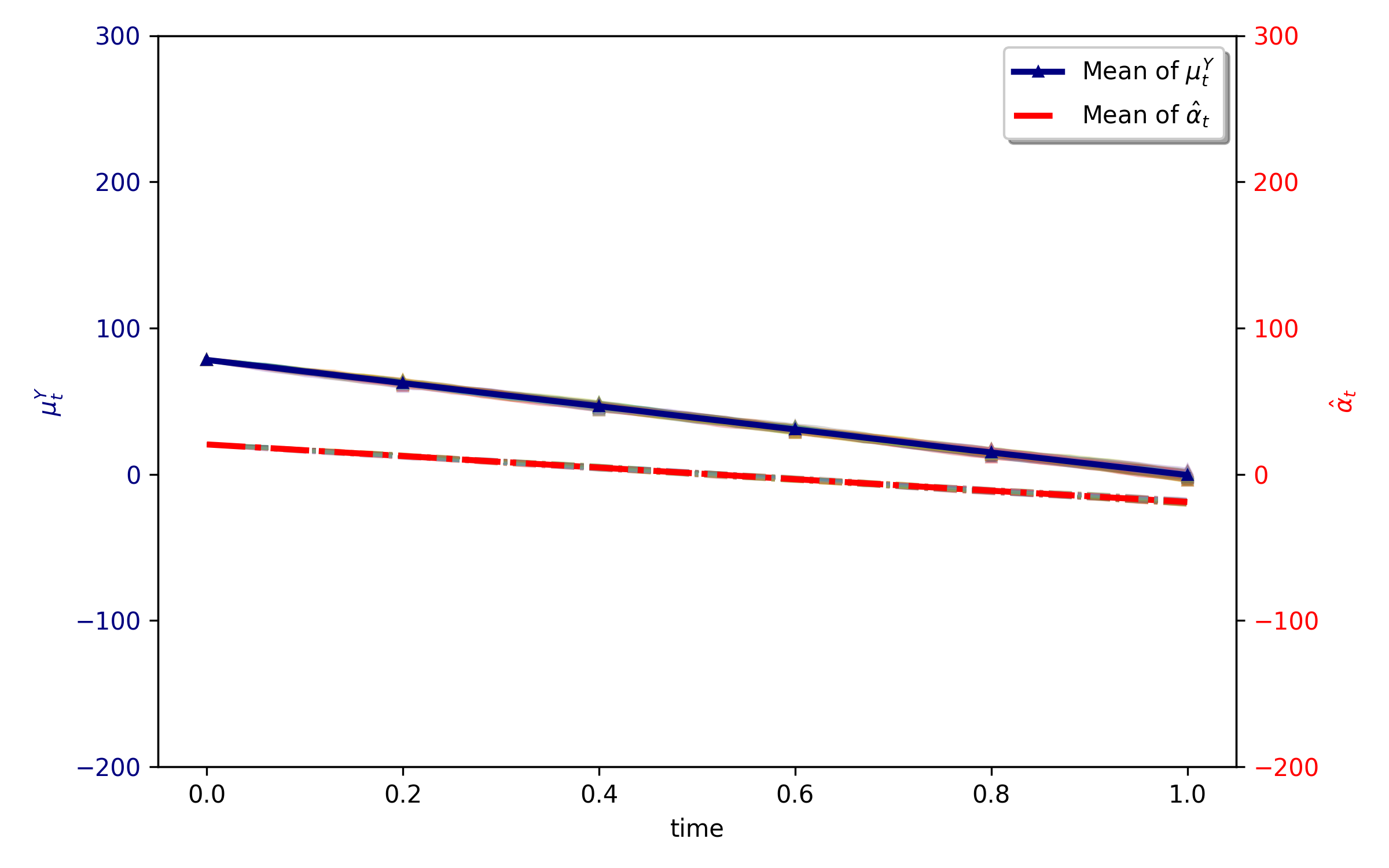}}
\caption{$\mu_0^X = 2000 MWh$, $D = 1500 MWh$, $\sigma^0 = 50$, $T=1$, $r=1$, $L_Y\approx 0.80934$.}
\label{Exam01}
\end{figure}

\begin{figure}[H]
\centering
\subfloat[$\mu^X$ and price function]{\includegraphics[width=0.48\textwidth]{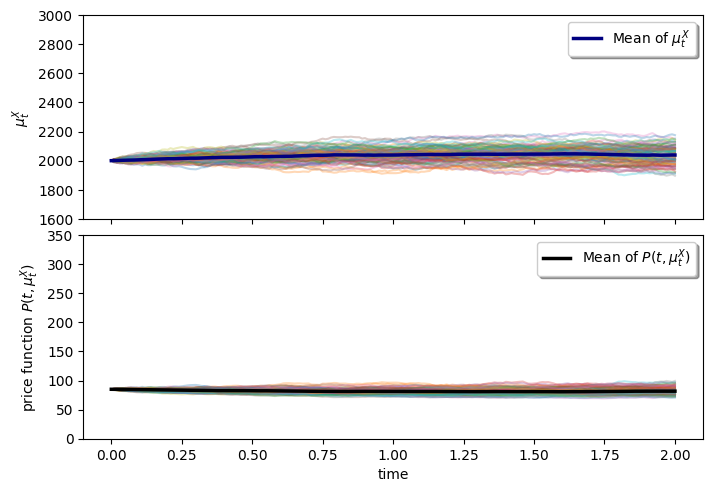}}
\subfloat[$\mu^Y$ and $\hat{\alpha}$]{\includegraphics[width=0.48\textwidth]{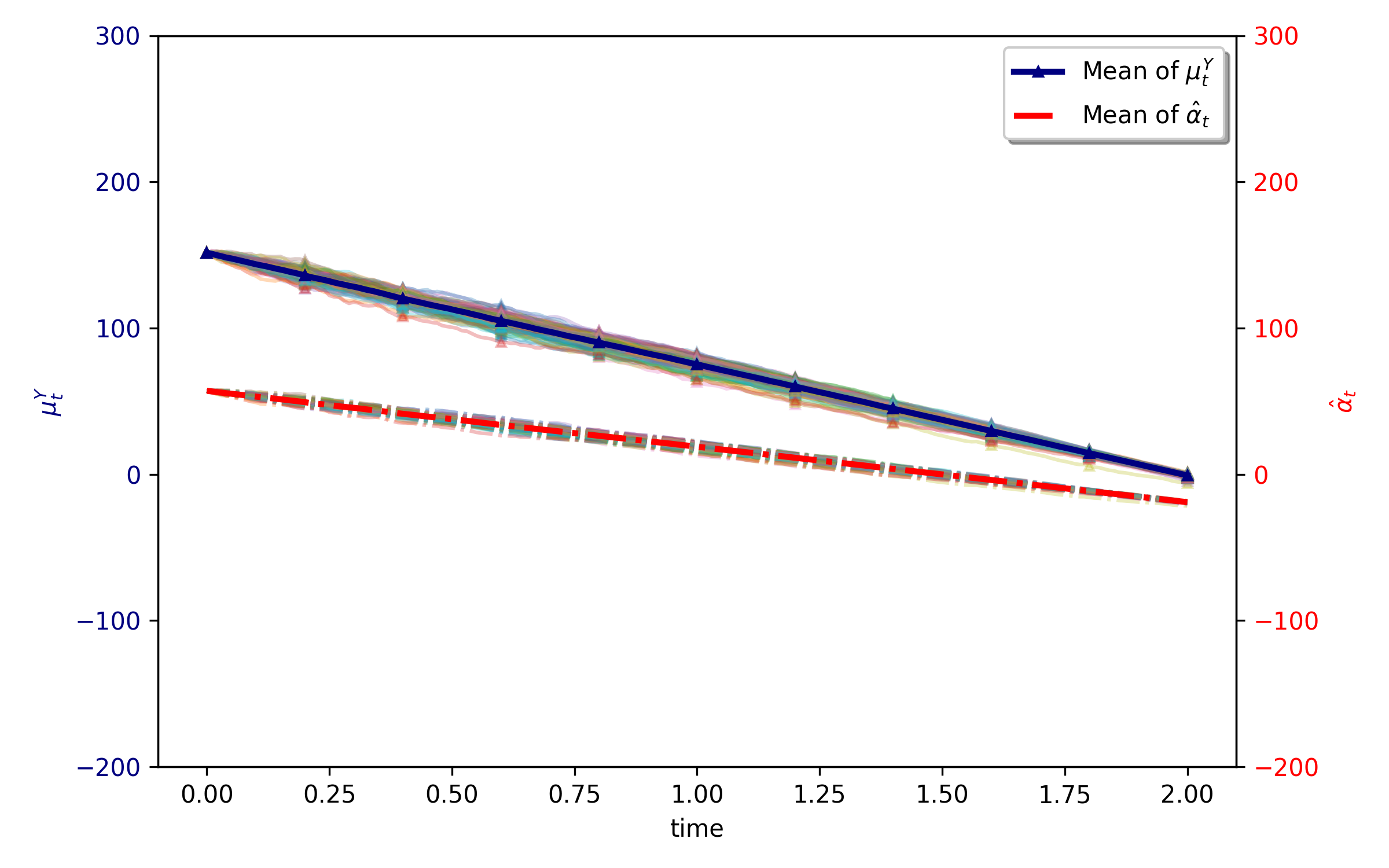}}
\caption{$\mu_0^X = 2000 MWh$, $D = 1500 MWh$, $\sigma^0 = 50$, $T=2$, $r=1$, $L_Y\approx2.71849$ after 8000 iterations.}
\label{Exam03}
\end{figure}

With $\sigma^0$ being relatively small, we see a slight growth in the generation capacity which leads to overproduction, see Figure \ref{Exam32}; whereas, if the price function is more sensitive to the marginal capacity, then the market will slowly resolve the tension from overproducing renewable energy, see Figure \ref{Exam31}.

\begin{figure}[H]
\centering
\subfloat[$\mu^X$ and price function]{\includegraphics[width=0.48\textwidth]{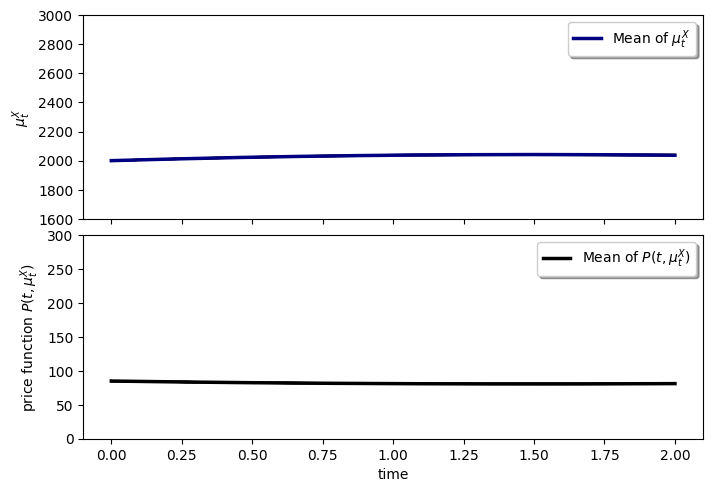}}
\subfloat[$\mu^Y$ and $\hat{\alpha}$]{\includegraphics[width=0.48\textwidth]{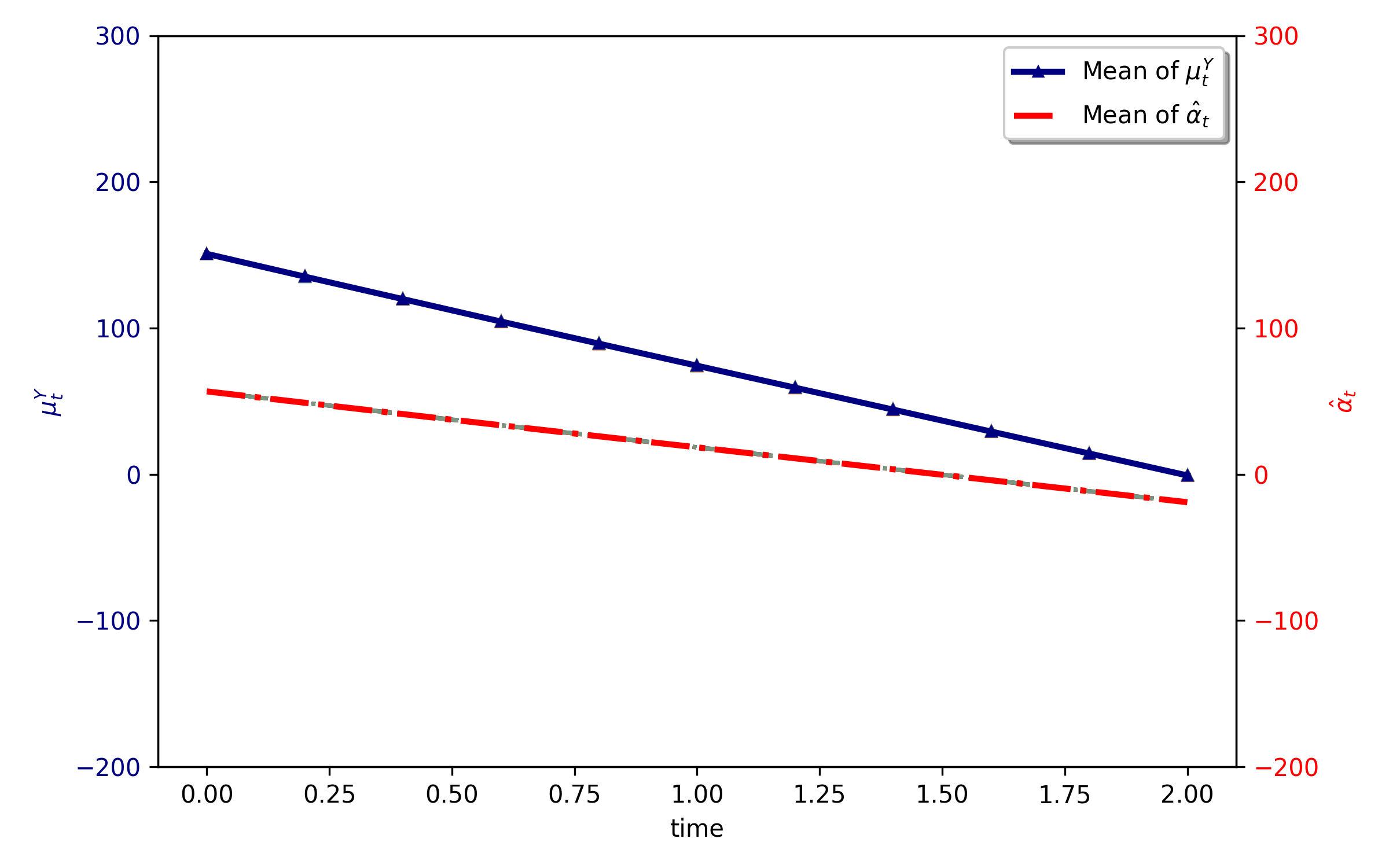}}
\caption{$\mu_0^X = 2000 MWh$, $D = 1500 MWh$, $\sigma^0 = 1$, $T=2$, $r=1$, $L_Y\approx0.04587$.}
\label{Exam32}
\end{figure}

\begin{figure}[H]
\centering
\subfloat[$\mu^X$ and price function]{\includegraphics[width=0.48\textwidth]{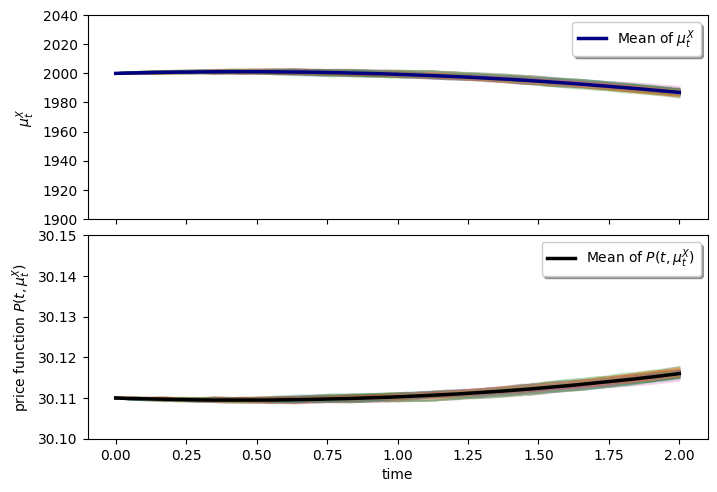}}
\subfloat[$\mu^Y$ and $\hat{\alpha}$]{\includegraphics[width=0.48\textwidth]{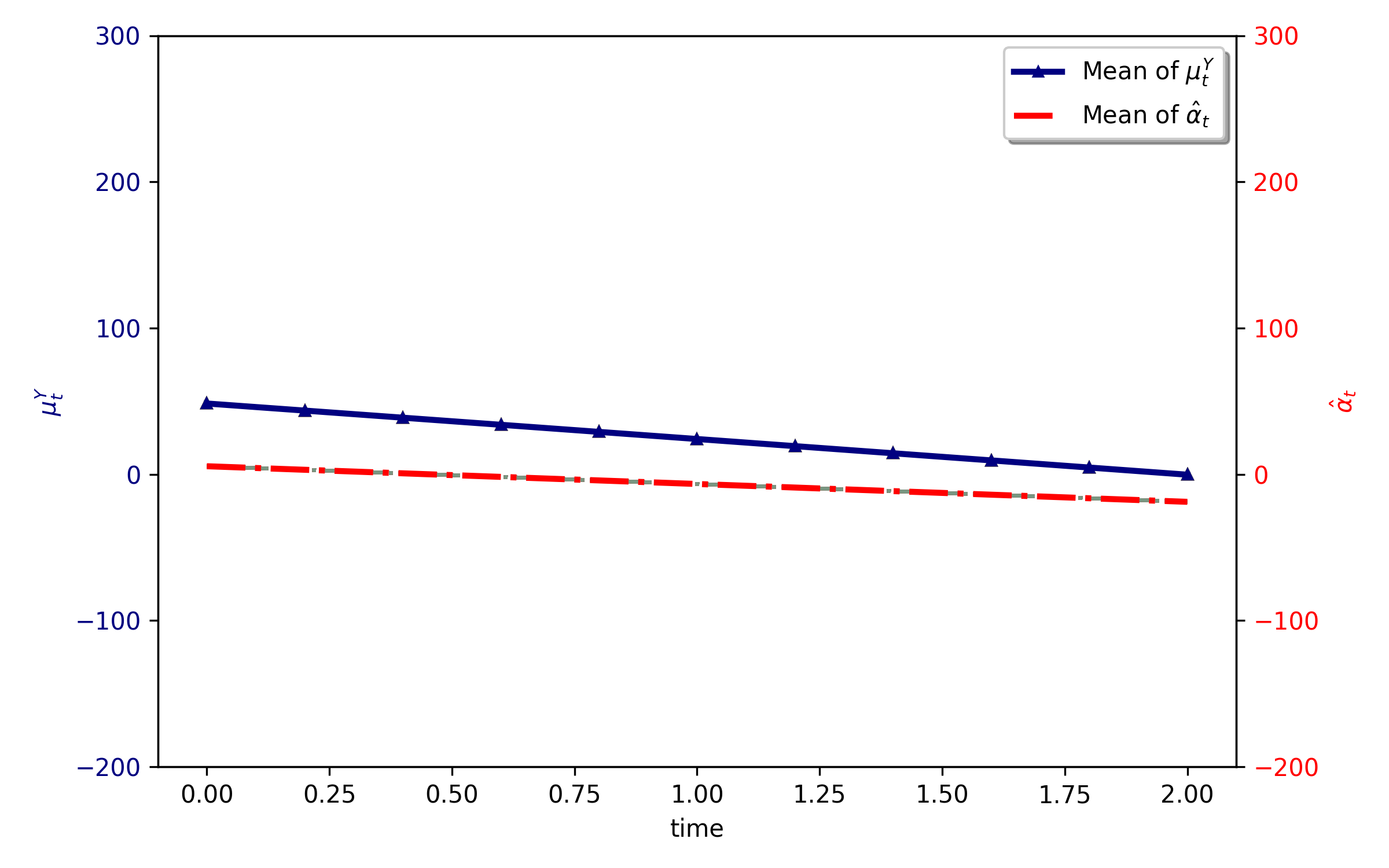}}
\caption{$\mu_0^X = 2000 MWh$, $D = 1500 MWh$, $\sigma^0 = 1$, $T=2$, $r=2$, $L_Y\approx0.00474$.}
\label{Exam31}
\end{figure}
Finally, we stress that our plots show periods where $\hat \alpha_t$ becomes negative, that would correspond to active decommissioning. The model then captures that when revenues are too low, it is optimal to shut down capacity even before natural degradation. The same happens in the next example.
\end{exam}

\begin{exam}[Solar Photovoltaic Technology Continued] \label{E2C}
Following \cite[Page 290]{B2002}, we extend the price function defined in \cite[Page 698]{ABBC2023} such that
\begin{align}
P(\mu_t^X)=
\begin{cases}
\min\{p_0 + \frac{p_1}{(\mu_t^X+\epsilon_1)^r}, P_{cap}\}, & \mu_t^X + \epsilon_1 \geq \epsilon_2,\\
\min\{p_0 + \frac{p_1}{\epsilon_2^r}, P_{cap}\}, &\mu_t^X + \epsilon_1 < \epsilon_2.
\end{cases} \label{NLPFs1}
\end{align}
where $p_0 = 30 \$ \slash MWh$, $p_1 = 405000 \$$ for $r=1$ and $p_1 = 607500000 \$\cdot MWh$ for $r = 2$, $\epsilon_1 = 0.0001 MWh$, $\epsilon_2 = D =1500 MWh$ such that the capped price is $300 \$ \slash MWh$.

Figures \ref{Exam12}, \ref{Exam128}, \ref{Exam001} and  \ref{Exam007} tell the same story as Figures \ref{Exam02}, \ref{Exam04}, \ref{Exam002} and \ref{Exam14}, which is expected as we set the same capped price for excess demand. Notice that as before, the price function remains at the cap price as $\mu_t^X < D$ for all $t\in[0, T]$. 

\begin{figure}[H]
\centering
\subfloat[$\mu^X$]{\includegraphics[width=0.48\textwidth]{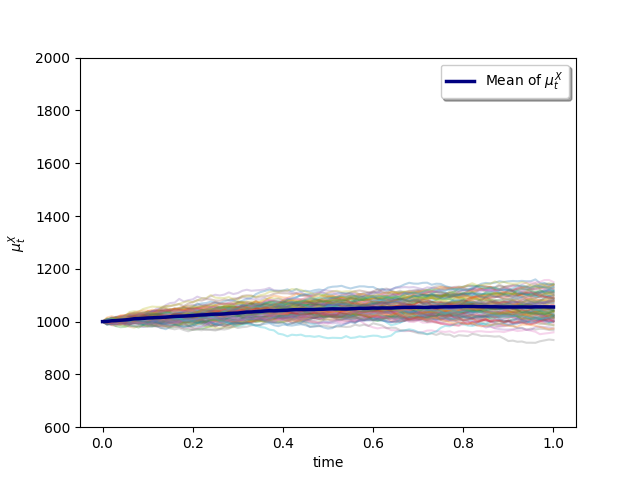}}
\subfloat[$\mu^Y$ and $\hat{\alpha}$]{\includegraphics[width=0.48\textwidth]{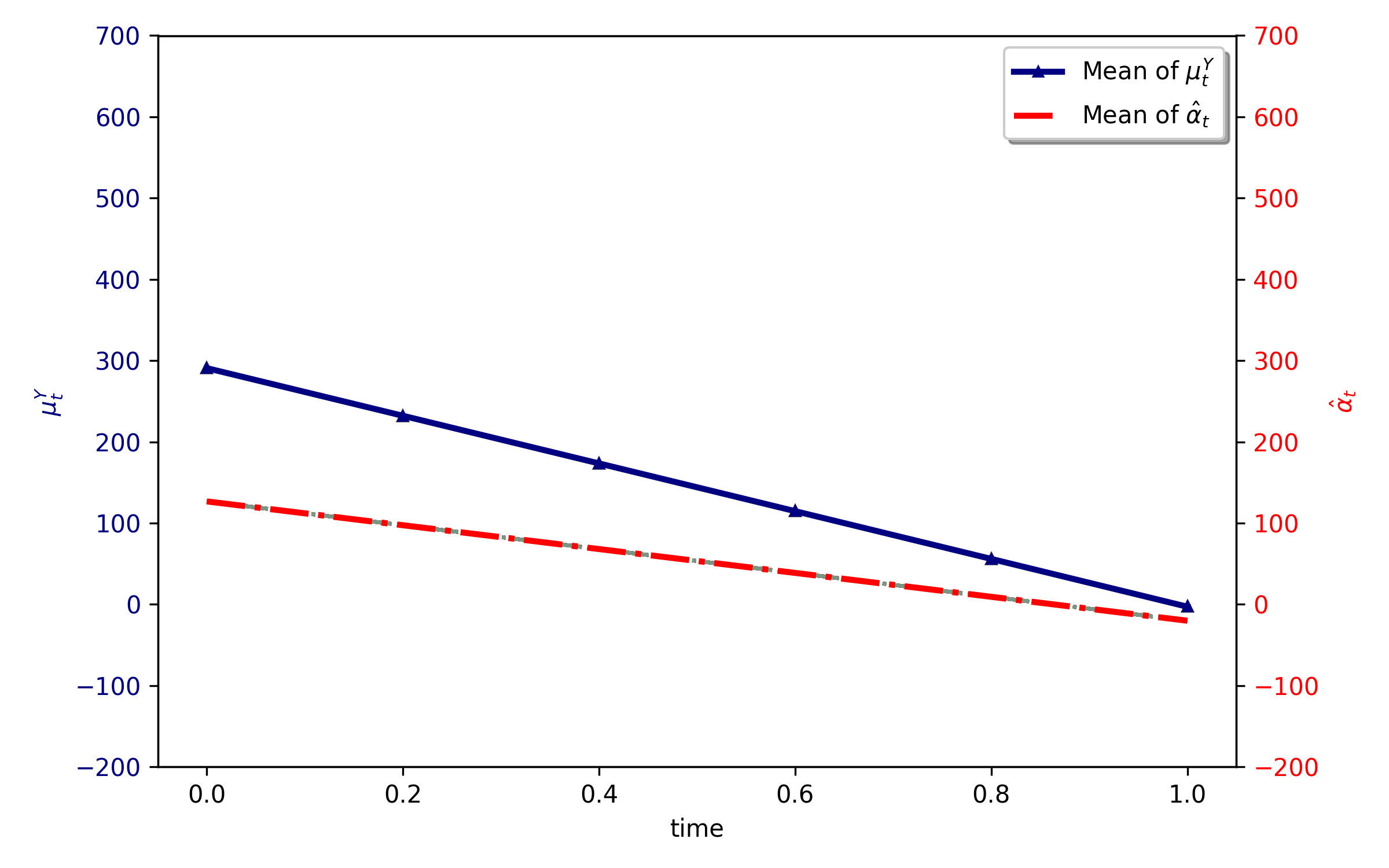}}
\caption{$\mu_0^X = 1000 MWh$, $D = 1500MWh$, $\sigma^0  = 50$, $T=1$, $r=1$, $L_Y\approx0.000003$. }
\label{Exam12}
\end{figure}

\begin{figure}[H]
\centering
\subfloat[$\mu^X$]{\includegraphics[width=0.48\textwidth]{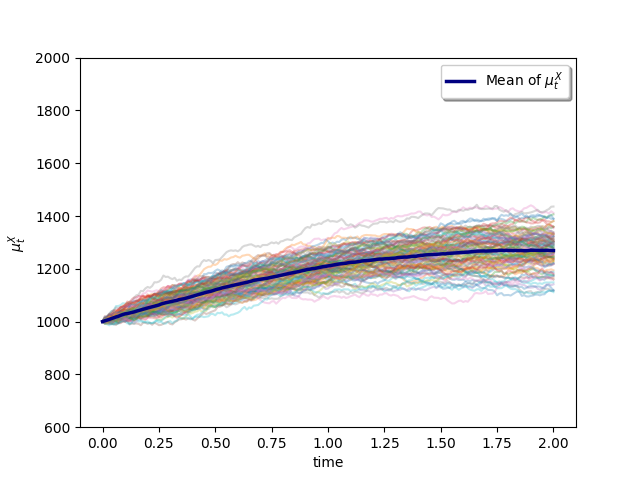}}
\subfloat[$\mu^Y$ and $\hat{\alpha}$]{\includegraphics[width=0.48\textwidth]{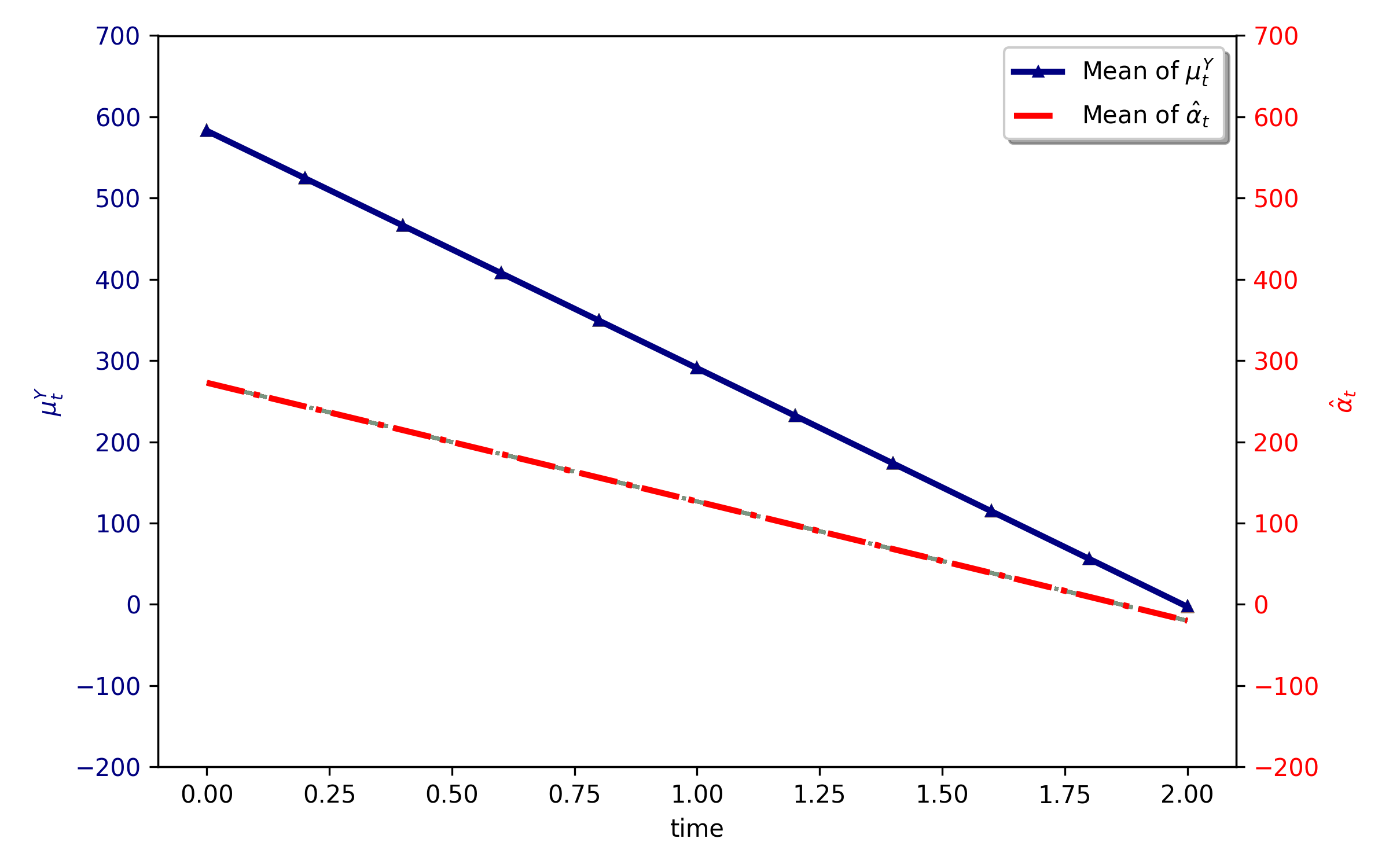}}
\caption{$\mu_0^X = 1000 MWh$, $D = 1500MWh$, $\sigma^0  = 50$, $T=2$, $r=1$, $L_Y\approx3.32101$ }
\label{Exam128}
\end{figure}

\begin{figure}[H]
\centering
\subfloat[$\mu^X$]{\includegraphics[width=0.48\textwidth]{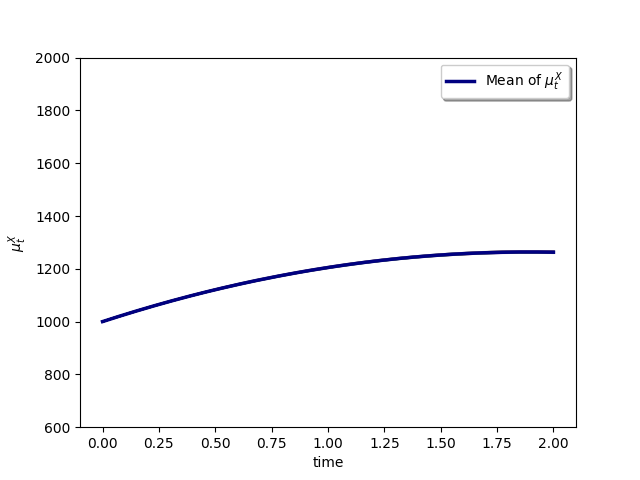}}
\subfloat[$\mu^Y$ and $\hat{\alpha}$]{\includegraphics[width=0.48\textwidth]{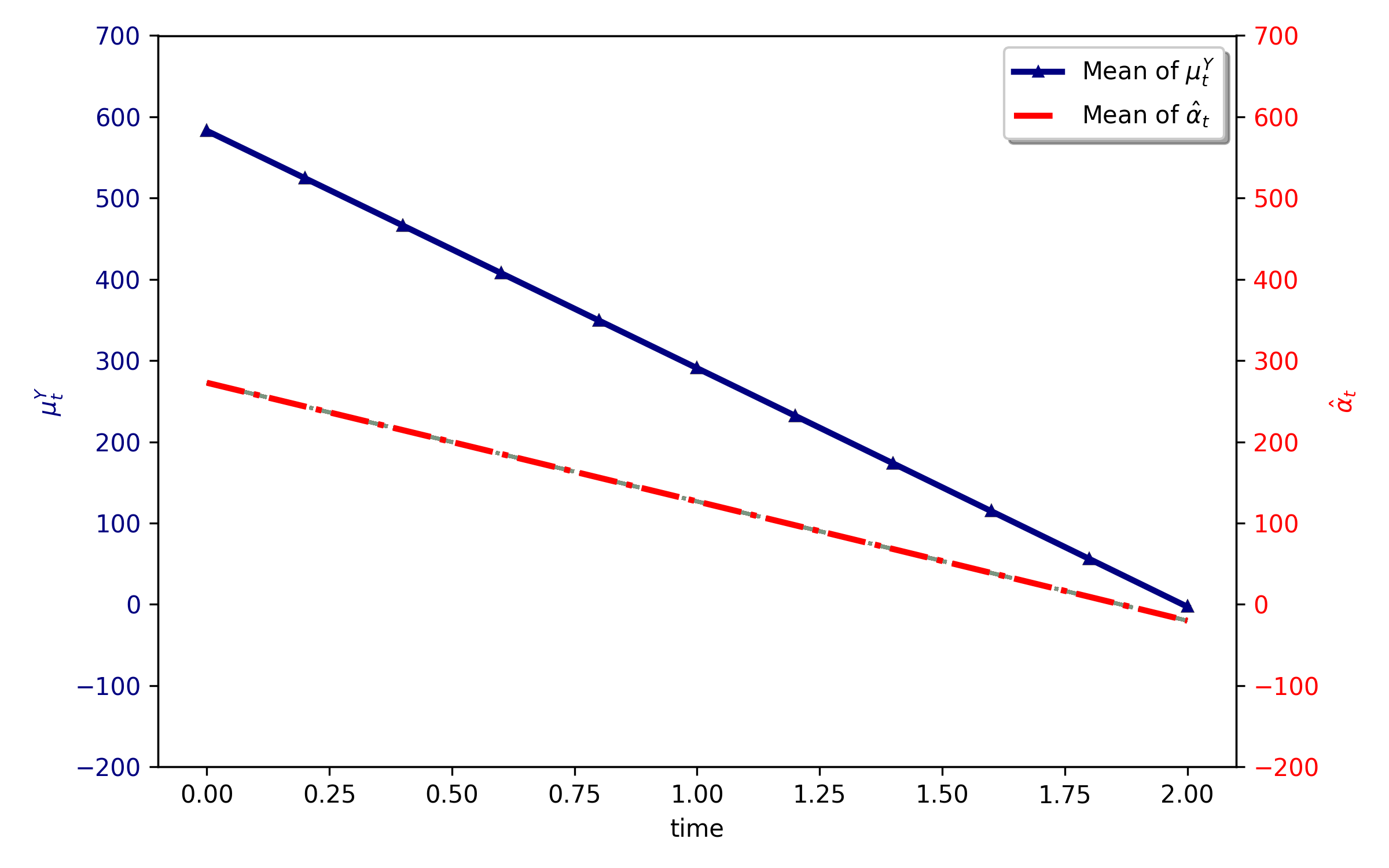}}
\caption{$\mu_0^X = 1000 MWh$, $D = 1500MWh$, $\sigma^0 = 1$, $T=2$, $r=1$, $L_Y\approx0.00001$.}
\label{Exam001}
\end{figure}

\begin{figure}[H]
\centering
\subfloat[$\mu^X$]{\includegraphics[width=0.48\textwidth]{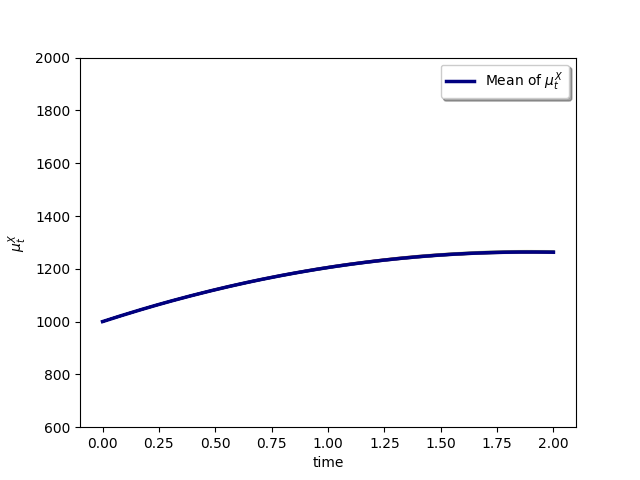}}
\subfloat[$\mu^Y$ and $\hat{\alpha}$]{\includegraphics[width=0.48\textwidth]{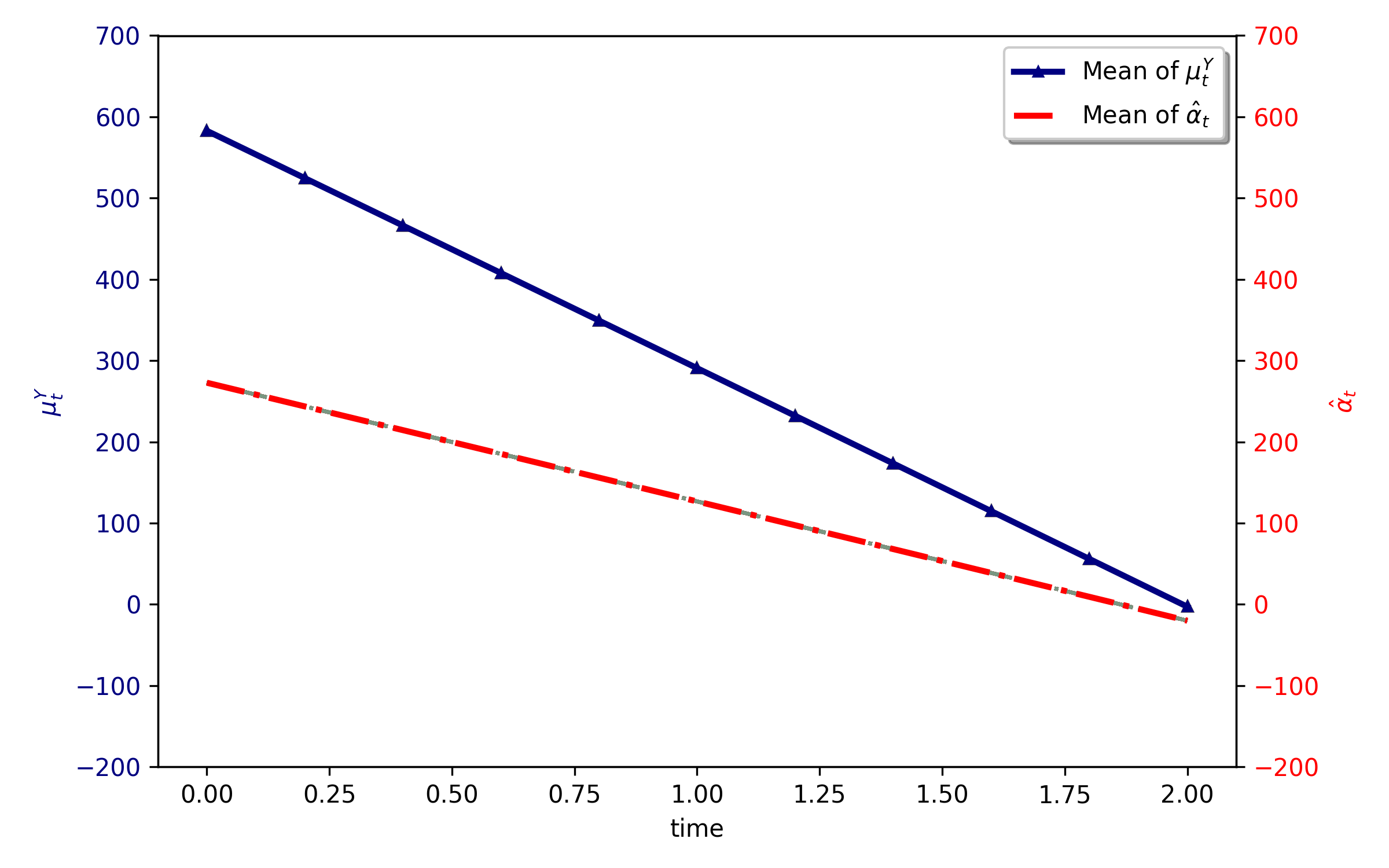}}
\caption{$\mu_0^X = 1000 MWh$, $D = 1500MWh$, $\sigma^0 = 1$, $T=2$, $r=2$, $L_Y\approx0.00001$.}
\label{Exam007}
\end{figure}

With excess supply and $r=1$, the current market mechanism not only fails to adjust the energy supply and demand but also intensifies the situation, see Figures \ref{Exam11}, \ref{Exam13} and \ref{Exam1202}. For $r=2$, although the price function is more sensitive to the generation capacity, we still see a tendency (but slower, compared with $r=1$) of increase in the development of new energy generation capacity in Figure \ref{Exam1201}.

\begin{figure}[H]
\centering
\subfloat[$\mu^X$ and price function]{\includegraphics[width=0.48\textwidth]{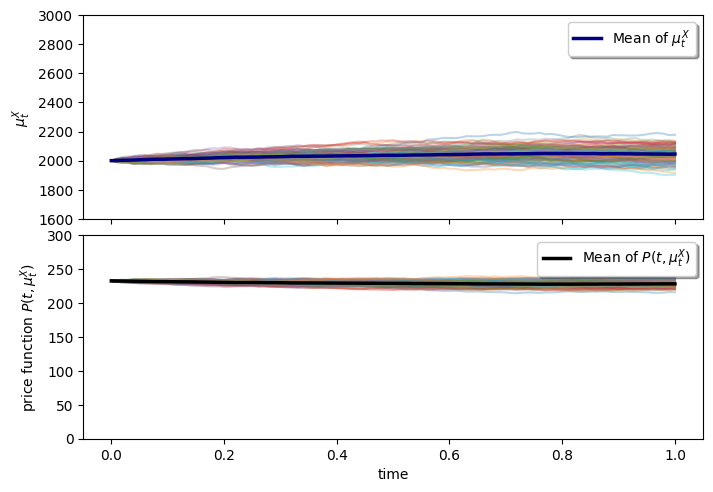}}
\subfloat[$\mu^Y$ and $\hat{\alpha}$]{\includegraphics[width=0.48\textwidth]{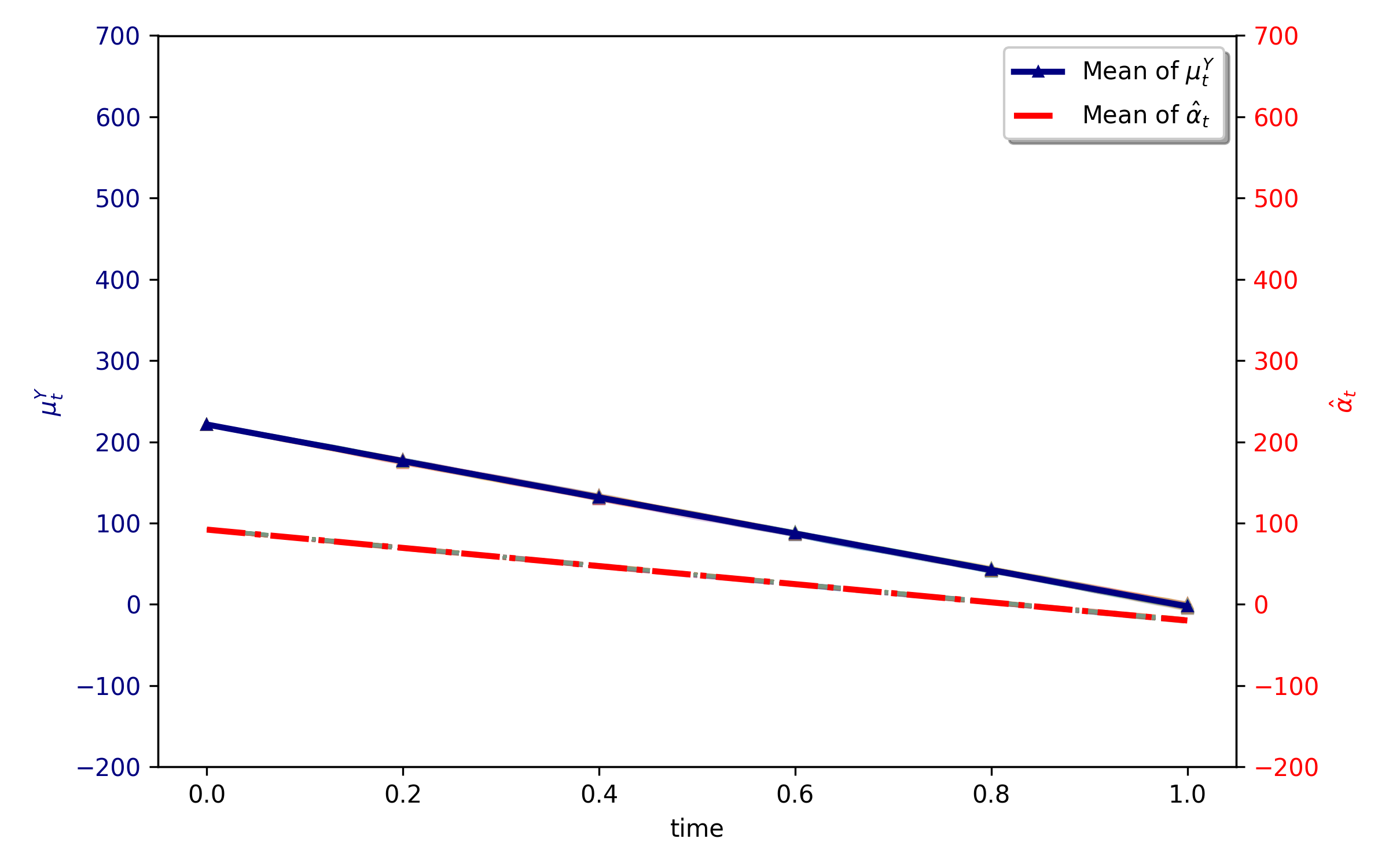}}
\caption{$\mu_0^X = 2000 MWh$, $D = 1500MWh$, $\sigma^0  = 50$, $T=1$ and $r=1$, $L_Y\approx2.14013$.}
\label{Exam11}
\end{figure}

\begin{figure}[H]
\centering
\subfloat[$\mu^X$ and price function]{\includegraphics[width=0.48\textwidth]{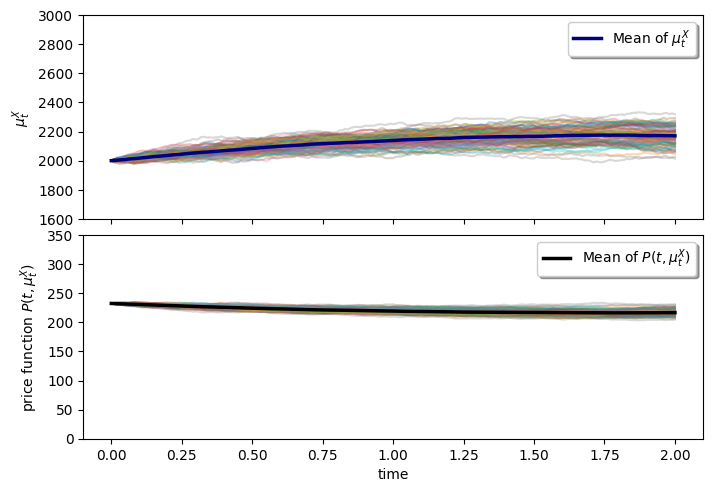}}
\subfloat[$\mu^Y$ and $\hat{\alpha}$]{\includegraphics[width=0.48\textwidth]{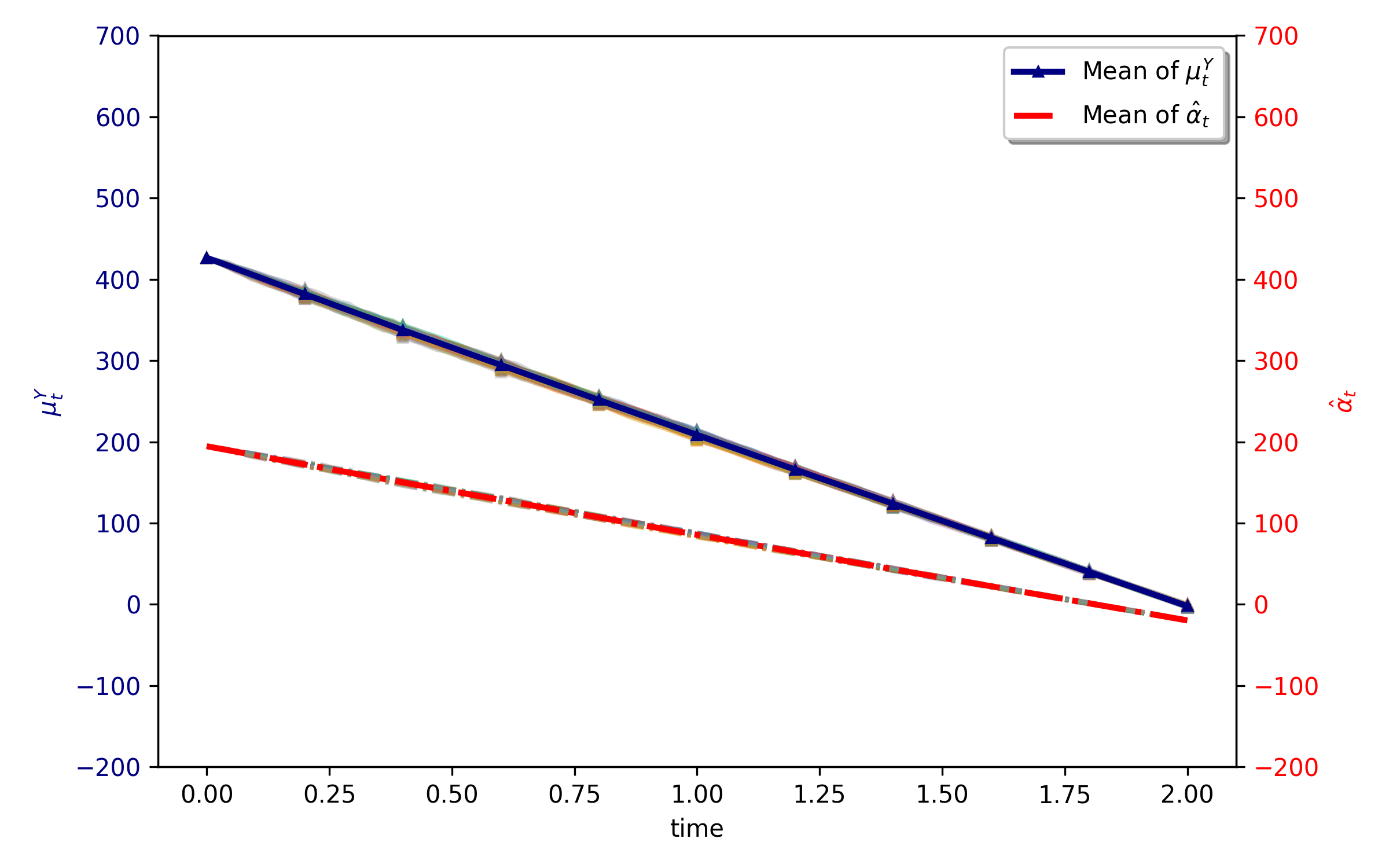}}
\caption{$\mu_0^X = 2000 MWh$, $D = 1500MWh$, $\sigma^0  = 50$,  $T=2$, $r=1$, $L_Y\approx3.34998$ after $10000$ iterations.}
\label{Exam13}
\end{figure}
\end{exam}

\begin{figure}[H]
\centering
\subfloat[$\mu^X$ and price function]{\includegraphics[width=0.48\textwidth]{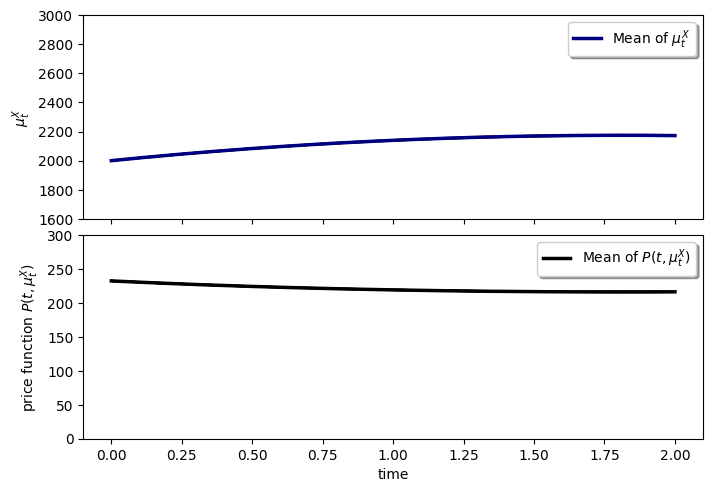}}
\subfloat[$\mu^Y$ and $\hat{\alpha}$]{\includegraphics[width=0.48\textwidth]{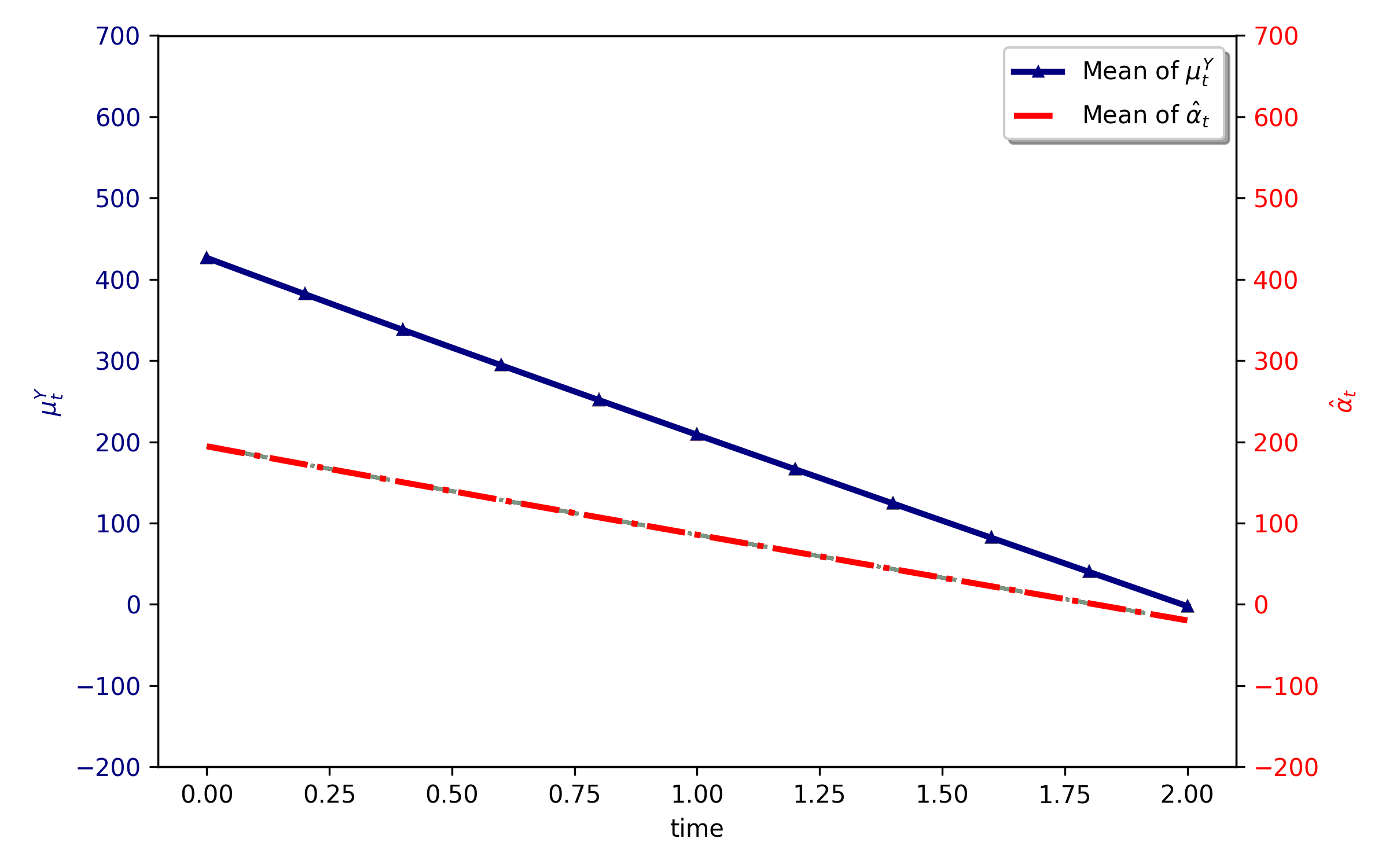}}
\caption{$\mu_0^X = 2000 MWh$, $D = 1500MWh$, $\sigma^0  = 1$, $T=2$, $r=1$, $L_Y\approx0.00119$.}
\label{Exam1202}
\end{figure}

\begin{figure}[H]
\centering
\subfloat[$\mu^X$ and price function]{\includegraphics[width=0.48\textwidth]{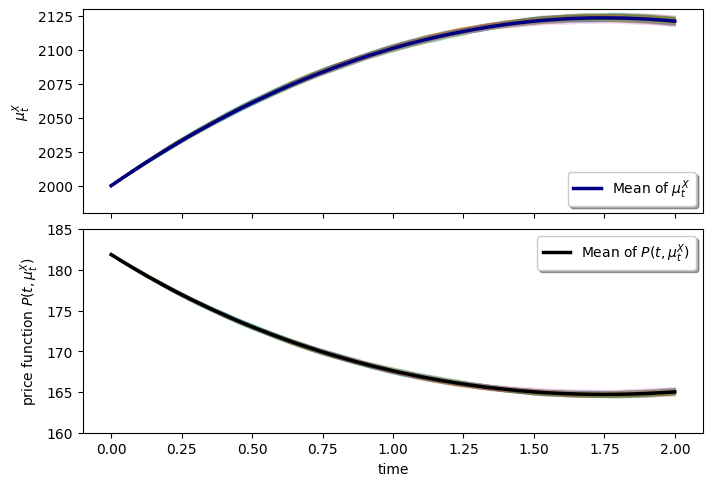}}
\subfloat[$\mu^Y$ and $\hat{\alpha}$]{\includegraphics[width=0.48\textwidth]{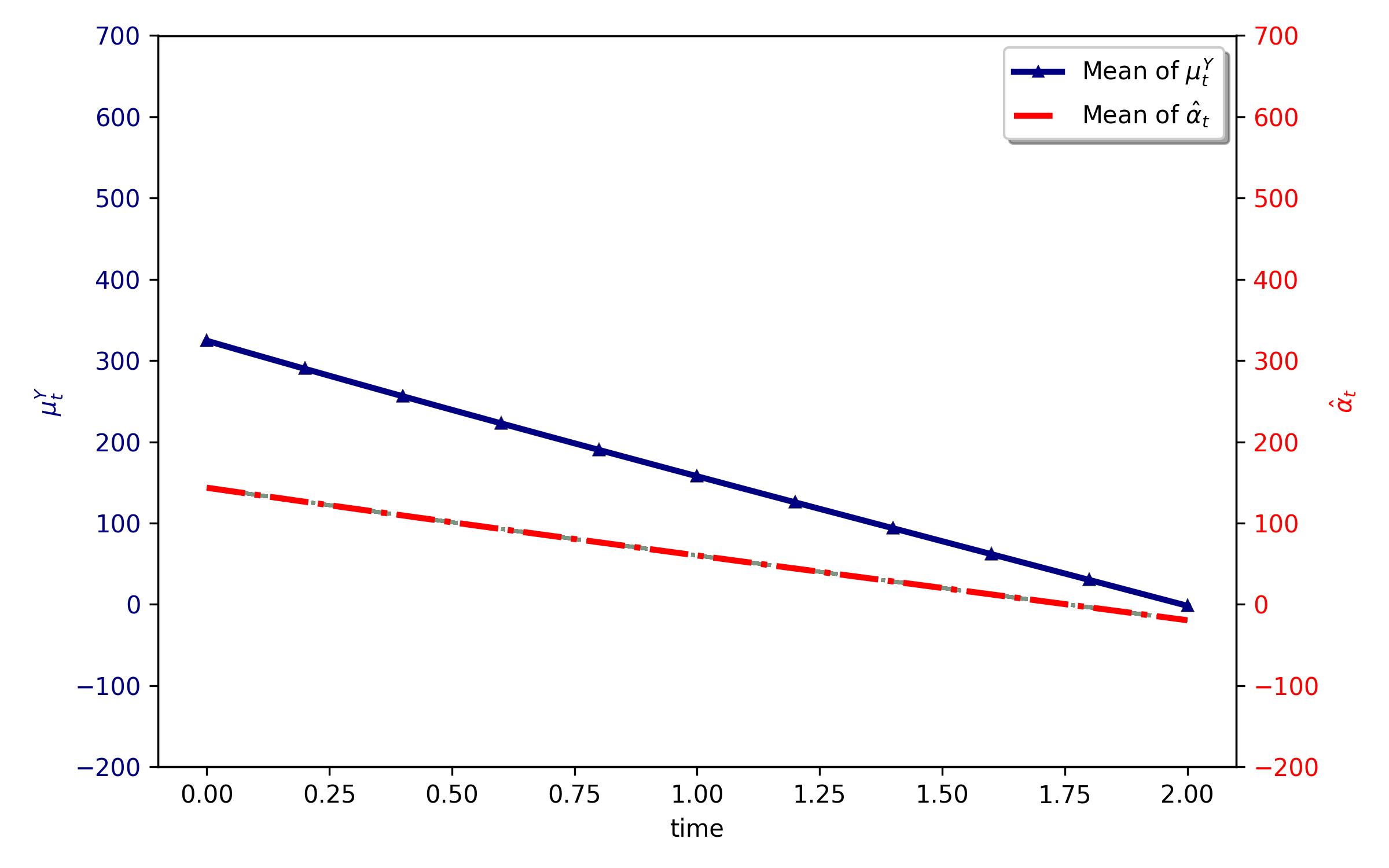}}
\caption{$\mu_0^X = 2000 MWh$, $D = 1500MWh$, $\sigma^0  = 1$, $T=2$, $r=2$, $L_Y\approx0.00204$.}
\label{Exam1201}
\end{figure}

\begin{remark}
(i) The relatively smooth appearance of the sample path, such as $\mu^X$ in Figures \ref{Exam002} (a), \ref{Exam32} (a) and $\hat{\alpha}$ in Figures \ref{Exam01} (b), \ref{Exam11} (b), is due to the fact that the magnitude of its fluctuations, as illustrated in Figure \ref{CMOFANDB} (a), is small compared with the overall magnitude of $\mu^X$ and $\hat{\alpha}$. Although the scale of the $y$-axis can be adjusted to make these fluctuations more visible (see Figures \ref{Exam31} (a) and \ref{Exam1201} (a) and (b)), we use a relatively similar scale across the figures to facilitate comparison. 

(ii) Note that for Figures \ref{Exam02}, \ref{Exam04}, \ref{Exam002} and \ref{Exam14}, the solution for $\mu^Y$ is explicit, thanks to the price function remaining at its cap price. See Figure \ref{CMOFANDB} (b). The same observation applies to Figures \ref{Exam12}, \ref{Exam128}, \ref{Exam001} and  \ref{Exam007}.

(iii) Each of the $y_{\theta_y}$ and $z_{\theta_z}$-network is a 3-hidden layer fully connected neural network with 100 neurons per hidden layer, ReLU activation function in all hidden layers and a linear output layer of dimension $1$. The input dimension is 2, corresponding to the state variable $\mu^X$ and the time variable. We report the loss function in every figure after 1000 iterations unless specified otherwise.

\begin{figure}[H]
\centering
\subfloat[Fluctuation Magnitude of $\mu^X$]{\includegraphics[width=0.5\textwidth]{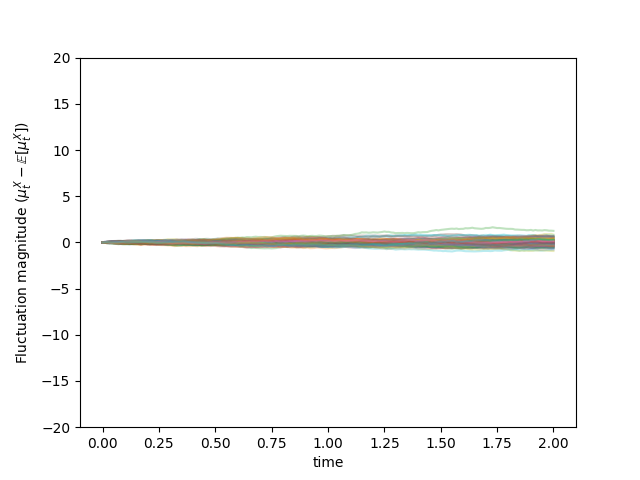}}
\subfloat[Explicit solutions for $\mu_Y$]{\includegraphics[width=0.5\textwidth]{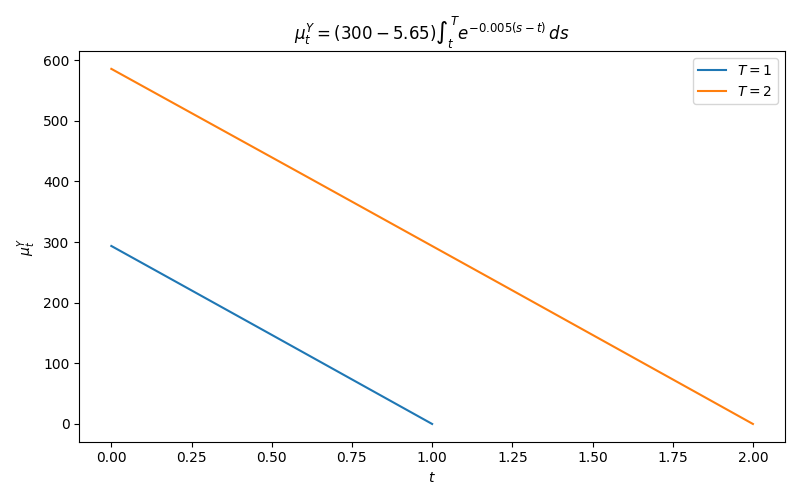}}
\caption{}
\label{CMOFANDB}
\end{figure}
\end{remark}

\section{A Stackelberg Mean Field Game} 
\subsection{The Model} \label{ASMFG}
In this section, we consider the social planner's problem, which consists of optimising the installation subsidies in order to balance the energy generation capacity $\mu_t^X$ and the public demand baseline at time $t$. The demand process $D_t$ evolves according to
\begin{align}
dD_t &= a (\tilde{B}(t) - D_t) dt, \qquad D_0 > 0,\label{Demand}\\
\tilde{B}(t) &= B_0 + B_1 \cos(2\pi t - B_2) - 2\pi\sin(2\pi t - B_2),
\end{align}
where $a$ and $B_0, B_1, B_2$ are strictly positive constants; this captures seasonal effects as in \cite{ACL2013}. 

In what follows, the social planner subsidises installation and is indifferent as to whether the subsidy is paid at construction or during operation of the power plant, see \cite{CLP2016}; additionally, the producer is aware of the subsidy level before installing new energy capacity. Let $S\geq 0$ be a constant and let $v\in\mathcal{H}_{\mathbb{F}^0}^2([0, T];[-S, S])$ denote installation subsidies (per unit of new  installed capacity) provided by the social planner. The upper bound of the subsidies prevents producers from earning free money by installing an infinite amount of new capacity, see \cite[Page 706]{ABBC2023}, while the negative lower bound represents taxation designed to discourage the development of new capacities. 

Given a  subsidy policy $v$ announced by the government, the producer maximises the modified running profit
\begin{align}
f(t, x,  \mu, \alpha; v) &=   x \left( P(t, \mu) - c_p\right) - (c_I - v) \alpha - c_a \alpha^2. \label{ASGRC}
\end{align}
Let $(\hat{\alpha}, {\mu}^X)$ be the unique MFE characterised in the previous sections corresponding to \eqref{ASGRC}, which depends on the social planner's controls $v$. That is, for a given admissible control $v$, we know from Section \ref{SMP} that the producer's optimal control is 
\begin{align}
\hat{\alpha}_t = \frac{1}{2c_a}\left( \hat{Y}_t - c_I + v_t \right),\label{TCOPIS}
\end{align}
and the associated FBSDE system becomes
\begin{align}
d\hat{X}_t & = \left(-\delta \hat{X}_t +  \frac{1}{2c_a}\left( \hat{Y}_t - c_I + v_t \right) \right) dt + \sigma dW_t + \sigma^0 dW_t^0,\qquad\,\,\, \hat{X}_0=\xi_0, \label{FSDES3}\\
d\hat{Y}_t & = \left( \delta \hat{Y}_t + c_p - P(t, \mu_t^X) \right)dt  + \hat{Z}_t dW_t + \hat{Z}_t^0dW_t^0,\qquad\qquad\qquad \hat{Y}_T=0.\label{BSDES3}
\end{align}

\begin{remark}
In this model, the social planner does not directly subsidise production costs such that no link between prices and production costs is severed, which could otherwise result in an inefficient allocation of resources in a competitive market. The taxation is to prevent the overproduction of the subsidised good caused by subsidies, since production and consumption might be expanded beyond the point where the marginal social benefit of consuming the good is equal to or falls below the marginal social costs of production. See \cite{SC1999}.
\end{remark}

The social planner anticipates the optimal control applied by the producer in response to her control $v$ and aims to minimise the objective functional by solving 
\begin{align}
&\inf_{v \in  \mathcal{H}_{\mathbb{F}^0}^2([0, T]; [-S, S])} \underbrace{ \mathbb{E} \left[ \int_0^T \left(\lambda_d(D_t - \mu_t^X)^2 +  \hat{\alpha}_t v_t  \right) dt  \right]}_{J(0, \mu_0^X; v)}, \label{SPOCP}
\end{align}
where $\lambda_d>0$ converts the generation marginal capacity into capital loss, $\hat{\alpha}_t v_t$ is the cost of subsidising the total amount of new energy capacity installation; then, by inserting \eqref{TCOPIS} into \eqref{SPOCP}, 
\begin{align}
\inf_{v \in \mathcal{H}_{\mathbb{F}^0}^2([0, T]; [-S, S])} \mathbb{E} \left[ \int_0^T \left(\lambda_d (D_t - \mu_t^X)^2 +    \frac{v_t }{2c_a}\left({\mu}_t^Y - c_I + v_t \right)  \right) dt  \right], \label{OPre}
\end{align}
where 
\begin{align}
d\mu_t^X & =  \left(-\delta\mu_t^X +  \frac{1}{2c_a}\left( \mu_t^Y- c_I+ v_t \right) \right) dt  + \sigma^0 dW_t^0, \qquad\,\,\,\,\,\,\,\mu_0^X=\xi_0, \label{FSDEs3c}\\
d\mu_t^Y & = \left( \delta \mu_t^Y + c_p - P(t, \mu_t^X) \right) dt  + \hat{Z}_t^0dW_t^0, \qquad\qquad\qquad \mu_T^Y=0. \label{BSDEs3c}
\end{align}

\begin{remark}
In \eqref{SPOCP}, given that the state processes $(\mu^X, \mu^Y, \hat{Z}^0)$ are $\mathbb{F}^0$-progressively measurable, it is only natural for us to look for the $\mathbb{F}^0$-progressively measurable control.
\end{remark}

\begin{remark}
Overproducing renewable energy could cause capital loss. For instance: (i) in California, it is reported that there is a steady growth in renewable curtailment; (ii) in North Germany, huge amounts of wind energy are produced but most demand is in the south, it is often necessary for the state to curtail wind turbines and ask neighbouring countries to absorb excess power which has created international tensions; (iii) in South Australia, there is extremely high rooftop solar penetration and uncontrolled solar power makes the grid fragile without sufficient dispatchable resources and inertia.
\end{remark}

For notational convenience we set, throughout the remainder of this section,
\begin{align*}
&g(t, x, y, v) = \lambda_d (D_t - x)^2 + \frac{1}{2c_a} ( y v  - c_I v + v^2),  \\
& l(t, x, y, v) = -\delta x + \frac{1}{2c_a} \left( y - c_I + v \right), \\
& h(t, x, y) = \delta y +  c_p - P(t, x) .
\end{align*}
Admittedly, the stochastic maximum principle is not applicable, as the joint convexity does not hold in our case, see \cite{OS2010}. In the next section, we therefore derive the so-called \textit{extended HJB system} to characterise the solution for problem \eqref{OPre} in the sense that we first use decoupling field to characterise $\mu^Y$ as a function of $\mu^X$, that is, $\mu_t^Y = \phi(t, \mu_t^X)$ and that by the four step scheme introduced by \cite{MPY1994} and \cite{PT1999}, $\phi(t, x)$ satisfies quasilinear PDE; after which, we invoke DPP to derive standard HJB equation for the value function of problem \eqref{OPre}. 

\subsection{The Extended HJB System} \label{TEHJBS}

We begin by restricting ourselves to the following admissible Markov control set, see also \cite[Page 8]{DL2024} for similar treatment: 
\begin{align*}
\mathbb{V} & = \{ v \in \mathcal{C}^0([0, T]\times\mathbb{R}; [-S, S]): |v(s, x) - v(t, y)|\leq L_c(|x-y| + |s-t|^q)\\
&\qquad \text{ for $s, t\in[0, T]$, $x, y\in\mathbb{R}$, where $q>0$, $L_c>0$ are fixed constants}.\}
\end{align*}
where $\mathcal{C}^0([0, T]\times \mathbb{R}; [-S, S])$ denotes the space of continuous functions from $[0, T]\times \mathbb{R}$ into $[-S, S]$. More precisely, hereafter, instead of problem \eqref{OPre}, we only investigate the following problem:
\begin{align}
\inf_{v \in \mathbb{V}} \underbrace{\mathbb{E} \left[ \int_0^T \left(\lambda_d (D_t - \mu_t^X)^2 +    \frac{v(t, \mu_t^X) }{2c_a}\left({\mu}_t^Y - c_I + v(t, \mu_t^X) \right)  \right) dt  \right]}_{J(0, \mu_0^X; v)} \label{OPree}.
\end{align}
The reason is that, in order to apply DPP, we must ensure the decoupling field exists uniquely for the FBSDE and for all the feedback control laws in $\mathbb{V}$, such decoupling field is warranted.

\begin{Cor}\label{MARKOVBSDE}
For $v\in\mathbb{V}$, we know that the FBSDEs \eqref{FSDEs3c}-\eqref{BSDEs3c} admit a unique solution $(\mu^X, \mu^Y, \hat{Z}^0) \in  \mathcal{L}_{\mathbb{F}^0}^2([0, T]; \mathbb{R})\times\mathcal{L}_{\mathbb{F}^0}^2([0, T]; \mathbb{R})\times\mathcal{H}_{\mathbb{F}^0}^2([0, T]; \mathbb{R})$ and it holds for all $t\in[0, T]$ that 
\begin{align}
\mu_t^Y = \phi(t, \mu_t^X).
\end{align}
where the bounded function $\phi$ is Lipschitz continuous in $x$, uniformly in $t$; and it is a unique viscosity solution of the following backward quasilinear second-order parabolic PDE:
\begin{align}
&\frac{\partial \phi}{\partial t}(t, x) + \frac{\partial \phi}{\partial x}(t, x) \left( -\delta x + \frac{1}{2c_a} \phi(t, x) - \frac{c_I}{2c_a} + \frac{1}{2c_a} v(t, x) \right)  \nonumber\\ 
& \qquad+ \frac{(\sigma^0)^2}{2} \frac{\partial^2 \phi}{\partial x^2}(t, x) -\delta \phi(t, x) - c_p  + P(t, x) = 0, \qquad  \phi(T, x) = 0. \label{PDEFMUY}
\end{align}
\end{Cor}
\begin{proof}
The boundedness of function $\phi$ follows from the fact that
\begin{align*}
\underbrace{- c_p \int_t^T e^{-\delta(s-t)} ds}_{\text{$y_l$}}& \leq {\mu}_t^Y = \mathbb{E}\left[\int_t^T e^{-\delta(s-t)} \left( P(s, \mu_s^{X}) - c_p \right) ds \bigg{|}\mathcal{F}_t^0 \right] \leq \underbrace{ \left( P_{cap} - c_p\right)  \int_t^T e^{-\delta(s-t)} ds}_{\text{$y_u$}},
\end{align*}
where $P_{cap}$ is the capped spot price. The rest of the conclusion is immediate from \cite[Theorem 5.1 and Corollary 4.1]{PT1999}, \cite[Lemma 3.2]{LW2014}.
\end{proof}

Our next result establishes the existence of optimal control for problem \eqref{OPree} by following \cite{BGM2011}.
\begin{Lemma}\label{EOOPC}
There exists $\hat{v}\in\mathbb{V}$ such that $J(0, \mu_0^X; \hat{v}) =  \inf\limits_{v\in\mathbb{V}} J(0, \mu_0^X; v)$.
\end{Lemma}
\begin{proof}
Let $(v_n)_{n\geq 0} \in \mathbb{V}$ be a minimising sequence such that
\begin{align*}
\lim_{n\to\infty} J(0, \mu_0^X; v_n) = \inf_{v\in\mathbb{V}} J(0, \mu_0^X; v).
\end{align*}
Let $(\mu^{X, n}, \mu^{Y, n}, Z^{0, n})$ be the solution of the following FBSDE associated with control $v_n$,
\begin{align}
d\mu_r^{X, n} &= l\left(r, \mu_r^{X, n}, \mu_r^{Y, n}, v_n(r, \mu_r^{X, n})\right) dr + \sigma^0 dW_r^0, \qquad \mu_0^{X, n} = \xi_0, \label{xn1}\\
d\mu_r^{Y, n} &=  h\left(r, \mu_r^{X, n}, \mu_r^{Y, n}\right) dr + Z_r^{0, n} dW_r^0, \qquad \mu_T^{Y, n} = 0,\label{xn2}
\end{align}
after which, it follows that (i) by \cite[Lemma 3.3]{BGM2011}, there exists a positive constant $C$ such that
\begin{align*}
\sup_{n\geq 0} \mathbb{E} \left[ \sup_{0\leq s\leq T} | \mu_s^{X, n} |^2 +   \sup_{0\leq s\leq T} | \mu_s^{Y, n} |^2 + \int_0^T |Z_s^{0, n}|^2 ds  \right] < C;
\end{align*}
and (ii) by \cite[Lemma 3.4 and Lemma 3.6]{BGM2011} and Corollary \ref{MARKOVBSDE}, the sequence $(\mu^{X, n}, \mu^{Y, n}, \int_0^\cdot Z_s^{0, n}dW_s^0, W^0) $ is tight on the space $C([0, T], \mathbb{R}) \times C([0, T], \mathbb{R}) \times D([0, T], \mathbb{R}) \times C([0, T], \mathbb{R})$ endowed with the topology of uniform convergence on the first, second and fourth factors and $S$-topology on the third factor respectively. Then, by the definition of set $\mathbb{V}$. we know that $(v_n)_{n\geq0}$ is uniformly bounded and equicontinuous, after which, by the Arzel\`a--Ascoli theorem,  we know that (iii) the control sequence $v_n(t, \mu_t^{X, n})$ is tight on the space $C([0, T], \mathbb{R})$ endowed with the topology of uniform convergence.

Statements (i), (ii) and (iii) show that the sequence of processes
\begin{align*}
\gamma_n = (\mu^{X, n}, \mu^{Y, n}, M^n, v_n, W^0)
\end{align*}
is tight on the space $\Gamma = C([0, T], \mathbb{R}) \times C([0, T], \mathbb{R}) \times D([0, T], \mathbb{R}) \times C([0, T], \mathbb{R})\times C([0, T], \mathbb{R})$ equipped with the product topology of the uniform convergence on the first, second, fourth and fifth factors and the $\mathcal{S}$-topology on the third factor (see \cite[section 2]{BGM2011} for its definition); where $M_t^n:= \int_0^t Z_s^{0, n}dW_s^0$. Then, by \cite[Page 346]{BGM2011}, there exists a probability space $(\tilde{\Omega}, \tilde{\mathcal{F}}^0, \tilde{\mathbb{P}})$, a sequence $\tilde{\gamma}^n = (\tilde{\mu}^{X, n}, \tilde{\mu}^{Y, n}, \tilde{M}^n, \tilde{v}_n, \tilde{W}^{0, n})$ and $\tilde{\gamma} = (\tilde{\mu}^{X}, \tilde{\mu}^{Y}, \tilde{M}, \tilde{v}, \tilde{W}^0)$ defined on this probability space such that:

(iv) for each $n\geq 0$, $\gamma^n$ and $\tilde{\gamma}^n$ are identically distributed; 

(v) there exists a subsequence of $\tilde{\gamma}^n$ that converges to $\tilde{\gamma}$, $\tilde{\mathbb{P}}$-almost surely on the space $\Gamma$; 

(vi) $v_n$ converges to $\tilde{v}$, $\tilde{\mathbb{P}}-$a.s.;

(vii) $\sup\limits_{0\leq t\leq T} |\tilde{\mu}_t^{X, n} - \tilde{\mu}_t^{X}| \to 0$, $\sup\limits_{0\leq t\leq T} |\tilde{\mu}_t^{Y, n} - \tilde{\mu}_t^{Y}| \to 0$, $\tilde{\mathbb{P}}$-a.s as $n\to\infty$.

By property (iv), we have
\begin{align*}
d\tilde{\mu}_r^{X, n} &= l\left(r, \tilde{\mu}_r^{X, n}, \tilde{\mu}_r^{Y, n}, \tilde{v}_n(r, \tilde{\mu}_r^{X, n})\right) dr + \sigma^0 d\tilde{W}_r^{0, n}, \qquad \tilde{\mu}_0^{X, n} = \xi_0,\\
d\tilde{\mu}_r^{Y, n} &=  h\left(r, \tilde{\mu}_r^{X, n}, \tilde{\mu}_r^{Y, n}\right) dr + \tilde{Z}_r^{0, n} d\tilde{W}_r^{0, n}, \qquad\qquad\qquad\,\, \tilde{\mu}_T^{Y, n} = 0,
\end{align*}
of which, by property (v), (vi) and (vii), let $n\to\infty$ such that
\begin{align*}
d\tilde{\mu}_r^{X} &= l\left(r, \tilde{\mu}_r^{X}, \tilde{\mu}_r^{Y}, \tilde{v}(r, \tilde{\mu}_r^{X})\right) dr + \sigma^0 d\tilde{W}_r^{0}, \qquad \tilde{\mu}_0^{X} = \xi_0,\\
d\tilde{\mu}_r^{Y} &=  h\left(r, \tilde{\mu}_r^{X}, \tilde{\mu}_r^{Y}\right) dr + \tilde{Z}_r^{0} d\tilde{W}_r^{0}, \qquad\qquad\qquad \tilde{\mu}_T^{Y} = 0.
\end{align*}
Now we show that the following limit holds in probability:
\begin{align*}
\lim_{n\to\infty} \int_0^T g\left( r, \tilde{\mu}_r^{X, n},  \tilde{\mu}_r^{Y, n}, \tilde{v}_n(r, \tilde{\mu}_r^{X, n}) \right)dr = \int_0^T g\left(r,  \tilde{\mu}_r^{X}, \tilde{\mu}_r^{Y}, \tilde{v}(r, \tilde{\mu}_r^{X}) \right)dr.
\end{align*}
By Markov's inequality, let $\epsilon>0$,
\begin{align*}
&\tilde{\mathbb{P}}\left(  \bigg{|} \int_0^T  g\left( r, \tilde{\mu}_r^{X, n},  \tilde{\mu}_r^{Y, n}, \tilde{v}_n(r, \tilde{\mu}_r^{X, n}) \right) - g\left(r,  \tilde{\mu}_r^{X}, \tilde{\mu}_r^{Y}, \tilde{v}(r, \tilde{\mu}_r^{X}) \right)  dr   \bigg{|} > \epsilon \right) \\
& \leq \frac{1}{\epsilon} \tilde{\mathbb{E}} \left[  \int_0^T  \bigg{|}  g\left( r, \tilde{\mu}_r^{X, n},  \tilde{\mu}_r^{Y, n}, \tilde{v}_n(r, \tilde{\mu}_r^{X, n}) \right) - g\left(r,  \tilde{\mu}_r^{X}, \tilde{\mu}_r^{Y}, \tilde{v}(r, \tilde{\mu}_r^{X}) \right)   \bigg{|}  dr   \right] \\
&\leq \frac{C}{\epsilon} \tilde{\mathbb{E}}\left[ \int_0^T \left( | \tilde{\mu}_r^{X, n} -  \tilde{\mu}_r^{X}| + | (\tilde{\mu}_r^{X, n})^2 -  (\tilde{\mu}_r^{X})^2| + |\tilde{\mu}_r^{Y, n} \tilde{v}_n(r, \tilde{\mu}_r^{X, n}) - \tilde{\mu}_r^{Y, n} \tilde{v}(r, \tilde{\mu}_r^{X})| \right.\right.\\
&\qquad \qquad \left. \left. + |\tilde{\mu}_r^{Y, n} \tilde{v}(r, \tilde{\mu}_r^{X}) - \tilde{\mu}_r^{Y} \tilde{v}(r, \tilde{\mu}_r^{X})| + | \tilde{v}_n(r, \tilde{\mu}_r^{X, n}) -  \tilde{v}(r, \tilde{\mu}_r^{X}) || \tilde{v}_n(r, \tilde{\mu}_r^{X, n}) +  \tilde{v}(r, \tilde{\mu}_r^{X}) | \right)dr \right]\\
&\leq \frac{C}{\epsilon} \tilde{\mathbb{E}}\left[ \int_0^T \left( | \tilde{\mu}_r^{X, n} -  \tilde{\mu}_r^{X}| + | (\tilde{\mu}_r^{X, n})^2 -  (\tilde{\mu}_r^{X})^2| + |\tilde{\mu}_r^{Y, n} - \tilde{\mu}_r^{Y}|  \right) dr \right]
\end{align*}
where the last inequality is from the boundedness of $\tilde{\mu}^{Y, n}, \tilde{\mu}^{Y}, \tilde{v}_n, \tilde{v}$ and the Lipschitz property of $\tilde{v}_n, \tilde{v}$; after which, by the dominated convergence theorem and property (vii), it follows that
\begin{align*}
\lim_{n\to\infty} \tilde{\mathbb{E}}\left[ \int_0^T \left( | \tilde{\mu}_r^{X, n} -  \tilde{\mu}_r^{X}| + | (\tilde{\mu}_r^{X, n})^2 -  (\tilde{\mu}_r^{X})^2| + |\tilde{\mu}_r^{Y, n} - \tilde{\mu}_r^{Y}|  \right) dr \right] = 0.
\end{align*}
and that
\begin{align*}
 \inf_{v\in\mathbb{V}} J(v) = \lim_{n\to\infty} J(0, \mu_0^{X}; \tilde{v}_n) = J(0, \mu_0^{X}; \tilde{v}) 
\end{align*}
which concludes the proof of existence of an optimal Markovian control.
\end{proof}

Let $(\mu_s^{X, (t,x)})_{s\in[t, T]}$ denote a strong solution to SDE \eqref{FSDEs3c} starting from $x$ at $s = t$. We can now define the value function:
\begin{definition}
The value function of problem \eqref{OPree} is defined as:
\begin{align}
V(t, x) &= \inf_{v\in\mathbb{V}} J(t, x; v),\,\,\, \forall(t, x)\in[0, T)\times\mathbb{R}, \qquad V(T, x) = 0, \,\,\, \forall x\in\mathbb{R}. \label{VF}
\end{align}
\end{definition}

By defining the operator as $\mathcal{L}^v:=\frac{(\sigma^0)^2}{2}\frac{\partial^2}{\partial x^2} + l(t, x, \phi(t, x), v(t, x)) \frac{\partial}{\partial x}$, we then can characterise the value function as follows: 

\begin{theorem}
The value function $V$ is a unique viscosity solution of the following HJB equation:
\begin{align}
&- \frac{\partial V}{\partial t}(t, x)  - \inf_{v\in\mathbb{V}} \bigg\{ \mathcal{L}^vV(t, x) + g(t, x, \phi(t, x), v(t, x)) \bigg\} = 0,  \label{HJB1}
\end{align}
for $(t, x)\in[0, T)\times\mathbb{R}$ and $V(T, x) = 0$ for $x\in\mathbb{R}$.
\end{theorem}
\begin{proof}
From \cite[Theorem 3.3, Chapter 4]{YZ1999}, it follows that the value function satisfies the dynamic programming principle: for $\theta\in[t, T]$,
\begin{align}
V(t, x) &= \inf_{v \in \mathbb{V}} \mathbb{E} \left[ \int_t^\theta g\left(s, \mu_s^{X, (t, x)}, \phi(s, \mu_s^{X, (t, x)}), v(s, \mu_s^{X, (t, x)})\right) ds + V(\theta, \mu_{\theta}^{X, (t, x)}) \right].
\end{align}
Then, starting from the dynamic programming principle, due to the same reasoning as in \cite[Pages 64-68]{P2009}, it holds true that the value function is a viscosity solution of the Hamilton-Jacobi-Bellman equation. The uniqueness of the viscosity solution is then established by appealing to the strong comparison theorem in \cite[Lemma 4.4.6, Page 80]{P2009}. 
\end{proof}

\begin{remark}
Equation \eqref{HJB1} is solved with $\phi$ held fixed, while $\phi$ solves \eqref{PDEFMUY} under optimal controls $\hat{v}$, which is determined by the infimum part of equation \eqref{HJB1}. The extended HJB system can be interpreted as a game between players indexed by time, see \cite[Page 28]{BM2010}: for each point $t$ in time, we have a player $t$ choosing $\hat{v}_t$ to minimise $\mathcal{L}^vV(t, x) + g(t, x, \phi(t, x), v(t, x))$; player $t$ can, however, only affect the dynamic of the process $\mu_t^X$ by choosing the control $\hat{v}_t$ exactly at time $t$, this is solved by equation \eqref{HJB1}; at another time $s \in(t, T]$, the control will be chosen by player $s$, and if both players on the half-open interval $(t, T]$ use the same control $\hat{v}$, that is, equation \eqref{PDEFMUY} is solved under control $\hat{v}$, then it is optimal for player $t$ to use $\hat{v}$.
\end{remark}

We summarise this section with the extended HJB system for $(t, x)\in[0, T]\times\mathbb{R}$:
\begin{align}
\begin{cases}
\frac{\partial V}{\partial t}(t, x) + \inf\limits_{v\in\mathbb{V}}\{ \mathcal{L}^{v} V(t, x) + g(t, x,\phi(t, x), v(t, x))\} = 0, & V(T, x) = 0,\\
\frac{\partial \phi}{\partial t}(t, x) + \mathcal{L}^{\hat{v}} \phi(t, x) - h(t, x, \phi(t, x)) = 0, & \phi(T, x) = 0,
\end{cases} \label{EXHJBS}
\end{align}
where $\hat{v}$ is defined by
\begin{align}
\hat{v}(t, x) \in \argmin_{v\in\mathbb{V}}\{ \mathcal{L}^{v} V(t, x) + g(t, x,\phi(t, x), v(t, x)) \}.
\end{align}

\subsection{Numerical Approximation of the Extended HJB System} \label{NAOEHJBS1}

From the previous section and \cite{HAW2018}, it follows that the unique viscosity solution to \eqref{EXHJBS} admits the following representation:
\begin{align}
V(t, \mu_t^{X}) &= \int_t^T g\left(s, \mu_s^{X, (t, x)}, \phi(s, \mu_s^{X, (t, x)}), \hat{v}(s, \mu_s^{X, (t, x)})\right) ds  - \int_t^T Z_s^{V} dW_s^0,\label{EXHJBO}\\
\phi(t, \mu_t^X) &= - \int_t^T h\left(s, \mu_s^{X, (t, x)}, \phi(s, \mu_s^{X, (t, x)}), \hat{v}(s, \mu_s^{X, (t, x)})\right) ds  - \int_t^T Z_s^{\phi} dW_s^0, \label{EXHJBT}
\end{align}
and $Z_s^{V} \approx \sigma^0 \frac{\partial V}{\partial x}(s, \mu_s^{X, (t, x)})$ and $\hat{v}(s, \mu_s^{X, (t, x)})$ is given by
\begin{align*}
\argmin_{v\in\mathbb{V}} \bigg\{  \frac{Z_s^V l(s, \mu_s^{X, (t, x)}, \phi(s, \mu_s^{X, (t, x)}), v(s, \mu_s^{X, (t, x)})}{\sigma^0} + g(s, \mu_s^{X, (t, x)}, \phi(s, \mu_s^{X, (t, x)}), v(s, \mu_s^{X, (t, x)})) \bigg\}.
\end{align*}
To apply the first-order condition and verify the existence of $Z^V$, we establish the differentiability of the value function.

\begin{Lemma} \label{VICD}
For $(t, x)\in[0, T]\times\mathbb{R}$, the value function is continuously differentiable with respect to $x$ and its partial derivative admits the following representation:
\begin{align}
\frac{\partial}{\partial x} V(t, x) = \frac{\partial}{\partial x} J(t, x; \hat{v}), \label{DVOVF}
\end{align}
where the derivative is taken with respect to state $x$ (initial data) and $\hat{v}\in\mathbb{V}$ is an optimal control. 
\end{Lemma}
\begin{proof}
We start with the case for $(t, x)\in[0, T]\times K$, where $K$ is any compact subset of $\mathbb{R}$.

(i) The process $(\mu_s^{X, (t, x)})_{s\in[t, T]}$ is the unique solution to the following SDE:
\begin{align*}
d\mu_s^{X, (t, x)} &= b(s, \mu_s^{X, (t, x)}) ds + \sigma^0 dW_s^0, \qquad \mu_t^X = x,
\end{align*}
where $b(s, y) = l(s, y, \phi(s, y), v(s, y))$ is Lipschitz continuous. Then, by \cite[Theorem 6.3 in page 42]{YZ1999}, for $p\geq1$, a generic constant $C_T>0$ and $x\in K$,
\begin{align}
\mathbb{E}\left[ \sup_{s\in[t, T]} |\mu_s^{X, (t, x)}|^p \right] \leq C_T (1+ |x|^p). \label{VICD1}
\end{align}
Since $x\in K \subset \mathbb{R}$, the right-hand side of \eqref{VICD1} is uniformly bounded by a constant $C_{T, K, p}$ depending on $T, K$ and $p$.

(ii) From \cite[Proposition 2.4]{BD2017}, as $b(t, x)$ fulfils the Lipschitz and linear growth conditions, $\mu_s^{X, (t, x)}$ is Malliavin differentiable w.r.t its initial condition for all $s\in[t, T]$, and for all $t\leq r \leq s\leq T$ we have
\begin{align}
D_r \mu_s^{X, (t, x)} = \exp\bigg\{ \int_r^s b'(u, \mu_u^{X, (t, x)}) du \bigg\}, \label{VICD11}
\end{align}
where $b'(s, y)$ denotes the (weak) derivative of $b$ with respect to $y$; and 
\begin{align}
\frac{\partial}{\partial x} \mu_s^{X, (t, x)} = \exp\bigg\{ \int_t^s b'(u, \mu_u^{X, (t, x)}) du \bigg\}. \label{VICD12}
\end{align}
As a consequence, $\frac{\partial}{\partial x} \mu_s^{X, (t, x)} = D_r \mu_s^{X, (t, x)} \frac{\partial}{\partial x} \mu_r^{X, (t, x)}$, and all equalities hold $\mathbb{P}$-a.s. Furthermore, by \cite[Corollary A.7]{BD2017}, for any compact subset $K\subset \mathbb{R}$ and $p\geq 1$,
\begin{align}
\sup_{x\in K}\sup_{s\in[t, T]} \mathbb{E}\left[ \left( \frac{\partial}{\partial x} \mu_s^{X, (t, x)}  \right)^p \right] \leq C_{K, p}, \label{VICD13}
\end{align}
for a constant $C_{K. p} >0$ depending on $K$ and $p$.

(iii) From step (i), equation \eqref{VICD1} and the definition of the running cost function, for $(t, x)\in[0, T]\times K$, it follows that
\begin{align}
\mathbb{E}\left[ \sup_{s\in[t, T]} |g\left(s, \mu_s^{X, (t, x)}, \phi(s, \mu_s^{X, (t, x)}), v(s, \mu_s^{X, (t, x)})\right)|^p \right] \leq C_\Phi (1+|x|^{2p}) \leq C_{\Phi, T, K, p}, \label{VICD2}
\end{align}
where $C_{\Phi, T, K, p}>0$ is a constant depending on $T, K$ and $p$ only; then, by Fubini's theorem, the cost functional can be rewritten as
\begin{align}
J(t, x; v) &= \mathbb{E} \left[ \int_t^T g\left(s, \mu_s^{X, (t, x)}, \phi(s, \mu_s^{X, (t, x)}), v(s, \mu_s^{X, (t, x)})\right) ds \right],\nonumber\\
&=  \int_t^T  \mathbb{E} [ \underbrace{ g\left(s, \mu_s^{X, (t, x)}, \phi(s, \mu_s^{X, (t, x)}), v(s, \mu_s^{X, (t, x)})\right)}_{\Phi\left(s, \mu_s^{X, (t, x)}\right)}] ds,
\end{align}
and note that
\begin{align}
\mathbb{E}\left[ \Phi\left(s, \mu_s^{X, (t, x)}\right) \right] \leq C_{\Phi, T, K, p}. \label{VICD22}
\end{align}
Now, we proceed as in the proof of \cite[Theorem 3.8]{BD20181} and \cite[Lemma 4.1]{BD2017}:

(iv) Let us define the function $\bar{\gamma}$ and the process $\beta$,
\begin{align*}
\bar{\gamma}(t, x):= \mathbb{E}\left[ \Phi\left(s, \mu_s^{X, (t, x)}\right) \int_t^s \beta_r dW_r^0 \right],\qquad \beta_r:= \frac{a(r)}{\sigma^0} \frac{\partial}{\partial x} \mu_r^{X, (t, x)},
\end{align*}
where $a:[t, T]\mapsto \mathbb{R}$ is a bounded integrable function satisfying $\int_t^s a(r) dr = 1$ and $a(r) = 0$ for $r\in[s, T]$.
Then, by the Cauchy-Schwarz inequality and It\^{o} isometry, it holds that for a constant $C_{\gamma, K, T, p} > 0$
\begin{align*}
|\bar{\gamma}(t, x)| &\leq \mathbb{E}\left[ |\Phi\left(s, \mu_s^{X, (t, x)}\right) \int_t^s \beta_r dW_r^0|\right] \leq C_\gamma \left( \mathbb{E}\left[\int_t^s \beta_r^2 dr\right] \mathbb{E}\left[ \Phi^2(s, \mu_s^{X, (t, x)}) \right] \right)^{\frac{1}{2}} < C_{\gamma, K, T, p},
\end{align*}
where the last inequality follows from \eqref{VICD2} with $C_{\gamma, K, T, p}$ depending on $K, T, p$.

(v) From \cite[Page 16]{BD2017}, we observe that for any infinitely differentiable function $\psi:[t, T]\times\mathbb{R}\mapsto\mathbb{R}$, the function $\Psi(t, x):=\mathbb{E}\left[\psi(s, \mu_s^{X, (t, x)})\right]$ is continuously differentiable on $x$ with derivative
\begin{align*}
\frac{\partial}{\partial x}\Psi(t, x) =  \mathbb{E}\left[ \psi\left(s, \mu_s^{X, (t, x)}\right) \int_t^s \beta_r dW_r^0 \right],
\end{align*}
which is the duality formula for the Malliavin derivative, see \cite[Page 16]{BD2017}.
Then, let function $\bar{\psi}$ be an arbitrary bounded and continuous, in particular, $\bar{\psi}(s, \mu_s^{X, (t, x)})\in\mathcal{L}_{\mathbb{F}^0}^2([t, T]; \mathbb{R})$. We can approximate $\bar{\psi}$ by a sequence of smooth functions $\{\psi_n\}_{n\geq0}$ with compact support such that $\psi_n(s, y)\to\bar{\psi}(s, y)$ as $n\to\infty$. By the exact same argument in \cite[Step 2 of Proof of Theorem 3.8]{BD20181}, it follows that $\bar{\Psi}(t, x):=\mathbb{E}\left[\bar{\psi}(s, \mu_s^{X, (t, x)})\right]$ is continuously differentiable on $x$ with derivative
\begin{align*}
\frac{\partial}{\partial x}\bar{\Psi}(t, x) =  \mathbb{E}\left[ \bar{\psi}\left(s, \mu_s^{X, (t, x)}\right) \int_t^s \beta_r dW_r^0 \right].
\end{align*}
From \cite[Step 3 of Proof of Theorem 3.8]{BD20181}, the above result can be extended to all bounded Borel measurable functions $\tilde{\psi}$. To do so, let us denote $\mathcal{G}:=\{ \bar{\psi}:[t, T]\times\mathbb{R}\mapsto\mathbb{R} \text{ is continuous and bounded} \}$ such that $\mathcal{G}$ is a multiplicative class, i.e. $\bar{\psi}_1, \bar{\psi}_2\in\mathcal{G}$, then $\bar{\psi}_1\bar{\psi}_2\in\mathcal{G}$; in addition, let $\mathcal{H}$ be the class of function $\tilde{\psi}:[t, T]\times\mathbb{R}\mapsto\mathbb{R}$ such that $\tilde{\Psi}(t, x):=\mathbb{E}\left[\tilde{\psi}(s, \mu_s^{X, (t, x)})\right]$ is continuously differentiable and fulfils
\begin{align*}
\frac{\partial}{\partial x}\tilde{\Psi}(t, x) =  \mathbb{E}\left[ \tilde{\psi}\left(s, \mu_s^{X, (t, x)}\right) \int_t^s \beta_r dW_r^0 \right].
\end{align*}
As a result, $\mathcal{G}\subseteq\mathcal{H}$. Then, $\mathcal{H}$ is a monotone vector space on $\mathbb{R}$ (see \cite[Page 7]{PPE2012} for its definition), implying that if $\{\tilde{\psi}_n\}_{n\geq 1}\subset \mathcal{H}$ and $0\leq\tilde{\psi}_1\leq \tilde{\psi}_2\leq\dots\leq \tilde{\psi}_n\leq \dots$ and $\lim\limits_{n\to\infty}\tilde{\psi}_n = \hat{\psi}$ and $\hat{\psi}$ is bounded, then $\hat{\psi}\in\mathcal{H}$. And by the same argument in \cite[Step 3]{BD20181} and Monotone Class Theorem in \cite[Page 7]{PPE2012}, we can conclude that $\mathcal{H}$ contains all bounded Borel measurable function on $\mathbb{R}$.

(vi) Next, as in \cite[Proof of Theorem 3.8]{BD20181}, we can approximate $\Phi$ by a sequence of bounded Borel measurable functions $\{\Phi_n\}_{n\geq0}(s, y) = \Phi(s, y) I\{ |y|\leq n \}$ such that $\Phi_n(s, y)\to\Phi(s, y)$ as $n\to\infty$ for $(s, y)\in[t, T]\times\mathbb{R}$. Let 
\begin{align*}
\gamma_n(t, x)&:= \mathbb{E}\left[ \Phi_n(s, \mu_s^{X, (t, x)}) \right],\qquad \tilde{\gamma}(t, x):=\mathbb{E}\left[ \Phi(s, \mu_s^{X, (t, x)}) \right],
\end{align*}
such that by \eqref{VICD22} and dominated convergence theorem, $\lim\limits_{n\to\infty} \sup\limits_{x\in K } |\gamma_n(t, x) - \tilde{\gamma}(t, x)| = 0$.

(vii) An appeal to step (v) yields
\begin{align*}
\frac{\partial}{\partial x} \gamma_n(t, x) = \mathbb{E}\left[ \Phi_n\left(s, \mu_s^{X, (t, x)}\right) \int_t^s \beta_r dW_r^0  \right],
\end{align*}
after which, we observe that
\begin{align*}
\bigg{|} \frac{\partial}{\partial x} \gamma_n(t, x) - \bar{\gamma}(t, x)  \bigg{|} & = \bigg{|}  \mathbb{E}\left[ \left(  \Phi_n\left(s, \mu_s^{X, (t, x)}\right) -  \Phi\left(s, \mu_s^{X, (t, x)}\right) \right) \int_t^s \beta_r d W_r^0 \right]\bigg{|}\\
&\leq  \mathbb{E}\left[ \bigg{|}  \Phi_n\left(s, \mu_s^{X, (t, x)}\right) -  \Phi\left(s, \mu_s^{X, (t, x)}\right)  \bigg{|}  \bigg{|}   \int_t^s  \beta_r  dW_r^0  \bigg{|}  \right]\\
&\leq \left( \mathbb{E}\left[ \bigg{|}  \Phi_n\left(s, \mu_s^{X, (t, x)}\right) -  \Phi\left(s, \mu_s^{X, (t, x)}\right) \bigg{|}^2  \right] \mathbb{E}\left[ \int_t^s \beta_r^2 dr \right]\right)^{\frac{1}{2}}\\
&\leq C_{K, P}  \mathbb{E}\left[ \bigg{|}  \Phi_n\left(s, \mu_s^{X, (t, x)}\right) -  \Phi\left(s, \mu_s^{X, (t, x)}\right) \bigg{|}^2  \right]^{\frac{1}{2}},
\end{align*}
where we have used the Cauchy-Schwarz inequality, inequality \eqref{VICD13}; then, by  \eqref{VICD22} and dominated convergence theorem, it is clear that
\begin{align*}
\lim_{n\to\infty} \sup_{x\in K}\mathbb{E}\left[ \bigg{|}  \Phi_n\left(s, \mu_s^{X, (t, x)}\right) -  \Phi\left(s, \mu_s^{X, (t, x)}\right) \bigg{|}^2  \right]^{\frac{1}{2}} = 0,
\end{align*}  
which entails that $\lim\limits_{n\to\infty} \sup\limits_{x\in K }\bigg{|} \frac{\partial}{\partial x} \gamma_n(t, x) - \bar{\gamma}(t, x)  \bigg{|} = 0$.

(viii) By combining steps (vi) and (vii), the differentiable limit theorem implies that the mapping $x\mapsto \tilde{\gamma}(t, x)$ is continuously differentiable for $(t, x)\in[0, T]\times K$ and
\begin{align*}
\frac{\partial}{\partial x} \tilde{\gamma}(t, x) = \bar{\gamma}(t, x),
\end{align*}
which is bounded by a constant $C_{\gamma, K, T, p}$ according to step (iv). From step (iii) and the fact that constant $C_{\gamma, K, T, p}$ is independent of $s$, it holds that the mapping $x\mapsto J(t, x; v)$ is continuously differentiable and for optimal control $\hat{v}\in\mathbb{V}$, we know $V(t, x) = J(t, x; \hat{v})$ and thereby establishing the continuous differentiability of the value function for $(t, x)\in[0, T]\times K$.

Finally, from the arbitrariness of compact set $K$, the conclusion follows.
\end{proof}

We can rewrite the system \eqref{EXHJBO} and \eqref{EXHJBT} in the forward manner with an optimisation problem aiming to match the terminal condition, that is,
\begin{align}
V(s, \mu_s^{X, (t, x)}) &= V(t, x) - \int_t^s g\left(r, \mu_r^{X, (t, x)}, \phi(r, \mu_r^{X, (t, x)}), \hat{v}(r, \mu_r^{X, (t, x)})\right) dr + \int_t^s Z_r^{V} dW_r^0, \label{NEXHJB3}\\
\phi(s,  \mu_s^{X, (t, x)}) &= \phi(t, x) + \int_t^s h\left(r, \mu_r^{X, (t, x)}, \phi(r, \mu_r^{X, (t, x)}), \hat{v}(r, \mu_r^{X, (t, x)})\right) dr  + \int_t^s Z_r^{\phi} dW_r^0. \label{NEXHJB5}
\end{align}
together with the loss functions
\begin{align}
L_V = \mathbb{E} \left[ V^2(T, \mu_T^{X, (t, x)}) \right],\qquad  L_\phi = \mathbb{E} \left[\phi^2(T, \mu_T^{X, (t, x)}) \right].
\end{align}

To derive the numerical algorithm to compute $V$, $\phi$, we treat $V(t, x)$, $\phi(t, x)$ and process $Z^V$, $Z^\phi$ as neural networks $V_{\theta^1}(t, x), \phi_{\theta^2}(t, x)$, $Z_{\theta^3}(r, \mu_r^{X, (t, x)})$,  $z_{\theta^4}(r, \mu_r^{X, (t, x)})$ with parameters $\hat{\pmb{\theta}} = (\theta^1, \theta^2, \theta^3, \theta^4)$, see \cite{HJ2016}, \cite{HAW2018} and \cite{DL2024}. We apply temporal discretisation to the forward system \eqref{NEXHJB3} - \eqref{NEXHJB5}. Given a partition of the time interval $[t, T]: t = t_0 < t_1 < \dots < t_M = T$, we consider the Euler-Maruyama scheme as before for $n = 0,\dots, M-1$:
\begin{align*}
{\mu}_{t_{n+1}}^X - {\mu}_{t_{n}}^X  &=  l\left( t_n, {\mu}_{t_n}^X,  \phi(t_n,  {\mu}_{t_{n}}^X), \hat{v} \right) \Delta t + \sigma^0 \Delta W_n^0,
\\
 \phi(t_{n+1},  {\mu}_{t_{n+1}}^X) - \phi(t_n,  {\mu}_{t_{n}}^X)   &=   h\left({t_{n}}, {\mu}_{t_{n}}^X,  \phi(t_n,  {\mu}_{t_{n}}^X) \right) \Delta t + z_{\theta^4}({t_n}, {\mu}_{t_n}^X)  \Delta W_n^0,\\
 V(t_{n+1},  {\mu}_{t_{n+1}}^X) - V(t_n,  {\mu}_{t_{n}}^X)   &=  - g\left({t_{n}}, {\mu}_{t_{n}}^X,  \phi(t_n,  {\mu}_{t_{n}}^X), \hat{v} \right) \Delta t + Z_{\theta^3}({t_n}, {\mu}_{t_n}^X)  \Delta W_n^0,
\end{align*}
where $\Delta t = \frac{T-t}{M}$ and $\Delta W_n^0 = W_{t_{n+1}}^0 - W_{t_n}^0$, $\mu_{t}^X = x$,  $\phi(t, x) = \phi_{\theta^2}(t, x)$ and $V(t, x) = V_{\theta^1}(t, x)$. The minimiser $\hat{v}$ is computed at each $t_n$ as:
\begin{align}
\hat{v} = (-S) \vee \frac{c_I - \frac{1}{\sigma^0} Z_{\theta_3}(t_n, \mu_{t_n}^X) - \phi(t_n, \mu_{t_n}^X)}{2} \wedge S,
\end{align}
by the first order condition and Lemma \ref{VICD}.

\begin{breakablealgorithm}
\caption{Algorithm for Solving the Extended HJB System \eqref{EXHJBS}}
\begin{algorithmic}
\State Initialise four neural networks
$V_{\theta^1}(t,x)$, $\phi_{\theta^2}(t,x)$,
$Z_{\theta^3}(t,x)$, and $z_{\theta^4}(t,x)$
on $[0,T]\times\mathbb{R}$, with initial parameters
$\hat{\pmb{\theta}}_0
=(\theta_0^1,\theta_0^2,\theta_0^3,\theta_0^4)$.
Let $K$ denote the number of training iterations, $B$ the batch size,
and $\pmb{\rho}_k$ the learning rate at iteration $k$.
\State Fix the initial condition $\mu_{t_0} = x$.
\For{$k=0,\ldots,K-1$}
    \State Set $V(t,x)=V_{\theta_k^1}(t,x)$ and
    $\phi(t,x)=\phi_{\theta_k^2}(t,x)$.
    Initialise $L_V^j=L_\phi^j=0$ for $j=1,\ldots,B$.
    \For{$j=1,\ldots,B$}
        \State Set $\mu_{t_0}^{X,j}=x$ and initialise $D_{t_0}$.
        \For{$i=0,\ldots,M-1$}
            \State Set $t_i=t+ i\Delta t$ and sample $\Delta W_i^{0,j}\sim N(0,\Delta t)$.
            \State $D_{t_{i+1}}^j =D_{t_i}+a\bigl(\widetilde{B}(t_i)-D_{t_i}\bigr)\Delta t $
            \State $\mu_{t_{i+1}}^{X,j}=\mu_{t_i}^{X,j}+l\bigl(t_i,\mu_{t_i}^{X,j}, \phi_{\theta_k^2}(t_i,\mu_{t_i}^{X,j}), \hat v(t_i,\mu_{t_i}^{X,j})\bigr)\Delta t +\sigma^0\Delta W_i^{0,j}.$
            \State $\phi_{\theta_k^2}(t_{i+1},\mu_{t_{i+1}}^{X,j})=\phi_{\theta_k^2}(t_i,\mu_{t_i}^{X,j})+h\bigl(t_i,\mu_{t_i}^{X,j}, \phi_{\theta_k^2}(t_i,\mu_{t_i}^{X,j})\bigr)\Delta t +z_{\theta_k^4}(t_i,\mu_{t_i}^{X,j}) \Delta W_i^{0,j}$. 
            \State $V_{\theta_k^1}(t_{i+1},\mu_{t_{i+1}}^{X,j}) =V_{\theta_k^1}(t_i,\mu_{t_i}^{X,j}) -g\bigl(t_i,\mu_{t_i}^{X,j}, \phi_{\theta_k^2}(t_i,\mu_{t_i}^{X,j}), \hat v(t_i,\mu_{t_i}^{X,j})\bigr)\Delta t +Z_{\theta_k^3}(t_i,\mu_{t_i}^{X,j}) \Delta W_i^{0,j}. $
        \EndFor
        \State $L_V^j=\bigl(V_{\theta_k^1}(t_M,\mu_{t_M}^{X,j})\bigr)^2$, $L_\phi^j= \bigl(\phi_{\theta_k^2}(t_M,\mu_{t_M}^{X,j})\bigr)^2$.
    \EndFor
    \State Compute the batch losses: $J_1(\theta_k^1,\theta_k^3)=\frac{1}{B}\sum_{j=1}^{B}L_V^j, J_2(\theta_k^2,\theta_k^4) =\frac{1}{B}\sum_{j=1}^{B}L_\phi^j$.
    \State Compute the gradients of $J_1$ with respect to
    $(\theta^1,\theta^3)$ and of $J_2$ with respect to
    $(\theta^2,\theta^4)$ using backpropagation.
    \State Update the network parameters: $\theta_{k+1}^1=\theta_k^1-\rho_k\nabla_{\theta^1}J_1, \,\, \theta_{k+1}^3=\theta_k^3-\rho_k\nabla_{\theta^3}J_1,\,\, \theta_{k+1}^2=\theta_k^2-\rho_k\nabla_{\theta^2}J_2,\,\, \theta_{k+1}^4=\theta_k^4-\rho_k\nabla_{\theta^4}J_2$.
    \State Set
    $\hat{\pmb{\theta}}_{k+1}
    =(\theta_{k+1}^1,\theta_{k+1}^2,
    \theta_{k+1}^3,\theta_{k+1}^4)$.
\EndFor
\end{algorithmic}
\end{breakablealgorithm}

\begin{exam}[Excess demand at $t = 0$.] \label{ED01}
In this example, we inherit the parameters from Example \ref{SPTExam}. From Figures \ref{Exam04}, \ref{Exam002} and \ref{Exam14}, we know that the market would have resolved the excess demand in the long run; but in a short term, it is not sufficient to address the supply shortage, see Figure \ref{Exam02} and therefore, government subsidy is essential. Figure \ref{Exam22r0.5} shows that, when the government subsidises the installation of generation capacity such that the installation cost is fully covered, more new capacity is installed. The expected profit of producing energy is decreasing over time due to the decay rate $\delta$ and the constant spot price. Then, we see that from SDE \eqref{FSDEs3c}, the generation capacity $\mu^{X, (0, 1000)}$ increases during the period when  the sum of subsidy $\hat{v}$ and expected profit $\mu^Y$ remains higher than the installation costs. Even though the supply shortage has not been fully resolved and the subsidy remains within the budget constraint, the social planner gradually reduces its subsidy. However, a comparison of Figures \ref{Exam22r0.005} and \ref{Exam22r0.5} shows that the level of subsidy provided for installation is affected by $\lambda_d$, the magnitude of the capital loss associated with the imbalance between supply and demand.
\begin{figure}[H]
\begin{adjustwidth}{-3cm}{-1cm}
\centering
\subfloat[$\mu^X$]{\includegraphics[width=0.4\textwidth]{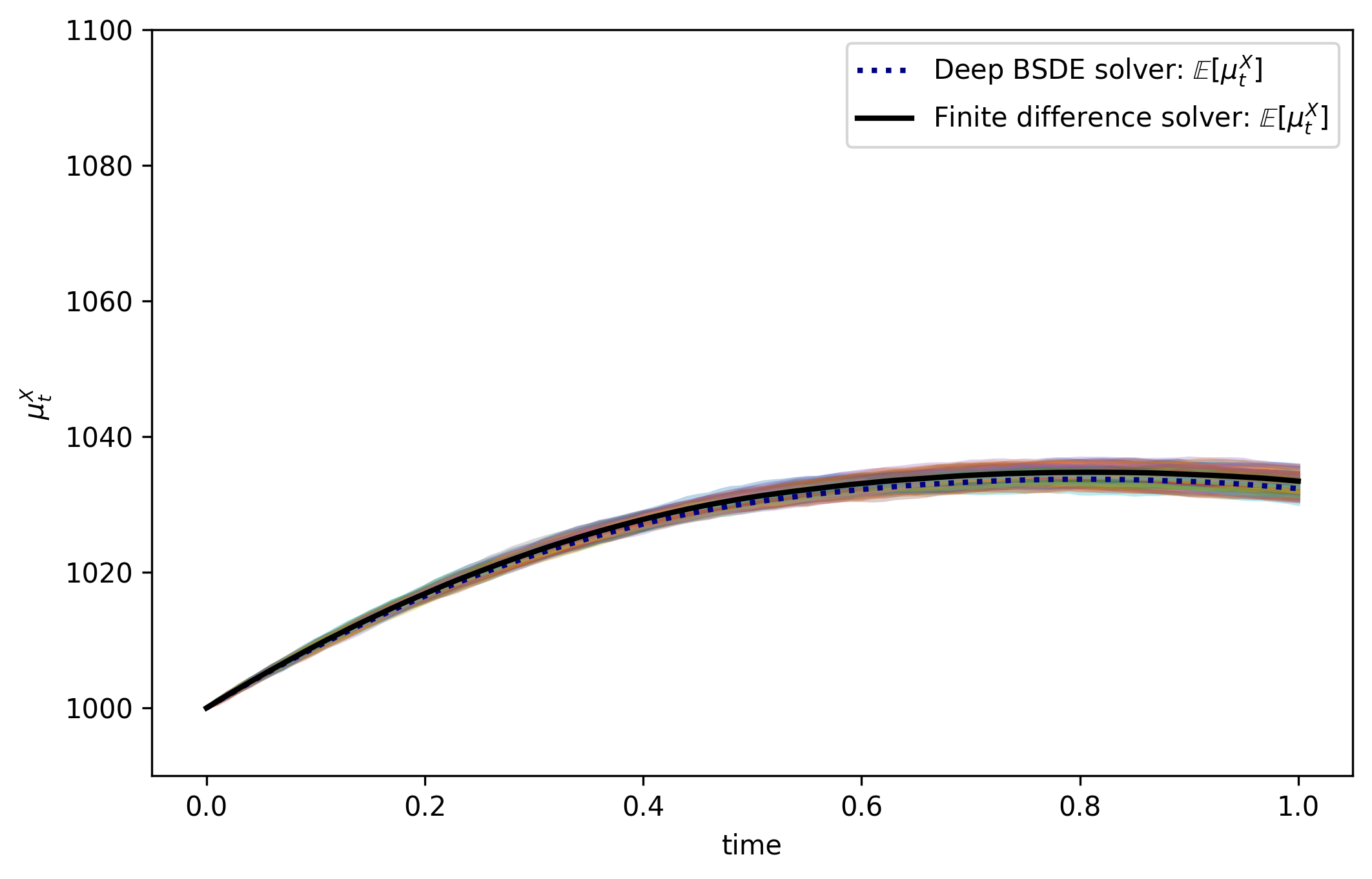}}
\subfloat[$\mu^Y$ and $\hat{\alpha}$]{\includegraphics[width=0.4\textwidth]{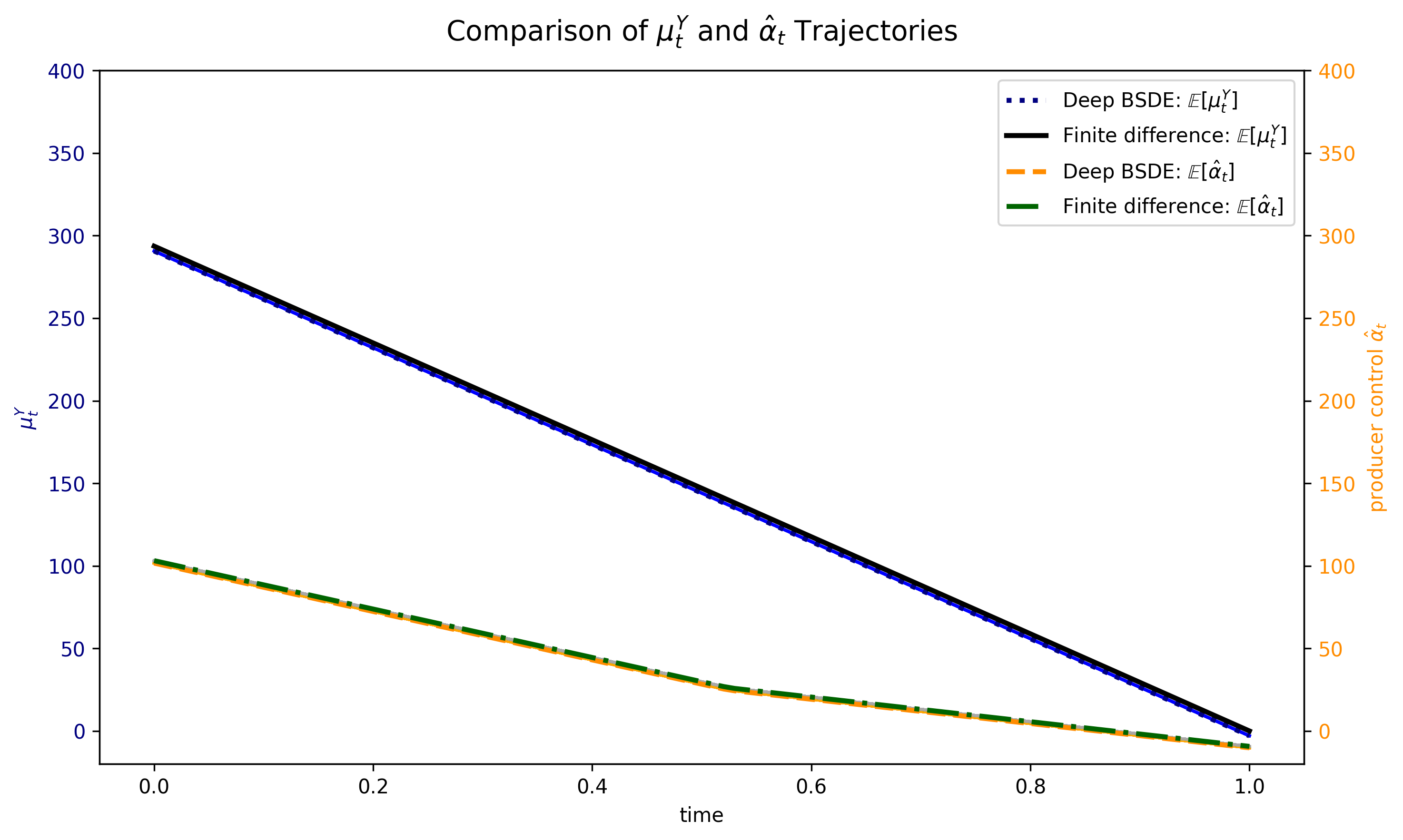}}\\
\subfloat[installation subsidy $\hat{v}$]{\includegraphics[width=0.4\textwidth]{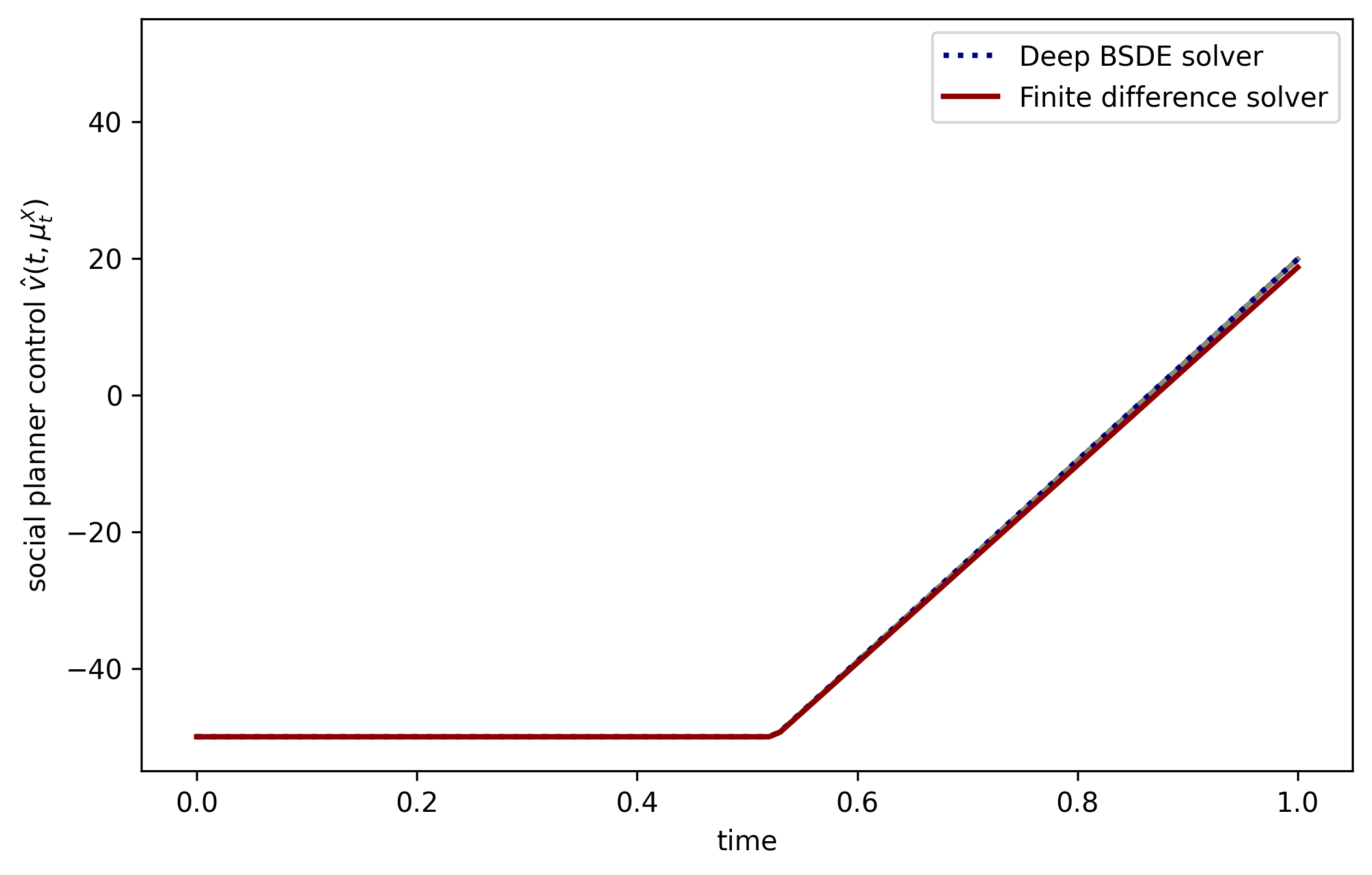}}
\subfloat[Value function $V$]{\includegraphics[width=0.4\textwidth]{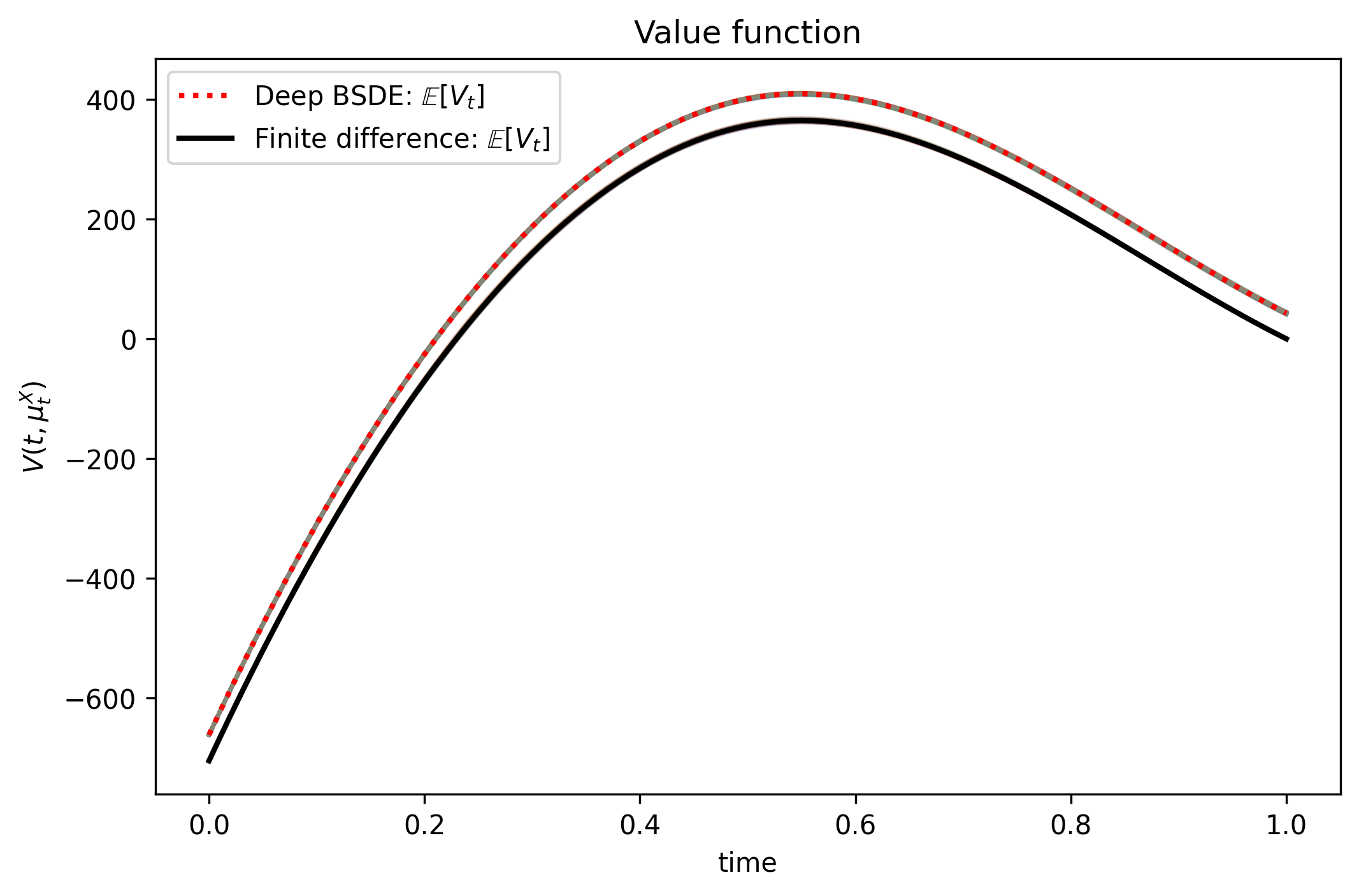}}
\end{adjustwidth}
\caption{$\mu_0^X = 1000 MWh$, $D = 1500 MWh$, $\sigma^0 = 1$, $T=1$, $r=1$, $\lambda_d = 0.005$, $S = 50\$\slash MWh$, $L_V\approx2$ after 3000 iterations.}
\label{Exam22r0.005}
\end{figure}
\begin{figure}[H]
\begin{adjustwidth}{-3cm}{-1cm}
\centering
\subfloat[$\mu^X$]{\includegraphics[width=0.4\textwidth]{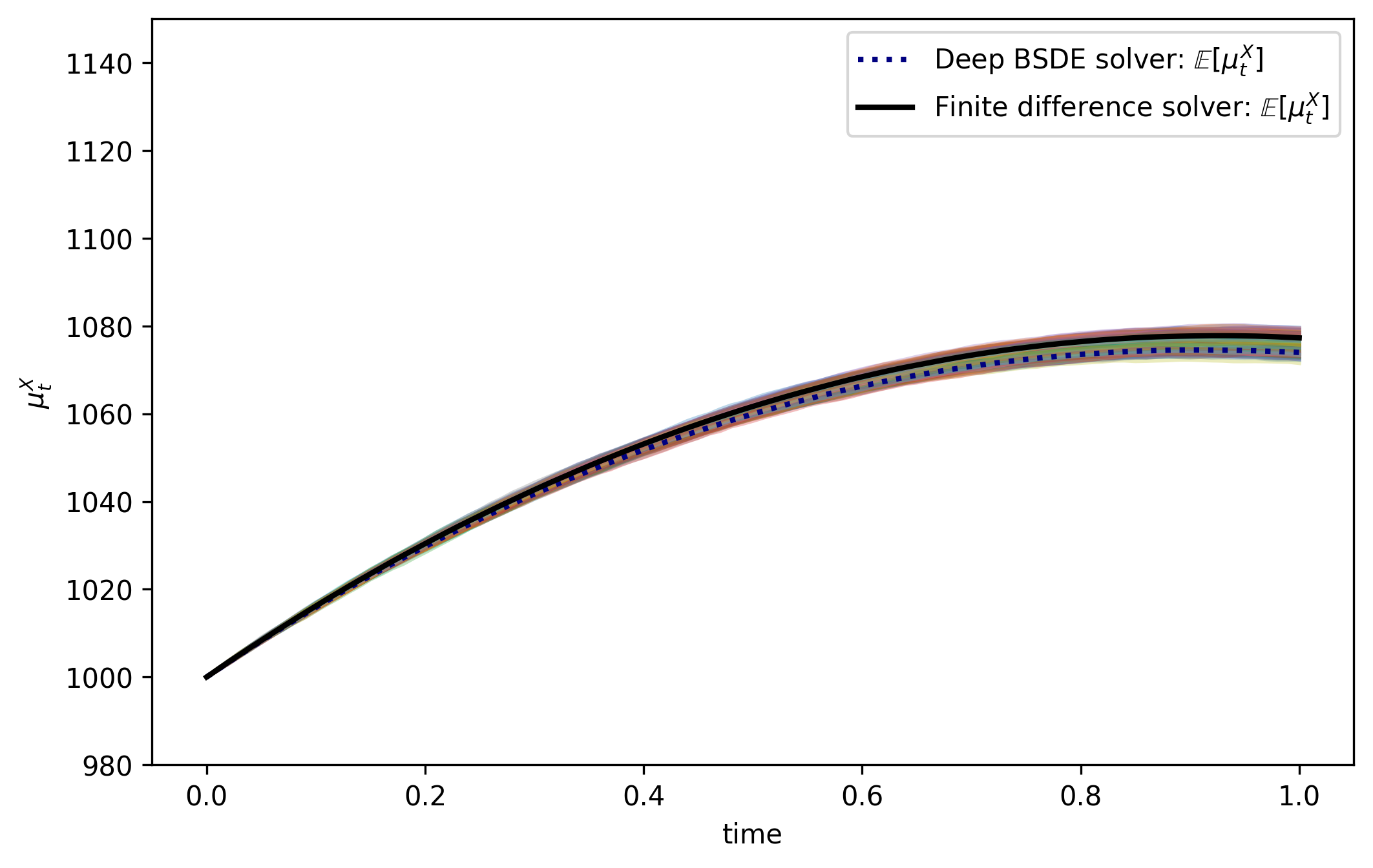}}
\subfloat[$\mu^Y$ and $\hat{\alpha}$]{\includegraphics[width=0.4\textwidth]{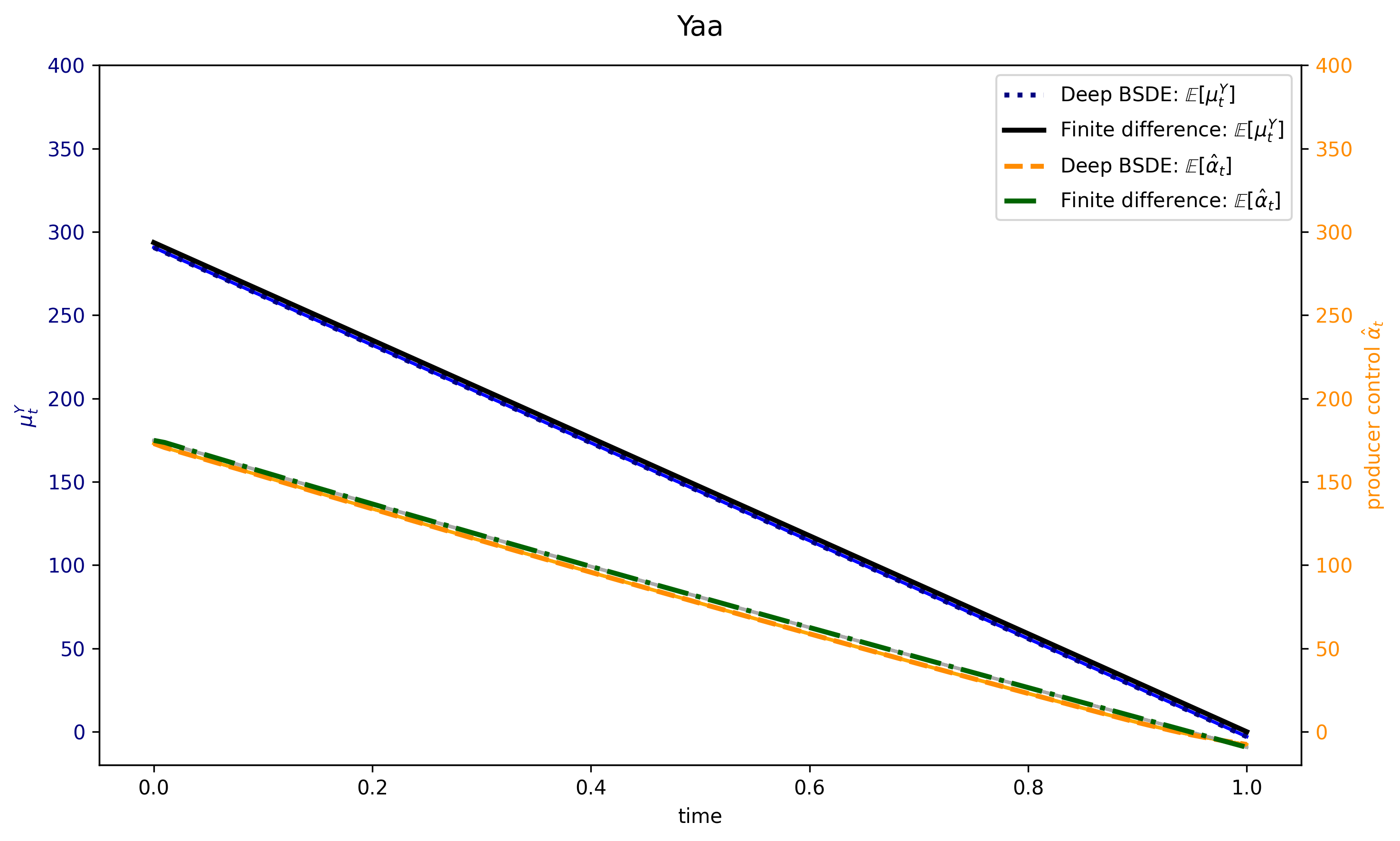}}\\
\subfloat[installation subsidy $\hat{v}$]{\includegraphics[width=0.4\textwidth]{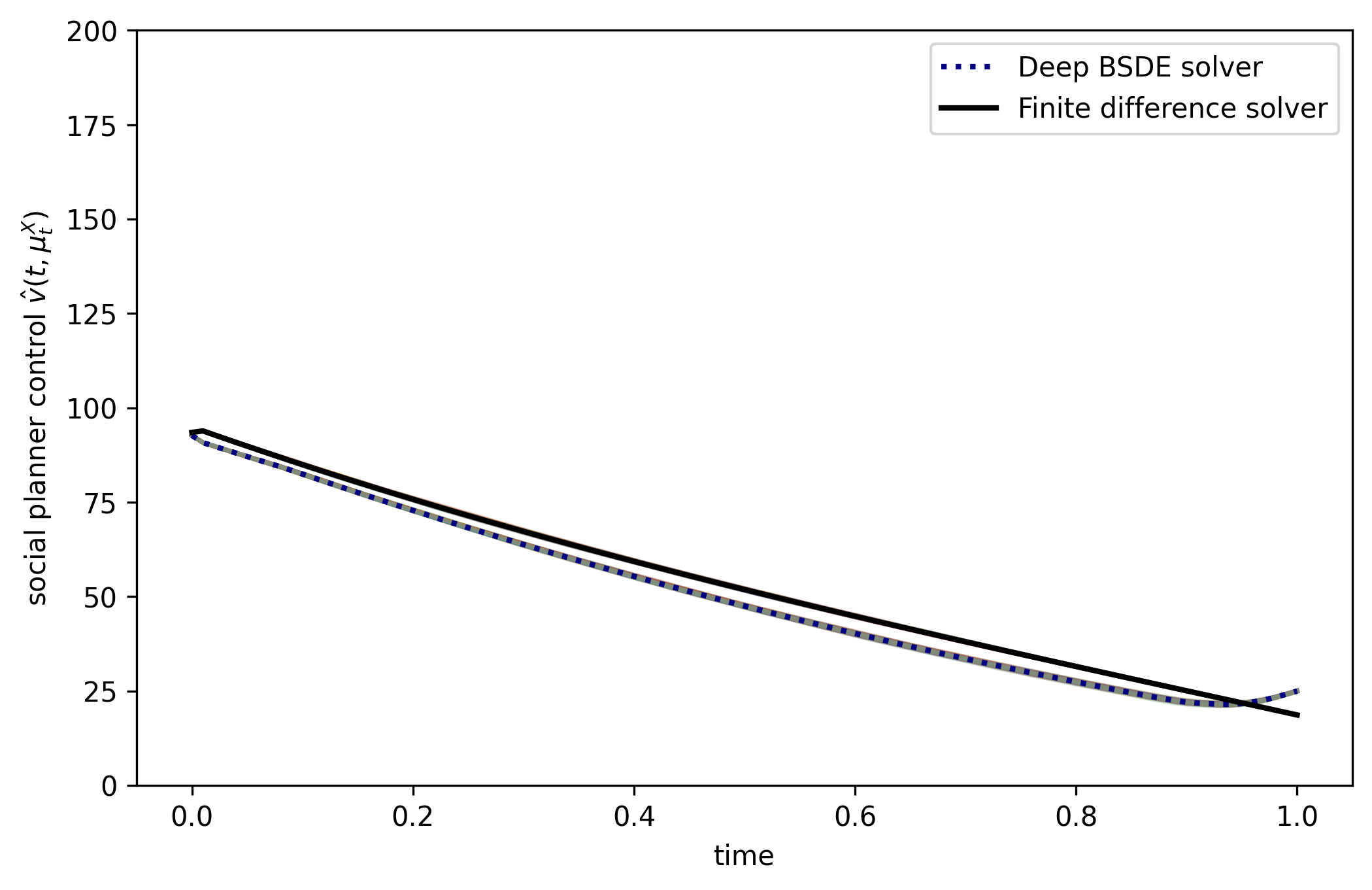}}
\subfloat[Value function $V$]{\includegraphics[width=0.4\textwidth]{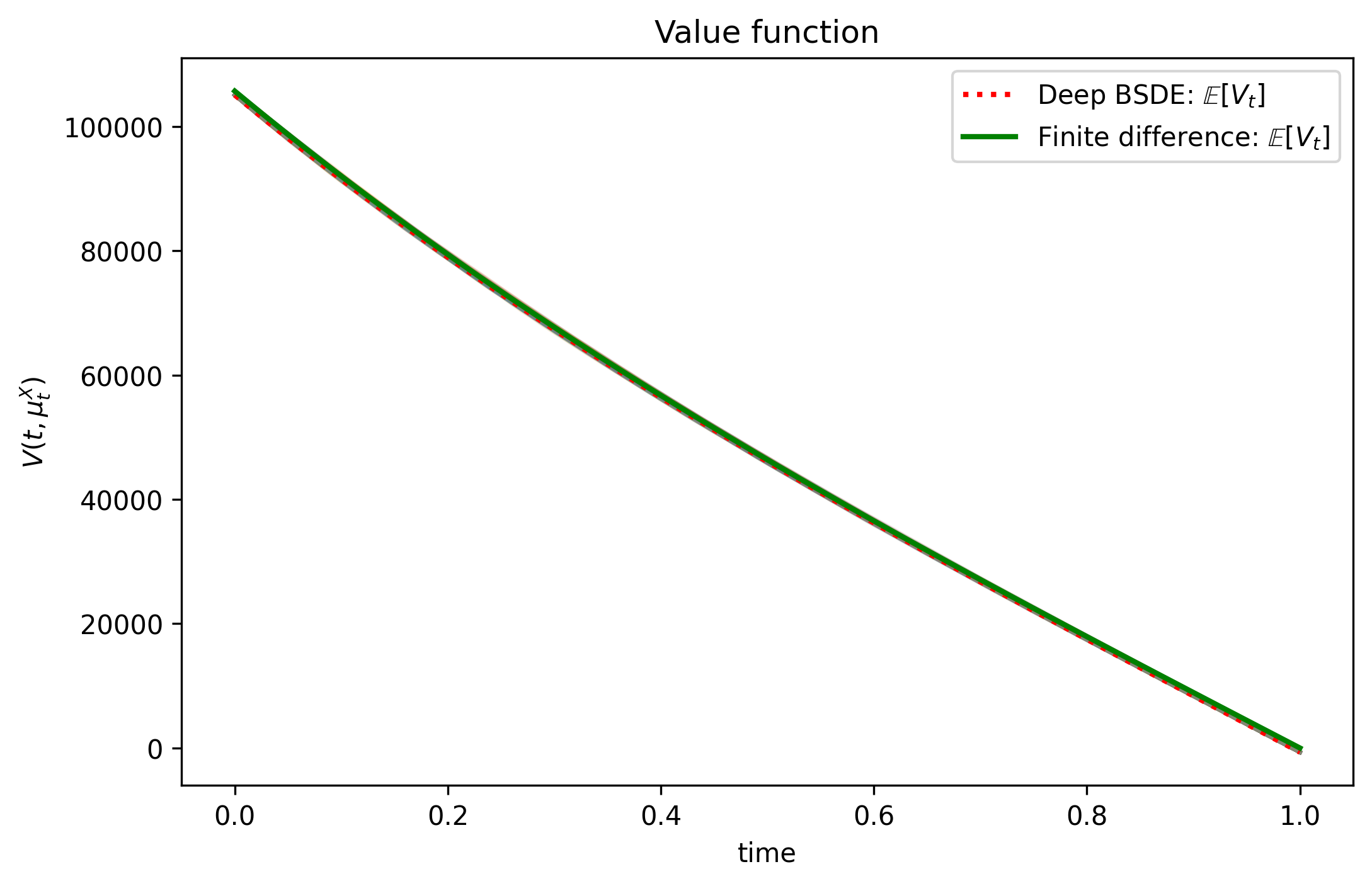}}
\end{adjustwidth}
\caption{$\mu_0^X = 1000 MWh$, $D = 1500 MWh$, $\sigma^0 = 1$, $T=1$, $r=1$, $\lambda_d = 0.5$, $S=500\$\slash MWh$, $10^{-4}L_V\approx 0.0005$ after 10000 iterations.}
\label{Exam22r0.5}
\end{figure}
\end{exam}

\begin{exam}[Excess supply at $t = 0$] \label{ES01}
In this example, we inherit the parameters from Example \ref{SPTExam}. Figures \ref{Exam01}, \ref{Exam03} and \ref{Exam32} tell us that the market fails to adjust the excess supply and to address this problem, government intervention is necessary. In Figure \ref{Exam32}, it is clear that given the low costs, the producer will increase the generation capacity at first and after excess supply drives the price to go further down, the producer will finally decrease the supply; therefore, as in Figure \ref{Exam34rr} (c), at the beginning of the taxation program, more tax should be imposed to discourage development of generation capacity. As the excess-supply condition subsides, the social planner may subsequently reduce the tax rate.
\begin{figure}[H]
\centering
\subfloat[$\mu^X$]{\includegraphics[width=0.4\textwidth]{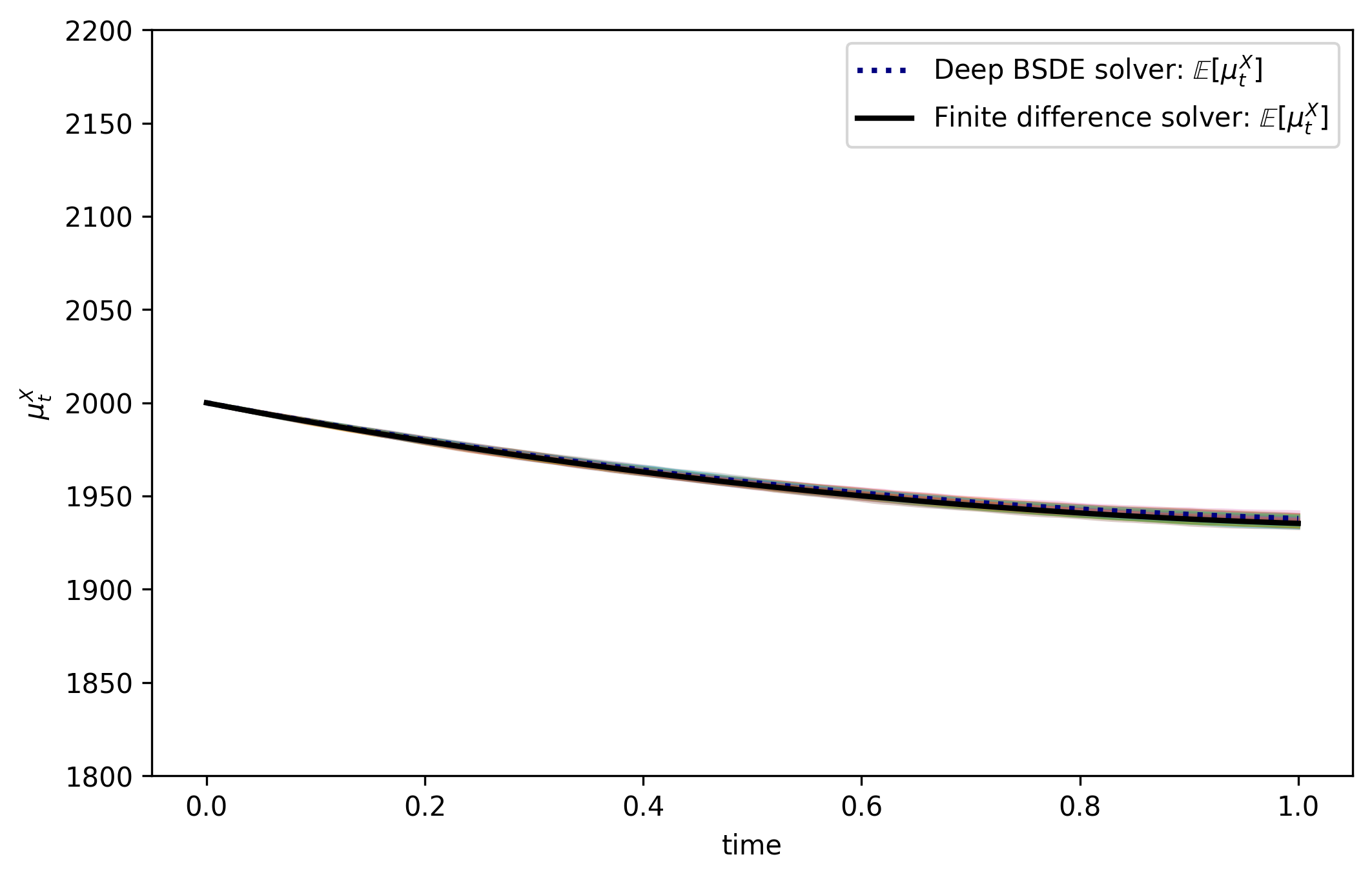}}
\subfloat[$\mu^Y$ and $\hat{\alpha}$]{\includegraphics[width=0.4\textwidth]{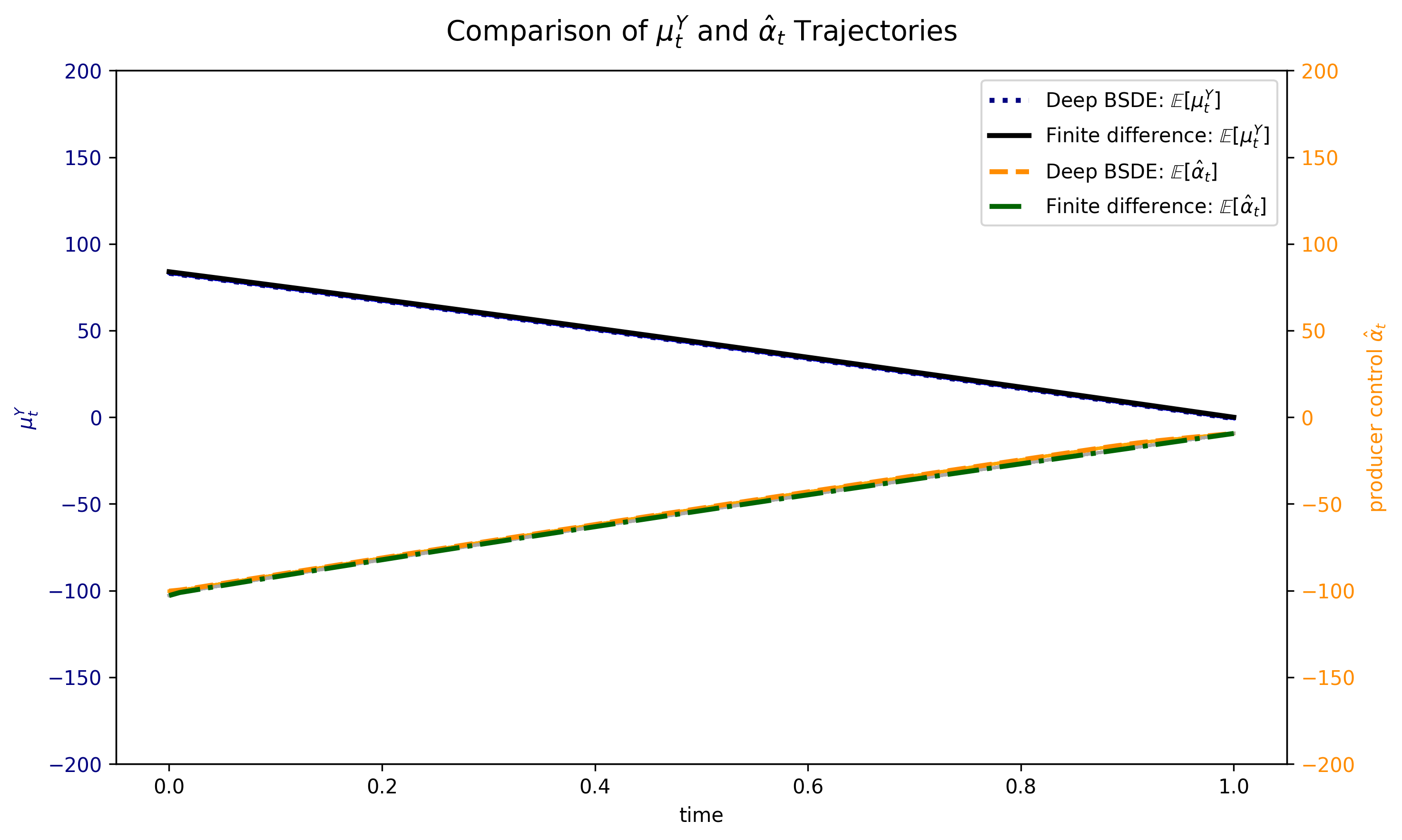}}\\
\subfloat[installation subsidy $\hat{v}$]{\includegraphics[width=0.4\textwidth]{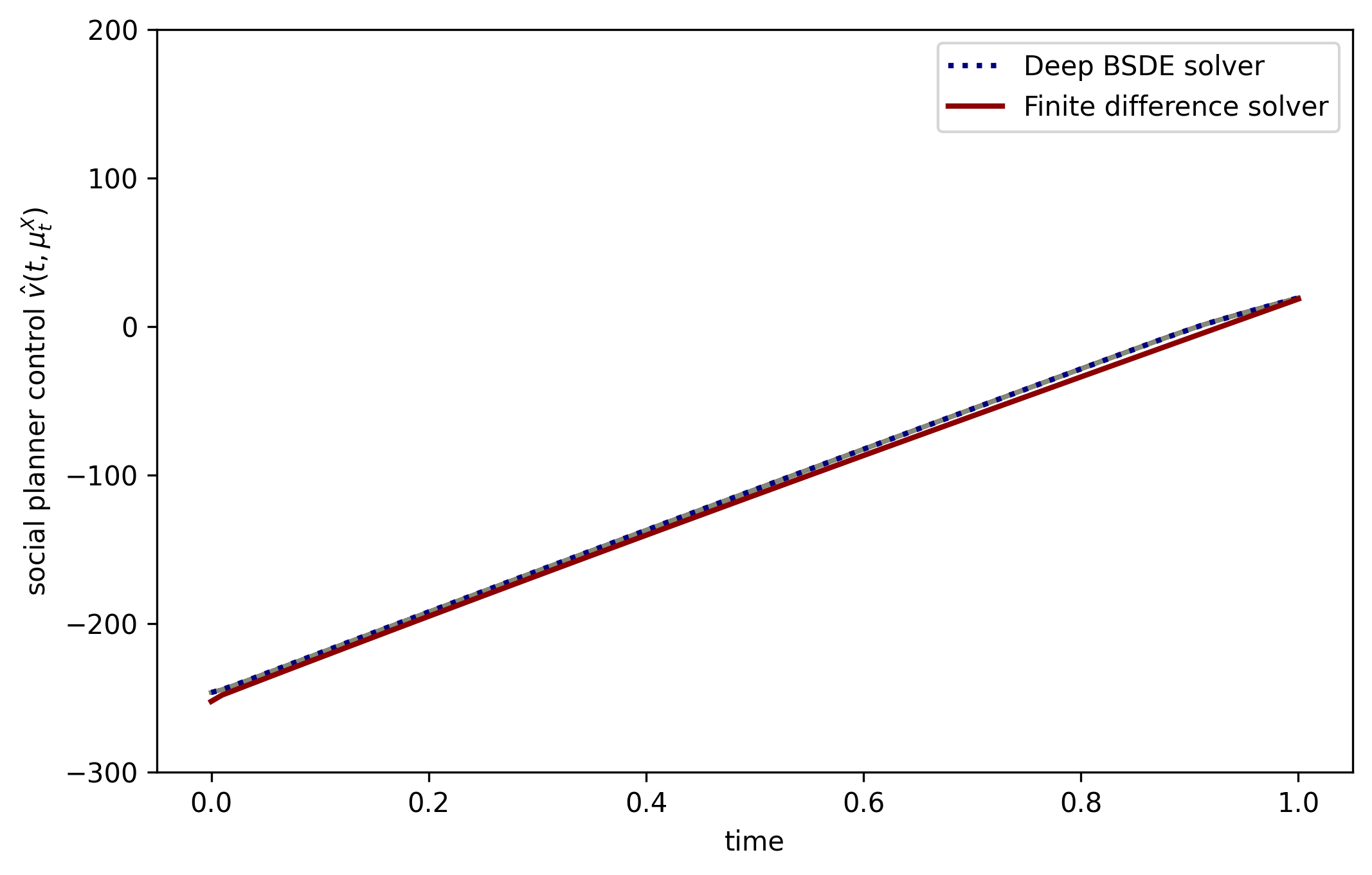}}
\subfloat[Value function V]{\includegraphics[width=0.4\textwidth]{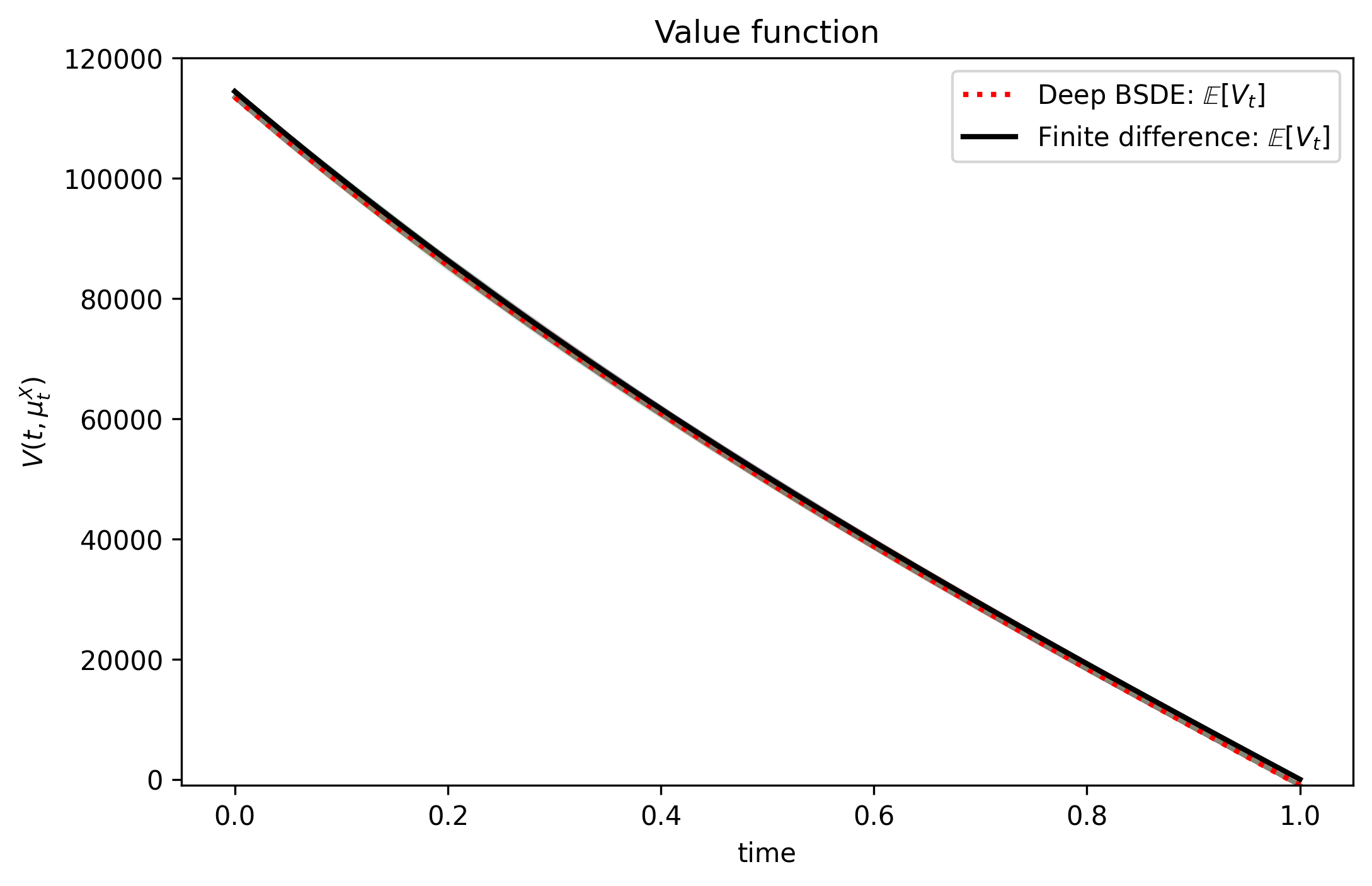}}
\caption{$\mu_0^X = 2000 MWh$, $D = 1500 MWh$, $\sigma^0 = 1$, $T=1$, $r=1$,$\lambda_d = 0.5$, $S = 500\$\slash MWh$. $L_V\approx 3$ after 5000 iterations.}
\label{Exam34rr}
\end{figure}
\end{exam}

\begin{exam}[Excess demand at $t=0$] \label{ED02}
In this example, we inherit the parameters from Example \ref{E2C}. In Figures \ref{Exam12}, \ref{Exam128} and \ref{Exam001}, despite more capacity being installed, the energy generation capacity still falls short compared with the demand. With a sufficiently high installation subsidy, it can effectively encourage the new installation of generation capacity. 
\begin{figure}[H]
\centering
\subfloat[$\mu^X$]{\includegraphics[width=0.4\textwidth]{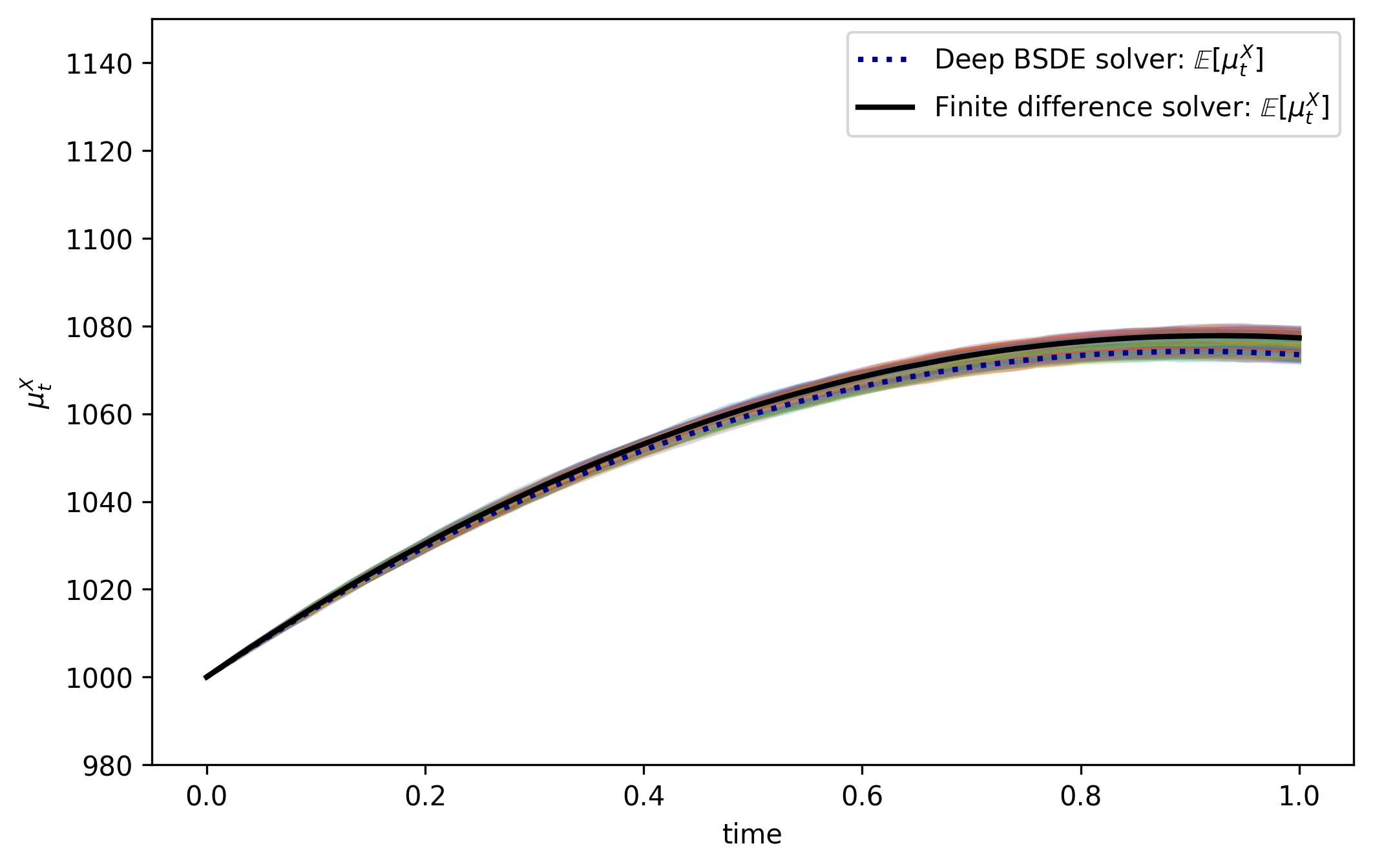}}
\subfloat[$\mu^Y$ and $\hat{\alpha}$]{\includegraphics[width=0.4\textwidth]{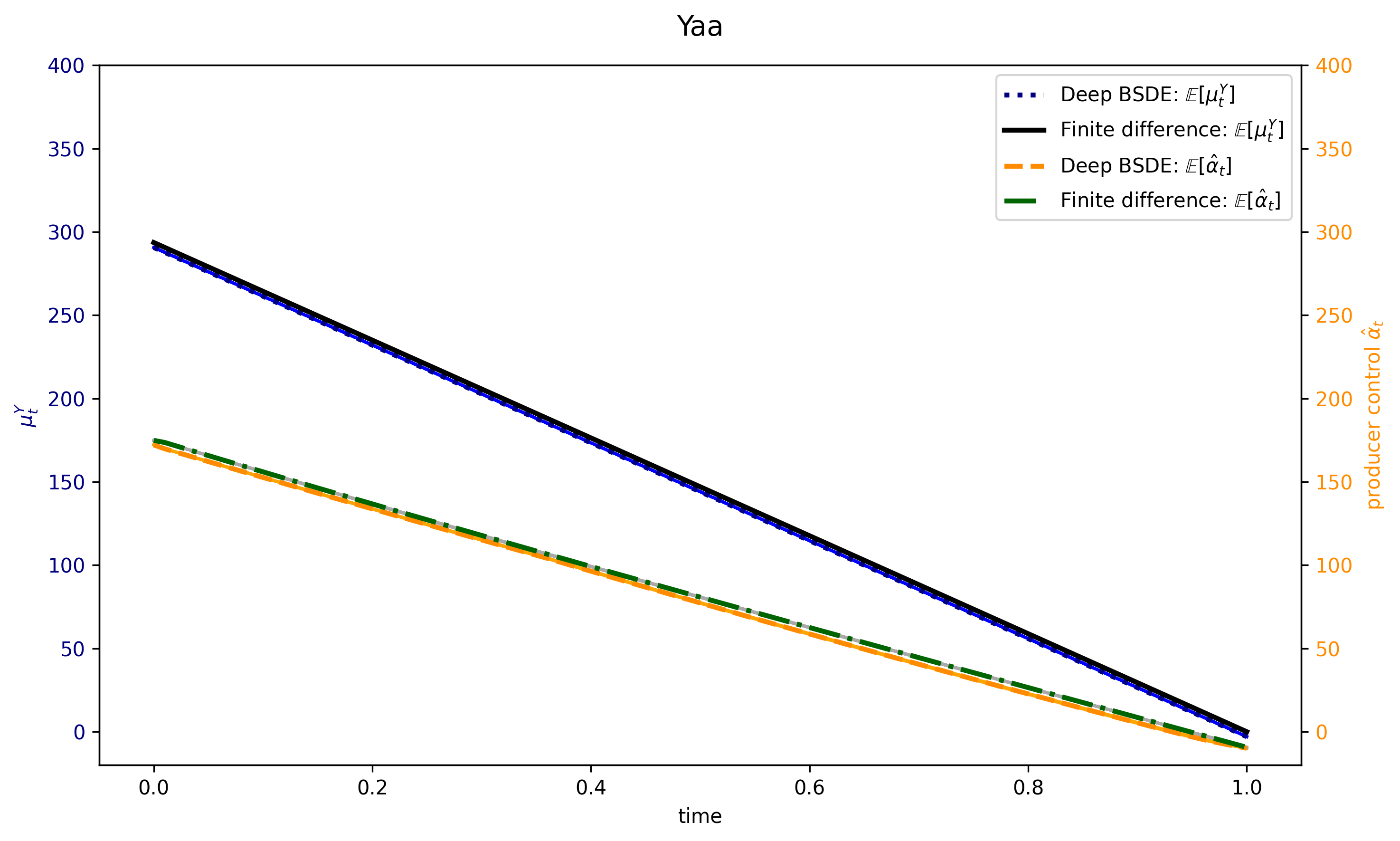}}\\
\subfloat[installation subsidy $\hat{v}$]{\includegraphics[width=0.4\textwidth]{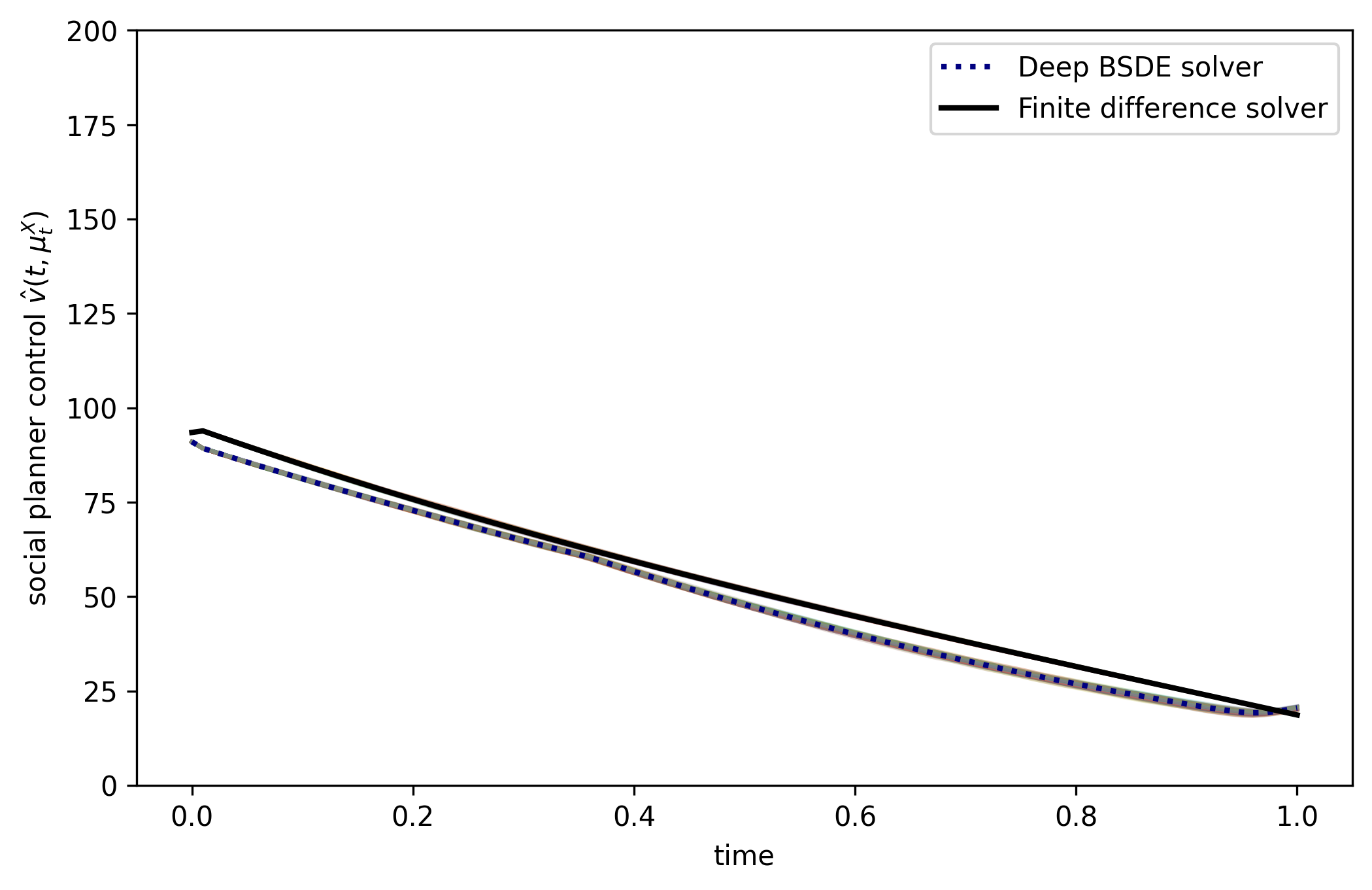}}
\subfloat[Value function V]{\includegraphics[width=0.4\textwidth]{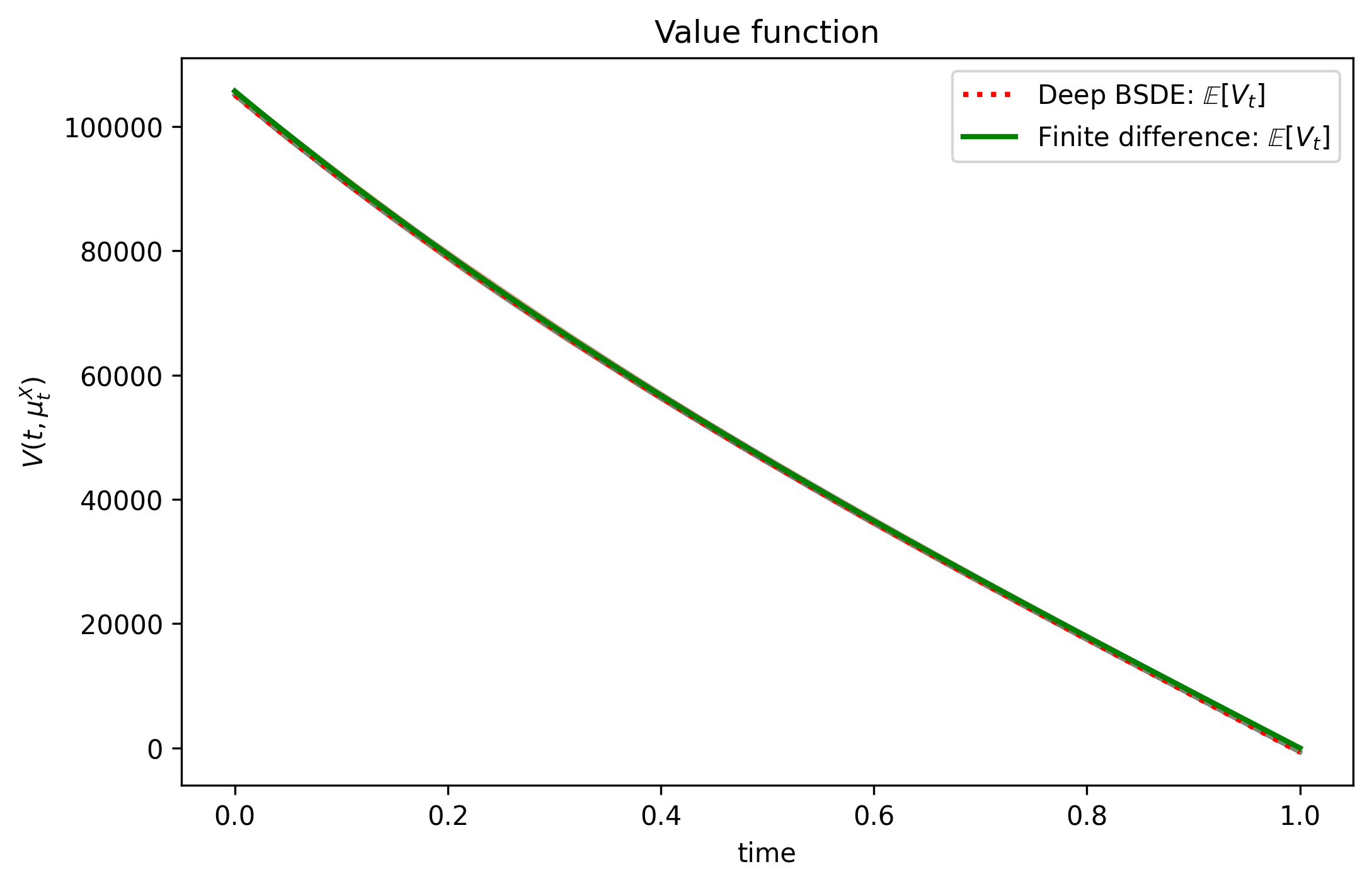}}
\caption{$\mu_0^X = 1000 MWh$, $D = 1500 MWh$, $\sigma^0 = 1$, $T=1$, $r=1$, $\lambda_d = 0.5$, $S = 500\$\slash MWh$, $L_V \approx 4$ after $5000$ iterations.}
\label{Exam34r}
\end{figure}
\end{exam}

\begin{exam}[Excess supply at $t = 0$.] \label{ES02}
In this example, we inherit the parameters from Example \ref{E2C}. Under the market mechanism proposed in Example \ref{E2C}, from Figures \ref{Exam11}, \ref{Exam13} and \ref{Exam1202}, it is clear that the representative producer continues to install more capacities even when the demand had long been satisfied. And in Figure \ref{Exam13}, the installation rate $\alpha$ is decreasing but remains positive, this is why in Figure \ref{Exam32r} higher level of taxation is imposed at the beginning of the program and once the issue of excess supply diminishes, the social planner can correspondingly ease the tax burden.
\begin{figure}[H]
\centering
\subfloat[$\mu^X$]{\includegraphics[width=0.4\textwidth]{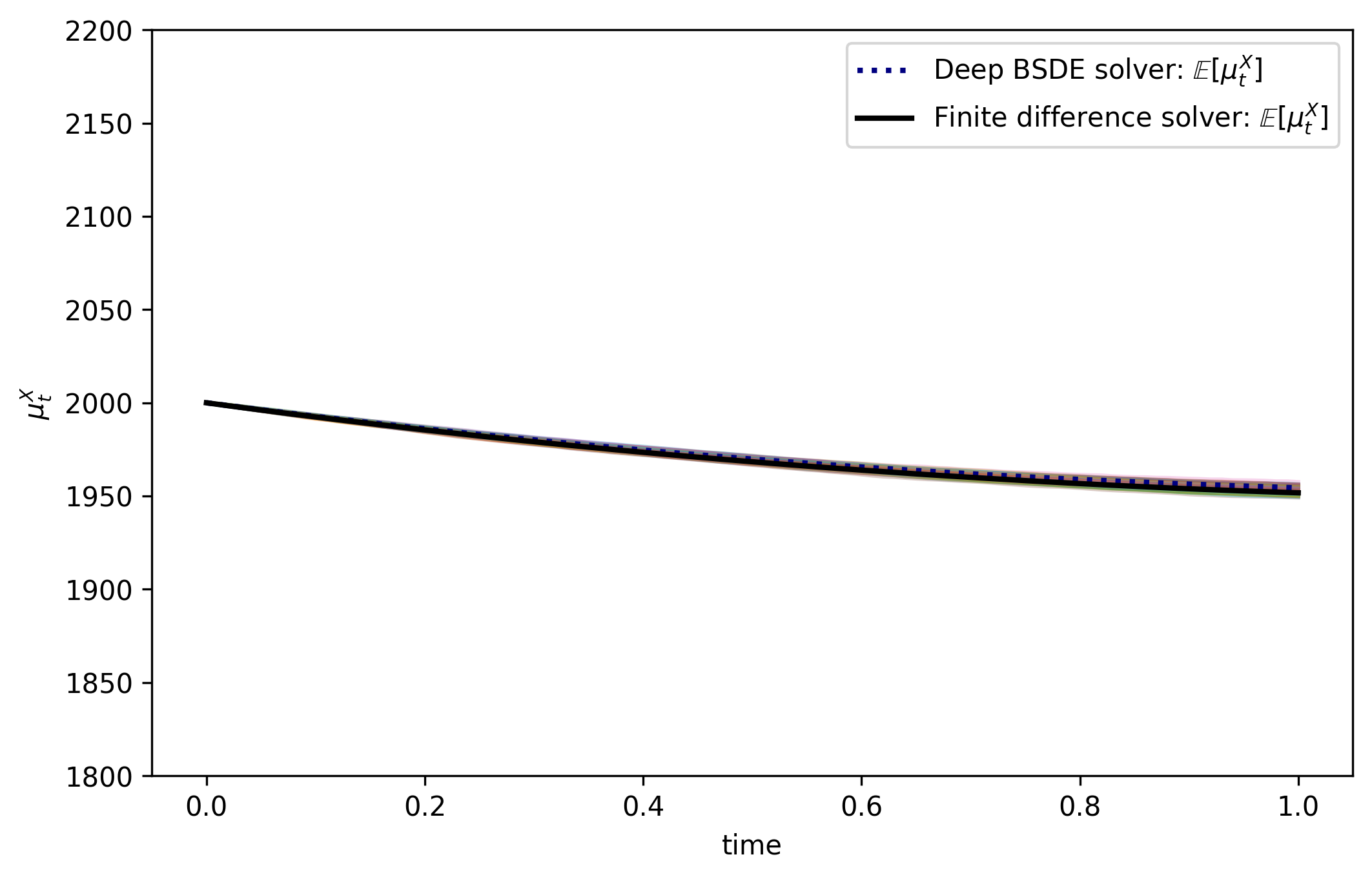}}
\subfloat[$\mu^Y$ and $\hat{\alpha}$]{\includegraphics[width=0.4\textwidth]{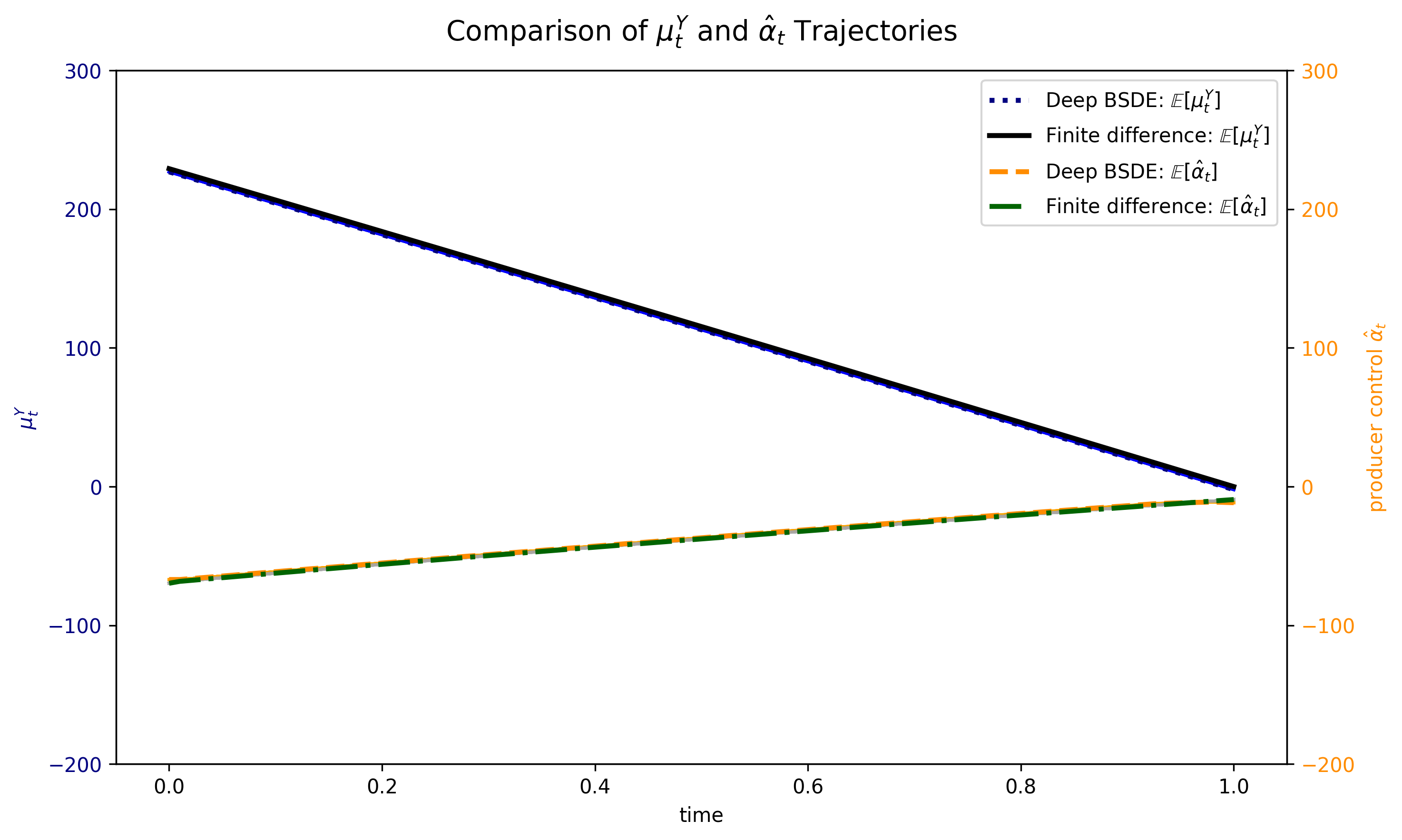}}\\
\subfloat[installation subsidy $\hat{v}$]{\includegraphics[width=0.4\textwidth]{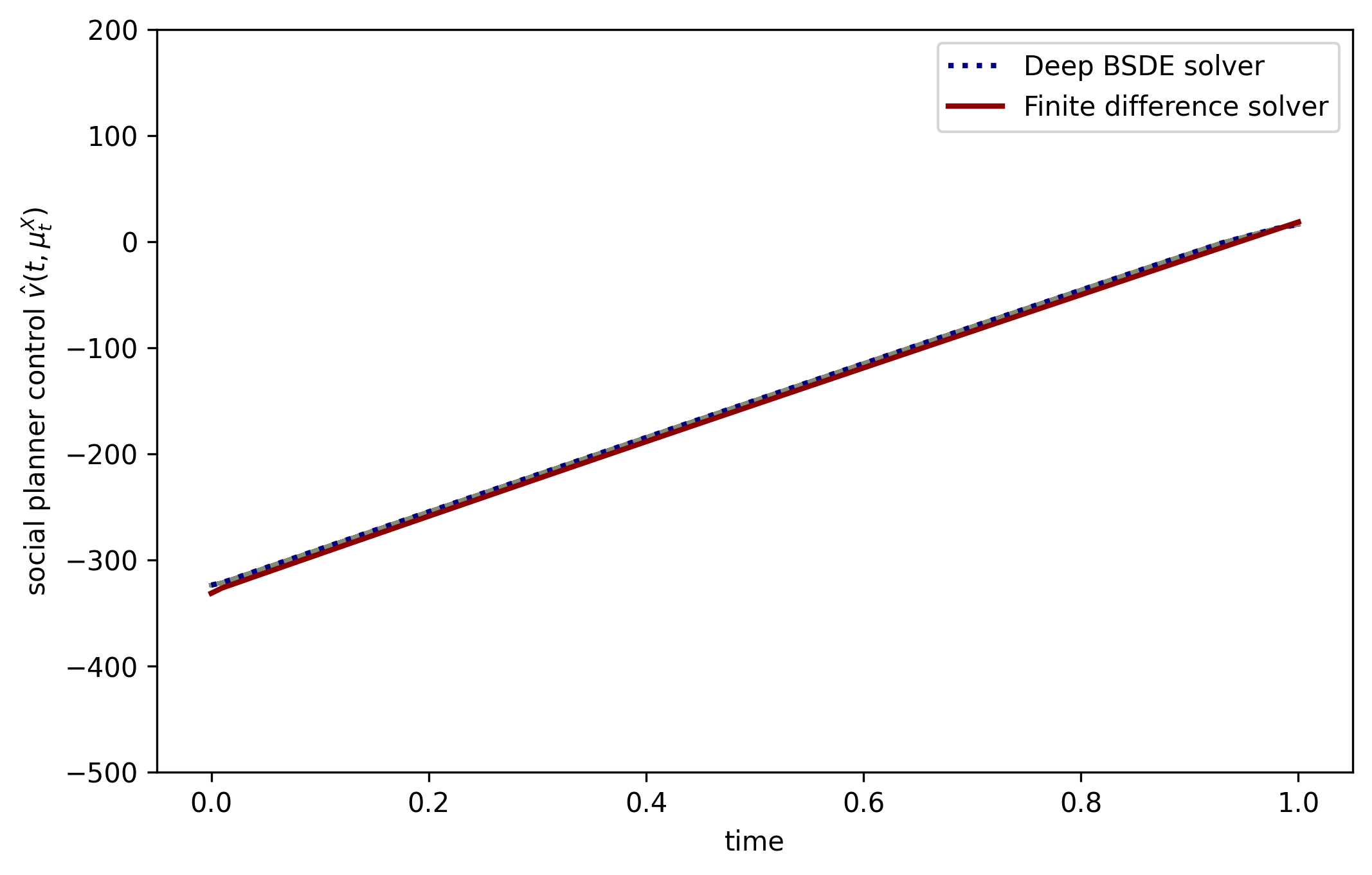}}
\subfloat[Value function V]{\includegraphics[width=0.4\textwidth]{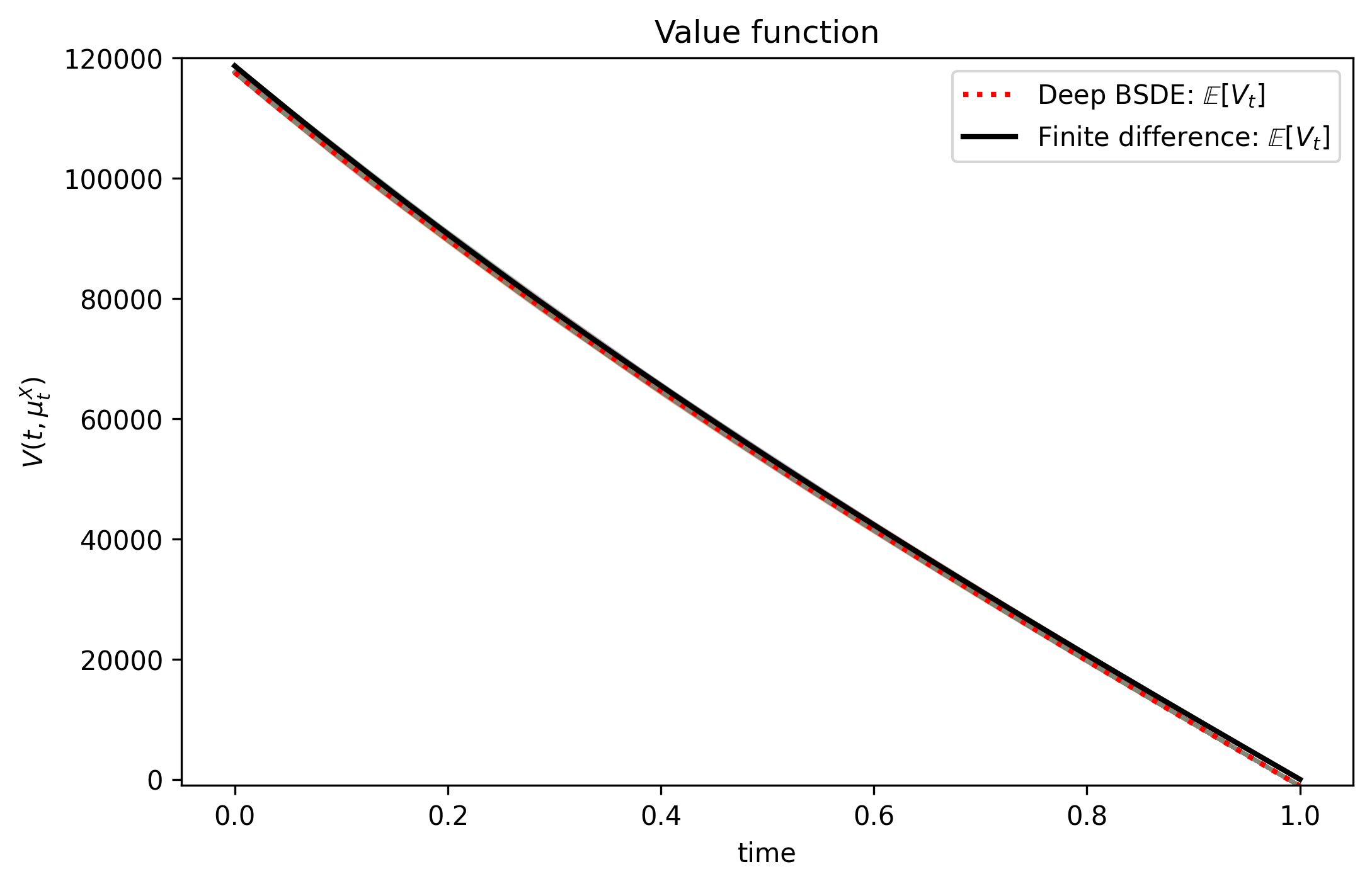}}
\caption{$\mu_0^X = 2000 MWh$, $D = 1500 MWh$, $\sigma^0 = 1$, $T=1$, $r=1$, $\lambda_d = 0.5$, $S = 500\$\slash MWh$, $L_V \approx 0.54315$ after $10000$ iterations.}
\label{Exam32r}
\end{figure}
\end{exam}

\begin{remark}
(i) Note that the values of target variables \(V\) and \(Z^V\) differ by several orders of magnitude. In addition, \(V\) has a large range of values, varying from approximately \(0\) to \(120000\) when \(\lambda_d=0.5\). For smaller $\lambda_d$, the algorithm converges much faster; while for $\lambda_d = 5$, the optimisation process tends to be trapped on a plateau. These scale differences can lead to large gradients and substantial parameter updates, making the training process highly unstable. To stabilise the training in Example \ref{ED01}, for \(\lambda_d=0.5\) we scale the loss function by multiplying it by \(10^{-4}\), i.e., we minimise \(10^{-4}L_V\). This rescaling enables the optimisation to escape the plateau observed when using the original loss scale. It should be emphasised, however, that multiplying the entire loss by a constant does not alter the relative contributions of \(V\) and \(Z^V\) to the objective. Rather, it changes the magnitude of the gradients and consequently the numerical dynamics of the optimisation. In particular, because \(Z^V\) is already much smaller in magnitude than \(V\), its contribution to the optimisation can remain difficult to resolve numerically, which calls for future investigation. The resulting trajectories are illustrated in Figure \ref{Exam22r0.5}. 

(ii) For Examples \ref{ED01}, \ref{ES01}, \ref{ED02} and \ref{ES02}, each subnetwork is fully connected and consists of five layers, with one input layer, three hidden layers and one output layer. We choose the rectifier function (ReLU) as our activation function in hidden layers and adopt batch normalisation after the first hidden layer. The subnetworks $\phi_{\theta^2}, z_{\theta^4}$ have 64 neurons in each hidden layer, whereas the subnetworks $V_{\theta^1}$ and $Z_{\theta^3}$ have 100 neurons in each hidden layer. The output layers have 1 dimension and do not use an activation function. The subnetworks are trained without pre-training. The network parameters are initialised using the default Keras initialisers except for $Z_{\theta^3}$ subnetwork, which is initialised through a normal distribution $N(0, 0.1)$  and setting the bias to zero. We implemented the proposed algorithm in TensorFlow using the Adam optimiser. The initial learning rates are 1e-4 for subnetworks $V_{\theta^1}, \phi_{\theta^2}, z_{\theta^4}$ and 1e-2 for $Z_{\theta^3}$. The learning rate of  $Z_{\theta^3}$ is subsequently reduced to 1e-5 in the training procedure. We use a training batch size of 2000 with 10000 samples used for validation.

(iii) Although finite-difference methods can be memory-intensive, training the neural networks in the deep BSDE approach can require considerably more computational time. Nevertheless, the two approaches have fundamentally different applicability. Finite-difference methods become impractical as the spatial dimension increases because of the curse of dimensionality, and their application also requires sufficient regularity of the value function for the derivatives appearing in the PDE to be well defined in the classical sense. In contrast, deep BSDE methods are designed to handle substantially higher-dimensional problems without constructing a spatial grid. Under appropriate conditions, they can approximate solutions in the viscosity sense even when a classical solution is not sufficiently smooth. Thus, although the deep BSDE approach may require longer training times, it provides a potentially scalable alternative to finite-difference methods for high-dimensional problems.
\end{remark}

\bmhead{Acknowledgements}
The authors acknowledge financial support under the National Recovery and Resilience Plan (NRRP), funded by the European Union, NextGenerationEU, Project Title: Probabilistic Methods for Energy Transition (n. P20224TM7Z). The authors also thank Anthony R\'{e}veillac for valuable discussions.

\bibliography{sn-bibliography}
\end{document}